\documentclass[oneside, leqno]{article}

\usepackage{geometry}

\newgeometry{
    top=1.5in,
    bottom=1.5in,
    outer=1in,
    inner=1in,
}
\usepackage[utf8]{inputenc}

\title{Enriched bi(co)ends}
\author{Nicola Carissimi}

\usepackage[T1]{fontenc}
\usepackage[utf8]{inputenc}
\usepackage[english]{babel}
\usepackage{amsmath}
\usepackage{amssymb}
\usepackage{amsthm}
\usepackage{xcolor}
\usepackage[nowrite,sans]{frontespizio}
\usepackage[hidelinks]{hyperref}
\usepackage[all, cmtip]{xy}
\usepackage{float}
\usepackage{mathtools}
\usepackage{dsfont}
\usepackage{fixltx2e}
\usepackage{ wasysym }
\usepackage{quoting}
\usepackage{bbold}
\usepackage{proof}
\usepackage{afterpage}
\usepackage{fancyhdr}
\usepackage{braket}
\usepackage{ mathrsfs }
\usepackage{ marvosym }
\usepackage{prftree}
\usepackage{stmaryrd}
\usepackage{centernot}
\usepackage{enumitem}
\usepackage{yhmath}
\usepackage[export]{adjustbox}
\usepackage[inference]{semantic}
\usepackage[toc,page]{appendix}
\usepackage{anyfontsize}
\usepackage{verbatim}
\usepackage{quiver}
\usepackage{tikz}
\usepackage{tikz-cd}
\usepackage{comment}
\usepackage[backend=bibtex,style=alphabetic]{biblatex}

\usetikzlibrary{shapes.geometric, arrows}

\newcommand{\id}{\mathrm{id}}
\newcommand{\rmod}{R\mathrm{\text{-}mod}}
\newcommand{\cat}{\mathcal}
\newcommand{\op}{^\mathsf{op}}
\newcommand{\epn}{\overset{..}{\Rightarrow}}
\newcommand{\Arrowu}{\rotatebox[origin=c]{-90}{$\Leftarrow$}}
\newcommand{\Arrowd}{\rotatebox[origin=c]{90}{$\Leftarrow$}}
\newcommand{\Arrowdl}{\rotatebox[origin=c]{45}{$\Leftarrow$}}
\newcommand{\Arrowdr}{\rotatebox[origin=c]{-45}{$\Rightarrow$}}
\newcommand{\Arrowur}{\rotatebox[origin=c]{45}{$\Rightarrow$}}
\newcommand{\Arrowul}{\rotatebox[origin=c]{-45}{$\Leftarrow$}}
\newcommand{\myrightleftarrows}[1]{\mathrel{\substack{\xrightarrow{#1} \\[-.5ex] \xleftarrow{#1}}}}
\newcommand{\xmapsfrom}[2][]{\xleftarrow[#1]{#2}\mapsfromchar}
\newcommand{\rednote}[1]{\textbf{#1}}
\newcommand{\psnat}{\mathsf{PsNat}}
\newcommand{\psfun}{\mathsf{PsFun}}
\newcommand{\verteq}{\rotatebox{90}{$\,=$}}
\newcommand{\diageq}{\rotatebox{45}{$\,=$}}
\newcommand*\cocolon{%
        \nobreak
        \mskip6mu plus1mu
        \mathpunct{}%
        \nonscript
        \mkern-\thinmuskip
        {:}%
        \mskip2mu
        \relax
}

\newcommand{\rcolon}{\rotatebox[origin=c]{180}{$\colon$}}

\newcommand{\adorn}[2]{\begin{array}[b]{@{}c@{}}#2\\#1\end{array}}

\newcommand{\pslan}[2]{\mathsf{PsLan}_{#1}{#2}}

\newcommand{\mot}{\mathsf{Mot}}
\newcommand{\spn}{\mathsf{Span}}
\newcommand{\mack}{\mathsf{Mack}}
\newcommand{\mor}{\mathsf{Hom}}

\newcommand{\G}{\mathbb{G}}
\newcommand{\J}{\mathbb{J}}

\newcommand{\gpd}{\mathsf{gpd}}

\newcommand{\ADD}{\mathsf{Add}}

\newcommand{\mysize}{\fontsize{5}{6}\selectfont}

\makeatletter
\newcommand{\bul}{} 
\DeclareRobustCommand\bul{%
  \mathord{\mathpalette\bul@{0.5}}%
}
\newcommand{\bul@}[2]{%
  \vcenter{\hbox{\scalebox{#2}{$\m@th#1\bullet$}}}%
}
\makeatother

\numberwithin{equation}{section}

\newtheorem{thm}[equation]{Theorem}
\newtheorem{prop}[equation]{Proposition}

\theoremstyle{definition}
\newtheorem{defn}[equation]{Definition}
\newtheorem{ex}[equation]{Example}

\newtheorem{rmk}[equation]{Remark}
\newtheorem{notat}[equation]{Notation}
\newtheorem{construction}[equation]{Construction}

\theoremstyle{remark}

\bibliography{refs}

\begin{document}

\maketitle

\tikzstyle{fixed size node} = [
    circle,
    draw,
    minimum size=.5cm, 
    inner sep=0, 
    align=center, 
    font=\normalsize 
]

\tikzstyle{none} = [
    minimum size=.5cm, 
    inner sep=0, 
    align=center, 
    font=\footnotesize 
]

\tikzstyle{squared} = [
    rectangle,
    draw,
    minimum size=.5cm, 
    inner sep=0, 
    align=center, 
    font=\normalsize 
]

\begin{abstract}
In this paper we introduce the theory of ends and coends in the context of enriched bicategories. This will be an enriched version of the theory introduced in \cite{Corner}, and a bicategorical version of the classical theory of enriched (co)ends, which can be found in \cite{Kelly2005} or in the more recent \cite{Fosco}. One of the main obstacles to the construction of such a theory is the amount of structure involved at this stage of categorification. A great help will be furnished by strictification results (Section \ref{sec strict}), as well as the powerful tool of string diagrams (Section \ref{sec string}), essential for making calculations manageable by a human being.
\end{abstract}

\tableofcontents

\section{Introduction}

Bicategories enriched over a monoidal bicategory $\cat V$ were first introduced by \cite{garner2015enriched}, and the present paper can also be considered as a continuation of their work. In that paper, the theory of enriched bicategories is wisely developed avoiding recourse to any kind of symmetry on the base of the enrichment, such as a braiding, which is necessary for the elementary constructions that will be presented here. These include the opposite $\cat V$-bicategory (Section \ref{section opposite}) and the tensor product of $\cat V$-bicategories (Section \ref{section tensor product}). These constructions, together with the notion of closed monoidal bicategory and its self-enrichment (Section \ref{section closed monoidal}), allow us to introduce the theory of enriched bi(co)ends (Section \ref{section bicoends}), opening the way to the definition of a $\cat V$-bicategory of $\cat V$-pseudofunctors (\ref{section enriched pseudofunctor bicat}), among the other things.

Two main results presented in this article are then the (verification of the) definition of opposite and tensor product of $\cat V$-bicategories (Theorem \ref{thmCop} and Theorem \ref{thmBtensorC}), when $\cat V$ is braided. Also, it is proved that a right closed monoidal bicategory is self-enriched (Theorem \ref{closed monoidal is self enriched}), and under the hypothesis of being closed, the theory of enriched biends and bicoends is introduced. In order to introduce this theory we define the enriched bicategorical version of extra-natural transformations (Section \ref{section extrapsnat}, and construction of a class of examples \ref{subsect ex}). These (enriched) \emph{extra-pseudonatural transformations} form a category, and we prove that enriched bi(co)ends are representing objects for this category (Proposition \ref{(co)end representing}), providing a conceptual clarification of the axioms defining them. This clarification appears to be new even with respect to the plain, non-enriched context of \cite{Corner}.

\subsection*{Acknowledgments:}
This article contains some of the results of my PhD thesis, written at the University of Lille under the supervision and with the kind help of Ivo Dell'Ambrogio, and with the financial support of Labex CEMPI and the University of Lille. I am thankful to colleagues and friends for the fruitful exchanges.

\section{Strictification results}\label{sec strict}

In this section we recollect the available strictification results for monoidal and braided monoidal bicategories, specializing more general results on tricategories. The theory of strictification for monoidal bicategories arises from the fact that tricategories cannot be fully stratified, not even in their one-object case (monoidal bicategories), nor in their one-object and one morphism case (braided monoidal categories). The standard reference for this is \cite{coherencetricat}, establishing that
\[\text{Any tricategory is triequivalent to a Gray category.}\]
This theorem boils down to its one-object case:
\begin{thm}[Gordon--Power--Street]
Any monoidal bicategory is monoidally biequivalent to a Gray monoid.
\end{thm}

As one starts to consider a braided structure on such a monoidal bicategory, the theory naturally lands in a higher categorical level. Since even the definition given by Todd Trimble of a tetracategory is fairly wild, it is understandable that a fully developed theory of strictification has yet to appear. However, the one-object and one-morphism case of a tetracategory (\emph{e.g.} a braided monoidal bicategory) and its strictification have been in fact well studied, and references considered in this paper for that are the work of Crans \cite{Crans98} and Gurski \cite{gurski2011loopspacescoherencemonoidal}. The first provided a definition of what a \emph{semi-strict braided monoidal bicategory} is, and the second proved the following:

\begin{thm}[Gurski]\label{strictification gurski}
Any braided monoidal bicategory is braided monoidally biequivalent to a semi-strict braided monoidal bicategory.
\end{thm}

What a Gray monoid (also known as a semi-strict monoidal bicategory) and a semi-strict braided monoidal bicategory are, will be briefly recalled in the next two subsections. An important aspect which is highlighted in \cite{gurski2011loopspacescoherencemonoidal} is that a semi-strict braided monoidal bicategory isn't just a semi-strict monoidal bicategory with a braided structure, but it satisfies more axioms, enhancing the power of the strictification result.

\subsection{Gray monoids}\label{subsect gray}

The Gray tensor product $\otimes_G$ has been first introduced in \cite{gray1974formal}, and defines a symmetric monoidal structure on the category of 2-categories and 2-functors. A \emph{Gray monoid} is then a monoid for this monoidal operation. The resulting structure happens to be a monoidal bicategory (this is checked for example at \cite{gurskiphd}, Section 5, by means of \emph{cubical pseudofunctors}), which is not a monoidal 2-category, since we replaced the cartesian structure with the Gray tensor product. Nonetheless, it's tensor product is strictly associative and unital. The semi-strictness is due to the fact that for every two pairs of objects $(A,B)$, $(A',B')$ and 1-cells $f\colon A\to A'$ and $g\colon B\to B'$, the commutativity of the interchange law for a monoidal 2-category is replaced by 2-isomorphisms called \emph{interchangers}:

\noindent
\begin{minipage}{.05\textwidth}
\begin{equation}\label{interchanger}
\phantom{c}
\end{equation}
\end{minipage}
\begin{minipage}{.9\textwidth}
\begin{center}
\begin{tikzcd}
	{A\otimes B} & {A\otimes B'} \\
	{A'\otimes B} & {A'\otimes B',}
	\arrow["{A\otimes g}", from=1-1, to=1-2]
	\arrow["{f \otimes B}"', from=1-1, to=2-1]
	\arrow[phantom, "{\Arrowdl\Sigma_{f,g}}"{description}, draw=none, from=1-2, to=2-1]
	\arrow["{f \otimes B'}", from=1-2, to=2-2]
	\arrow["{A' \otimes g}"', from=2-1, to=2-2]
\end{tikzcd}
\end{center}
\end{minipage}

\vspace{1em}
\noindent
Axioms for the interchanger will be given in Section \ref{sec string}. References for these axioms are, for example, the more general case of those for a Gray category in \cite{coherencetricat}.
 
\subsection{Braided Gray monoids and semi-strict braided monoidal bicategories}

Braided monoidal bicategory where first introduced at \cite{Mccrudden2000}. The structure (see Definition \ref{def braided gray mon}) is given by the braiding morphism, together two 2-isomorphisms $R$ and $S$ indexed by triples of objects and replacing the commutative hexagons of a braided monoidal category, satisfying four axioms. In the author's PhD thesis \cite{mythesis} a string diagrammatic non-strict formalism of the axioms for a braiding is given. Here, we are going to directly state the strict version of the same axioms (in Section \ref{sec string braid}), in which the monoidal structure of the braided monoidal bicategory is assumed to be that of a Gray monoid.

\begin{rmk}
Strictification for a braided monoidal bicategory yields a more specific structure than just a braided structure on a Gray monoid. The strictification theorem for braided monoidal bicategories would otherwise be just an easy consequence of the strictification for monoidal bicategories via a lifting of the braided structure.
\end{rmk}

The following is going to define both the structure of a braided Gray monoid and that of a semi-strict monoidal bicategory. It will be highlighted the difference between the braiding axioms, which are just a strict form of the axioms for a braided monoidal bicategory, and the additional constraints that we can achieve via a braided monoidal strictification.

\begin{defn}\label{def braided gray mon}
Let $\sigma_G$ denote the symmetry for the symmetric monoidal category of 2-categories and 2-functors under the Gray tensor product. A \emph{braided structure on a Gray monoid} $\cat B$ consist of a pseudonatural equivalence of 2-functors
\begin{center}
\begin{tikzcd}
	{\mathcal B\otimes_G\mathcal B} && {\mathcal B\otimes_G\mathcal B} \\
	& {\mathcal B}
	\arrow["{\sigma_G}", from=1-1, to=1-3]
	\arrow[""{name=0, anchor=center, inner sep=0}, "\otimes"', from=1-1, to=2-2]
	\arrow["\otimes", from=1-3, to=2-2]
	\arrow[phantom, "{\Arrowdl\beta}"{description}, draw=none, from=1-3, to=0]
\end{tikzcd}
\end{center}
and two modifications, for every triple of objects $A,B,C$,
\begin{center}
\begin{tikzcd}[column sep = 1em]
	A\otimes B\otimes C && B\otimes C\otimes A & A\otimes B\otimes C && C\otimes A\otimes B \\
	& B\otimes A\otimes C &&& A\otimes C\otimes B
	\arrow[""{name=0, anchor=center, inner sep=0}, "\beta", from=1-1, to=1-3]
	\arrow["{\beta 1}"', from=1-1, to=2-2]
	\arrow["{1\beta}", from=2-2, to=1-3]
	\arrow[""{name=1, anchor=center, inner sep=0}, "\beta", from=1-4, to=1-6]
	\arrow["{1\beta}"', from=1-4, to=2-5]
	\arrow["{\beta 1}", from=2-5, to=1-6]
	\arrow[phantom, "{\Arrowd R}"{description}, draw=none, from=0, to=2-2]
	\arrow[phantom, "{\Arrowd S}"{description}, draw=none, from=1, to=2-5]
\end{tikzcd}
\end{center}
satisfying the four braiding axioms (\ref{BA1})-(\ref{BA4}) below in Section \ref{sec string braid}.

A braided Gray monoid is, moreover, a \emph{semi-strict braided monoidal bicategory} if it satisfied the following further axioms (see \cite{Crans98} Definition 2.2). For every pair of objects $A, B$ in $\cat B$, the following pair of squares commute:
\begin{equation}\label{ax1 semi-strict braided}
\begin{tikzcd}
	{A\otimes \mathbb 1} & {\mathbb 1\otimes A} & {A\otimes \mathbb 1} \\
	A & A & A
	\arrow["\beta", from=1-1, to=1-2]
	\arrow[equals, from=1-1, to=2-1]
	\arrow["\beta", from=1-2, to=1-3]
	\arrow[equals, from=1-2, to=2-2]
	\arrow[equals, from=1-3, to=2-3]
	\arrow[equals, from=2-1, to=2-2]
	\arrow[equals, from=2-2, to=2-3]
\end{tikzcd}
\tag{s1}
\end{equation}
and the following identities hold true:
\begin{equation}\label{ax2 semi-strict braided}
\begin{tikzcd}[column sep=1em]
	AB & {AB\mathbb1} && {B\mathbb 1A} & BA \\
	& BA & {BA\mathbb 1} & BA
	\arrow[Rightarrow, no head, from=1-1, to=1-2]
	\arrow["\beta"', from=1-1, to=2-2]
	\arrow[""{name=0, anchor=center, inner sep=0}, "\beta", from=1-2, to=1-4]
	\arrow["{\beta 1}"', from=1-2, to=2-3]
	\arrow[Rightarrow, no head, from=1-4, to=1-5]
	\arrow[Rightarrow, no head, from=2-2, to=2-3]
	\arrow["{1\beta}"', from=2-3, to=1-4]
	\arrow[Rightarrow, no head, from=2-3, to=2-4]
	\arrow[Rightarrow, no head, from=2-4, to=1-5]
	\arrow[phantom, "{\Arrowd R}"{description}, draw=none, from=0, to=2-3]
	\arrow["\beta", curve={height=-18pt}, from=1-1, to=1-5]
\end{tikzcd}
=\id_{\beta}=
\begin{tikzcd}[column sep=1em]
	AB & {A\mathbb1B} && {\mathbb 1BA} & BA \\
	& AB & {\mathbb 1AB} & AB
	\arrow[Rightarrow, no head, from=1-1, to=1-2]
	\arrow[Rightarrow, no head, from=1-1, to=2-2]
	\arrow[""{name=0, anchor=center, inner sep=0}, "\beta", from=1-2, to=1-4]
	\arrow["{\beta 1}"', from=1-2, to=2-3]
	\arrow[Rightarrow, no head, from=1-4, to=1-5]
	\arrow[Rightarrow, no head, from=2-2, to=2-3]
	\arrow["{1\beta}"', from=2-3, to=1-4]
	\arrow[Rightarrow, no head, from=2-3, to=2-4]
	\arrow["\beta"', from=2-4, to=1-5]
	\arrow[phantom, "{\Arrowd R}"{description}, draw=none, from=0, to=2-3]
	\arrow["\beta", curve={height=-18pt}, from=1-1, to=1-5]
\end{tikzcd}
\tag{s2}
\end{equation}

\begin{equation}\label{ax3 semi-strict braided}
\begin{tikzcd}[column sep=1em]
	AB & {\mathbb 1AB} && {B\mathbb 1A} & BA \\
	& BA & {\mathbb 1BA} & BA
	\arrow[equals, from=1-1, to=1-2]
	\arrow["\beta"', from=1-1, to=2-2]
	\arrow[""{name=0, anchor=center, inner sep=0}, "\beta", from=1-2, to=1-4]
	\arrow["{1\beta}"', from=1-2, to=2-3]
	\arrow[equals, from=1-4, to=1-5]
	\arrow[equals, from=2-2, to=2-3]
	\arrow["\beta1"', from=2-3, to=1-4]
	\arrow[equals, from=2-3, to=2-4]
	\arrow[equals, from=2-4, to=1-5]
	\arrow[phantom, "{\Arrowd S}"{description}, draw=none, from=0, to=2-3]
	\arrow["\beta", curve={height=-18pt}, from=1-1, to=1-5]
\end{tikzcd}
=\id_{\beta}=
\begin{tikzcd}[column sep=1em]
	AB & {A\mathbb 1B} && {BA\mathbb 1} & BA \\
	& AB & {AB\mathbb 1} & AB
	\arrow[equals, from=1-1, to=1-2]
	\arrow["\beta", curve={height=-18pt}, from=1-1, to=1-5]
	\arrow[equals, from=1-1, to=2-2]
	\arrow[""{name=0, anchor=center, inner sep=0}, "\beta", from=1-2, to=1-4]
	\arrow["{1\beta}"', from=1-2, to=2-3]
	\arrow[equals, from=1-4, to=1-5]
	\arrow[equals, from=2-2, to=2-3]
	\arrow["\beta1"', from=2-3, to=1-4]
	\arrow[equals, from=2-3, to=2-4]
	\arrow["\beta"', from=2-4, to=1-5]
	\arrow["{\Arrowd S}"{description}, draw=none, from=0, to=2-3]
\end{tikzcd}
\tag{s3}
\end{equation}

\begin{equation}\label{ax4 semi-strict braided}
\begin{tikzcd}[column sep=1em]
	AB & \mathbb1AB && {AB\mathbb 1} & AB \\
	&& {A\mathbb 1B} \\
	&& AB
	\arrow[equals, from=1-1, to=1-2]
	\arrow[curve={height=-24pt}, equals, from=1-1, to=1-5]
	\arrow[equals, from=1-1, to=3-3]
	\arrow[""{name=0, anchor=center, inner sep=0}, "\beta", from=1-2, to=1-4]
	\arrow["{\beta 1}"', from=1-2, to=2-3]
	\arrow[equals, from=1-4, to=1-5]
	\arrow["{1\beta}"', from=2-3, to=1-4]
	\arrow[equals, from=3-3, to=1-5]
	\arrow[equals, from=3-3, to=2-3]
	\arrow[phantom,"{\Arrowd R}"{description}, draw=none, from=0, to=2-3]
\end{tikzcd}
=\id_\id=
\begin{tikzcd}[column sep=1em]
	AB & {AB \mathbb1} && {\mathbb 1AB} & AB \\
	&& {A\mathbb 1B} \\
	&& AB
	\arrow[equals, from=1-1, to=1-2]
	\arrow[curve={height=-24pt}, equals, from=1-1, to=1-5]
	\arrow[equals, from=1-1, to=3-3]
	\arrow[""{name=0, anchor=center, inner sep=0}, "\beta", from=1-2, to=1-4]
	\arrow["{1\beta }"', from=1-2, to=2-3]
	\arrow[equals, from=1-4, to=1-5]
	\arrow["\beta1"', from=2-3, to=1-4]
	\arrow[equals, from=3-3, to=1-5]
	\arrow[equals, from=3-3, to=2-3]
	\arrow[phantom, "{\Arrowd S}"{description}, draw=none, from=0, to=2-3]
\end{tikzcd}
\tag{s4}
\end{equation}

\end{defn}

\begin{rmk}
Observe that in axioms \eqref{ax2 semi-strict braided} and \eqref{ax3 semi-strict braided} every square commutes, either because of \eqref{ax1 semi-strict braided} (which also gives commutative squares in \eqref{ax4 semi-strict braided}), or because of strictness of unitality, like \begin{tikzcd}[column sep=1em, row sep=1em]
AB \arrow[equals, no head]{r}{} \arrow[swap]{d}{\beta} & AB\mathbb 1 \arrow{d}{\beta 1} \\%
BA \arrow[equals, no head]{r}{}& BA\mathbb 1
\end{tikzcd}; or again because of unitality for the pseudonatural transformation $\beta$, which evaluated at the identical monoidal unitor gives the commutativity of, for example, \begin{tikzcd}[column sep=1em, row sep=1em]
AB\mathbb 1 \arrow[equals, no head]{r}{} \arrow[swap]{d}{\beta} & AB \arrow{d}{\beta} \\%
B\mathbb 1 A \arrow[equals, no head]{r}{}& BA
\end{tikzcd}.

\end{rmk}

\section{String diagrams}\label{sec string}

In this section we fix notations for string diagrams and we point out the main rules that will be subsequently used. For what concerns the bicategorical structure, the orientation of 1-cell will be from the bottom to the top, while that of 2-cells is from left to right. For example, a 2-cell $\alpha\colon f\circ g\circ h\Rightarrow k\circ \ell$ is given by
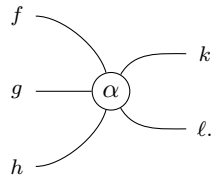
\begin{center}
\begin{tikzpicture}[scale=.5]
		\node [style=fixed size node] (0) at (0, 0) {$\alpha$};
		\node [style=none] (1) at (-2, 2) {};
		\node [style=none] (2) at (-2, 0) {};
		\node [style=none] (3) at (-2, -2) {};
		\node [style=none] (4) at (2, -1) {};
		\node [style=none] (5) at (2, 1) {};
		\node [style=none] (6) at (-2.5, 2) {$f$};
		\node [style=none] (7) at (-2.5, 0) {$g$};
		\node [style=none] (8) at (-2.5, -2) {$h$};
		\node [style=none] (9) at (2.5, -1) {$\ell$.};
		\node [style=none] (10) at (2.5, 1) {$k$};
		
		\draw [in=60, out=-180] (5.center) to (0);
		\draw [in=180, out=-60] (0) to (4.center);
		\draw [in=360, out=105, looseness=0.75] (0) to (1.center);
		\draw [in=360, out=180] (0) to (2.center);
		\draw [in=0, out=-105, looseness=0.75] (0) to (3.center);
\end{tikzpicture}
\end{center}

Each region is labeled with an object which is the domain of the string above it and the codomain of the string below it. These objects will usually remain implicit.

\begin{rmk}
The usual rules making string diagrams a powerful language and which are used on a constant base in calculations are the Eckmann-Hilton argument for 2-cells $\alpha\colon f\Rightarrow f'$ and $\beta\colon g\Rightarrow g'$:
\begin{center}
\begin{tikzpicture}[scale=.5]
		\node [style=fixed size node] (0) at (-5.5, 1) {$\alpha$};
		\node [style=fixed size node] (1) at (-3.5, -1) {$\beta$};
		\node [style=none] (2) at (-7.5, 1) {};
		\node [style=none] (3) at (-7.5, -1) {};
		\node [style=none] (4) at (-1.5, 1) {};
		\node [style=none] (5) at (-1.5, -1) {};
		\node [style=none] (6) at (1.5, 1) {};
		\node [style=none] (7) at (1.5, -1) {};
		\node [style=none] (8) at (7.5, -1) {};
		\node [style=none] (9) at (7.5, 1) {};
		\node [style=fixed size node] (10) at (5.5, 1) {$\alpha$};
		\node [style=fixed size node] (11) at (3.5, -1) {$\beta$};
		\node [style=none] (16) at (-1, 1) {$f'$};
		\node [style=none] (17) at (-1, -1) {$g'$};
		\node [style=none] (20) at (-8, 1) {$f$};
		\node [style=none] (21) at (-8, -1) {$g$};
		\node [style=none] (22) at (0, 0) {$=$};
		\node [style=none] (23) at (1, 1) {$f$};
		\node [style=none] (24) at (1, -1) {$g$};
		\node [style=none] (25) at (8, 1) {$f'$};
		\node [style=none] (26) at (8, -1) {$g'$};
		
		\draw [in=180, out=0] (2.center) to (0);
		\draw [in=180, out=0] (0) to (4.center);
		\draw [in=180, out=0] (3.center) to (1);
		\draw [in=180, out=0] (1) to (5.center);
		\draw [in=180, out=0] (6.center) to (10);
		\draw [in=180, out=0] (10) to (9.center);
		\draw [in=180, out=0] (7.center) to (11);
		\draw [in=180, out=0] (11) to (8.center);
\end{tikzpicture}
\end{center}
and the triangular identities for internal adjunctions $f\dashv g$:
\begin{center}
\begin{tikzpicture}[scale=.5]
		\node [style=none] (0) at (0, 0) {$=$};
		\node [style=none] (1) at (2, 0) {};
		\node [style=none] (2) at (1.5, 0) {$g$};
		\node [style=none] (3) at (-2, 1.5) {};
		\node [style=none] (4) at (-1.5, 1.5) {$g$};
		\node [style=none] (5) at (6, 0) {};
		\node [style=none] (6) at (6.5, 0) {$g$};
		\node [style=none] (7) at (-4, 1.5) {};
		\node [style=none] (8) at (-4, 0) {};
		\node [style=none] (9) at (-4, -1.5) {};
		\node [style=none] (10) at (-6, -1.5) {};
		\node [style=none] (11) at (-6.5, -1.5) {$g$};
		\node [style=none] (12) at (-3.75, 0.5) {$f$};

		\draw [in=180, out=0] (1.center) to (5.center);
		\draw (3.center) to (7.center);
		\draw [in=180, out=-180, looseness=1.50] (7.center) to (8.center);
		\draw [in=0, out=0, looseness=1.50] (8.center) to (9.center);
		\draw (9.center) to (10.center);
\end{tikzpicture}
\hspace{2em}\raisebox{2.5em}{\text{and}}\hspace{2em}
\begin{tikzpicture}[scale=.5]
		\node [style=none] (0) at (0, 0) {$=$};
		\node [style=none] (1) at (2, 0) {};
		\node [style=none] (2) at (1.5, 0) {$f$};
		\node [style=none] (3) at (-2, -1.5) {};
		\node [style=none] (4) at (-1.5, -1.5) {$f$};
		\node [style=none] (5) at (6, 0) {};
		\node [style=none] (6) at (6.5, 0) {$f.$};
		\node [style=none] (7) at (-4, -1.5) {};
		\node [style=none] (8) at (-4, 0) {};
		\node [style=none] (9) at (-4, 1.5) {};
		\node [style=none] (10) at (-6, 1.5) {};
		\node [style=none] (11) at (-6.5, 1.5) {$f$};
		\node [style=none] (12) at (-4.25, 0.5) {$g$};
		\draw [in=-180, out=0] (1.center) to (5.center);
		\draw (3.center) to (7.center);
		\draw [in=-180, out=180, looseness=1.50] (7.center) to (8.center);
		\draw [in=0, out=0, looseness=1.50] (8.center) to (9.center);
		\draw (9.center) to (10.center);
\end{tikzpicture}
\end{center}
\end{rmk}

\begin{rmk}
In the case of the bicategory of pseudofunctors between two fixed bicategories, pseudonatural transformations and modifications, a source of rules for string calculus comes from the axioms for a pseudonatural transformation (naturality in particular) and from the modification axiom. Translated in diagrams, they say the following:
\begin{itemize}[noitemsep,wide=0pt, leftmargin=\dimexpr\labelwidth + 2\labelsep\relax]
\item Let $t\colon F\Rightarrow G$ a pseudonatural transformation of pseudofunctors. That means, for every $u\colon X\to Y$ there are invertible crossings (2-cells)
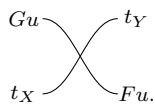
\begin{center}
\begin{tikzpicture}[scale=.5]
		\node [style=none] (0) at (-1, 1) {};
		\node [style=none] (1) at (-1, -1) {};
		\node [style=none] (2) at (1, 1) {};
		\node [style=none] (3) at (1, -1) {};
		\node [style=none] (4) at (1.5, 1) {$t_Y$};
		\node [style=none] (5) at (1.5, -1) {$Fu.$};
		\node [style=none] (6) at (-1.5, 1) {$Gu$};
		\node [style=none] (7) at (-1.5, -1) {$t_X$};
		
		\draw [in=0, out=180, looseness=0.75] (2.center) to (1.center);
		\draw [in=-180, out=0, looseness=0.75] (0.center) to (3.center);
\end{tikzpicture}
\end{center}
Then, the naturality axiom says that for every 2-cell $\alpha\colon u\to u'$ in the domain bicategory we have

\noindent
\begin{minipage}{.05\textwidth}
\begin{equation}\label{naturality axiom}
\phantom{c}
\end{equation}
\end{minipage}
\begin{minipage}{.9\textwidth}
\begin{center}
\begin{tikzpicture}[scale=.5]
		\node [style=none] (0) at (-5.5, 1) {};
		\node [style=none] (1) at (-5.5, -1) {};
		\node [style=none] (2) at (-3.5, 1) {};
		\node [style=none] (3) at (-3.5, -1) {};
		\node [style=none] (4) at (-1, 1) {$t_Y$};
		\node [style=none] (5) at (-3.5, -1.5) {$Fu$};
		\node [style=none] (6) at (-6, 1) {$Gu$};
		\node [style=none] (7) at (-6, -1) {$t_X$};
		\node [style=fixed size node] (8) at (-2.5, -1) {$F\alpha$};
		\node [style=none] (9) at (-1.5, -1) {};
		\node [style=none] (10) at (-1.5, 1) {};
		\node [style=none] (11) at (-1, -1) {$Fu'$};
		\node [style=none] (12) at (5.5, -1) {};
		\node [style=none] (13) at (5.5, 1) {};
		\node [style=none] (14) at (3.5, -1) {};
		\node [style=none] (15) at (3.5, 1) {};
		\node [style=none] (17) at (3.5, 1.5) {$Gu'$};
		\node [style=none] (18) at (6, -1) {$Fu'.$};
		\node [style=none] (19) at (6, 1) {$t_Y$};
		\node [style=fixed size node] (20) at (2.5, 1) {$G\alpha$};
		\node [style=none] (21) at (1.5, 1) {};
		\node [style=none] (22) at (1.5, -1) {};
		\node [style=none] (24) at (0, 0) {$=$};
		\node [style=none] (25) at (1, 1) {$Gu$};
		\node [style=none] (26) at (1, -1) {$t_X$};
		
		\draw [in=0, out=180, looseness=0.75] (2.center) to (1.center);
		\draw [in=-180, out=0, looseness=0.75] (0.center) to (3.center);
		\draw (2.center) to (10.center);
		\draw (3.center) to (8);
		\draw (8) to (9.center);
		\draw [in=-180, out=0, looseness=0.75] (14.center) to (13.center);
		\draw [in=0, out=-180, looseness=0.75] (12.center) to (15.center);
		\draw (14.center) to (22.center);
		\draw (15.center) to (20);
		\draw [in=360, out=180] (20) to (21.center);
\end{tikzpicture}
\end{center}
\end{minipage}
\item If $M\colon t\Rrightarrow s$ is a modification, then the data of the family of 2-cells $M_X\colon t_X\Rightarrow s_X$ is subject to the equality

\noindent
\begin{minipage}{.1\textwidth}
\begin{equation}\label{modification axiom}
\phantom{c}
\end{equation}
\end{minipage}
\begin{minipage}{.9\textwidth}
\begin{center}
\begin{tikzpicture}[scale=.5]
		\node [style=none] (0) at (1.5, 1) {};
		\node [style=none] (1) at (1.5, -1) {};
		\node [style=none] (2) at (3.5, 1) {};
		\node [style=none] (3) at (3.5, -1) {};
		\node [style=none] (4) at (6, 1) {$s_Y$};
		\node [style=none] (5) at (3.5, 1.5) {$t_Y$};
		\node [style=none] (6) at (1, 1) {$Gu$};
		\node [style=none] (7) at (1, -1) {$t_X$};
		\node [style=none] (9) at (5.5, -1) {};
		\node [style=fixed size node] (10) at (4.5, 1) {$M_Y$};
		\node [style=none] (11) at (6, -1) {$Fu$};
		\node [style=none] (12) at (-1.5, -1) {};
		\node [style=none] (13) at (-1.5, 1) {};
		\node [style=none] (14) at (-3.5, -1) {};
		\node [style=none] (15) at (-3.5, 1) {};
		\node [style=none] (18) at (-1, -1) {$Fu$};
		\node [style=none] (19) at (-1, 1) {$s_Y$};
		\node [style=none] (21) at (-5.5, 1) {};
		\node [style=none] (22) at (-5.5, -1) {};
		\node [style=none] (24) at (0, 0) {$=$};
		\node [style=none] (25) at (-6, 1) {$Gu$};
		\node [style=none] (26) at (-6, -1) {$t_X$};
		\node [style=none] (27) at (5.5, 1) {};
		\node [style=fixed size node] (28) at (-4.5, -1) {$M_X$};
		\node [style=none] (29) at (-3.5, -1.5) {$s_X$};
		
		\draw [in=0, out=180, looseness=0.75] (2.center) to (1.center);
		\draw [in=-180, out=0, looseness=0.75] (0.center) to (3.center);
		\draw (2.center) to (10);
		\draw [in=-180, out=0, looseness=0.75] (14.center) to (13.center);
		\draw [in=0, out=-180, looseness=0.75] (12.center) to (15.center);
		\draw (3.center) to (9.center);
		\draw (27.center) to (10);
		\draw (22.center) to (28);
		\draw (28) to (14.center);
		\draw (21.center) to (15.center);
\end{tikzpicture}
\end{center}
\end{minipage}
\vspace{1em}
for every $u\colon X\to Y.$
\end{itemize}
\end{rmk}

\subsection{Strings for monoidal bicategories}
When we deal with a monoidal bicategory, we will always interpret juxtaposition of (1- and 2-) morphisms in a string diagram as a tensor product. Juxtaposition with 1 is tensoring with the identity morphism. For example, if $f,g\colon A\to B$ are 1-cells and $\alpha\colon f\Rightarrow g$, the 2-cell $\id_C\otimes\alpha\colon C\otimes f\Rightarrow C\otimes g$ will be denoted 
\begin{center}
\begin{tikzpicture}[scale=.5]
		\node [style=fixed size node] (0) at (0, 0) {$1\alpha$};
		\node [style=none] (1) at (-2, 0) {};
		\node [style=none] (2) at (-2.5, 0) {$1f$};
		\node [style=none] (3) at (2.5, 0) {$1g$};
		\node [style=none] (4) at (2, 0) {};
		\draw (1.center) to (0);
		\draw (0) to (4.center);
\end{tikzpicture}
\end{center}
and if $\beta\colon h\Rightarrow k$ is another 2-cell, the tensor product $\alpha\otimes\beta$ is denoted
\begin{center}
\begin{tikzpicture}[scale=.5]
		\node [style=fixed size node] (0) at (0, 0) {$\alpha\beta$};
		\node [style=none] (1) at (-2, 0) {};
		\node [style=none] (2) at (-2.5, 0) {$fh$};
		\node [style=none] (3) at (2.5, 0) {$gk.$};
		\node [style=none] (4) at (2, 0) {};
		\draw (1.center) to (0);
		\draw (0) to (4.center);
\end{tikzpicture}
\end{center}
\begin{rmk}
It is here important to notice that to talk about a tensored pair of morphisms $fg$ is an abuse of notation consisting of identifying the two composites $1g\circ f1$ or $f1\circ 1g$. This is fundamentally permitted by the role of the interchanger, since the interchanger axioms allow the usual calculus rule justifying the omissions in computations. The axioms for the interchanger \eqref{interchanger}, which we will make a heavy and sometimes implicit use of, especially in the proof of Theorem \ref{thmBtensorC}, are as follows:
\begin{equation}\label{axiom 1 interchanger}
\Sigma_{\id,g}=\id_g,\hspace{2em}\Sigma_{f,\id}=\id_f.
\tag{$i$}
\end{equation}
\begin{equation}
1\Sigma_{f,g}=\Sigma_{1f,g},\hspace{2em}\Sigma_{f1,g}=\Sigma_{f,1g},\hspace{2em}\Sigma_{f,g}1=\Sigma_{f,g1}.
\tag{$ii$}
\end{equation}
\begin{minipage}{.05\textwidth}
\begin{equation}\label{axiom 3 interchanger}
\phantom{c}
\tag{$iii$}
\end{equation}
\end{minipage}
\begin{minipage}{.9\textwidth}
\begin{center}

\end{minipage}
\end{rmk}

\begin{rmk}\label{join identities}
Since the braided monoidal bicategories we work with are always assumed to be semi-strict, we have the following equality which also will be repeatedly used in the proof of Theorem \ref{thmBtensorC}. We may refer to this operation as a ``joining of identities'', and it just says that for any $f,g\colon A\to A'$ and any 2-cell $\gamma\colon f\Rightarrow g$, we have
\begin{center}
%
\end{center}
and similarly on the left.
\end{rmk}

\subsection{Strings and braiding}\label{sec string braid}
The braiding axioms for a braided Gray monoid (our replacement for a braided monoidal bicategory), whose structure is defined at Definition \ref{def braided gray mon}, are the following identities:

\noindent
\begin{minipage}{.05\textwidth}
\begin{equation}\label{BA1}
\phantom{c}
\tag{BA1}
\end{equation}
\end{minipage}
\begin{minipage}{.9\textwidth}
\begin{center}
%
\end{center}
\end{minipage}
\vspace{1em}

\noindent
The above strict version of the braiding axioms can be found in \cite{D_coppet_2024}.

\begin{rmk}
Rules provided by \eqref{naturality axiom} and \eqref{modification axiom} in the context of braided monoidal bicategories express as the following. The naturality of the pseudonatural transformation $\beta$ gives
\vspace{1em}

\noindent
\begin{minipage}{.1\textwidth}
\begin{equation}\label{naturality beta}
\phantom{c}
\end{equation}
\end{minipage}
\begin{minipage}{.9\textwidth}
\begin{center}
%
\end{center}
\end{minipage}
\vspace{1em}

\noindent
The modification property for $R$ (and similarly for $S$) gives:

\vspace{1em}
\noindent
\begin{minipage}{.1\textwidth}
\begin{equation}\label{modification property R and S}
\phantom{c}
\end{equation}
\end{minipage}
\begin{minipage}{.9\textwidth}
\begin{center}
%
\end{center}
\end{minipage}
\end{rmk}

\subsection{Strings and pseudoadjunctions}\label{sec string adj}

\subsubsection{Pseudoadjunctions}

A coherent notion of adjunction of pseudofunctors between bicategories (or 1-cells in a possibly weak tricategory) is that of a \emph{pseudoadjunction}. A more general notion has been first introduced in \cite{BettiPower}, and is there called \emph{local adjunction}, but is probably nowadays more commonly known as \emph{lax 2-adjunction}. Let us recall these definitions and fix some notations. Two pseudofunctors $F\colon\cat C\myrightleftarrows{}\cat D\rcolon G$ (or two 1-cells in a tricategory) form a \emph{local adjunction} if there are
\begin{itemize}
\item Pseudonatural transformations (2-cells) $\eta\colon \id\Rightarrow GF,\varepsilon\colon FG\Rightarrow\id$
\item Modifications (3-cells) $s,t$, called \emph{triangulators}
\begin{center}
\begin{tikzcd}
F \arrow[rd, Rightarrow, bend right, no head] \arrow[r, "F\eta", Rightarrow] \arrow[rd, phantom, "\Arrowur s" description, bend left=1em, pos=0.65] & FGF \arrow[d, "\varepsilon F", Rightarrow] & G \arrow[r, "\eta G", Rightarrow] \arrow[rd, phantom, "\Arrowdl t" description, bend left=1em, pos=0.65] \arrow[rd, Rightarrow, no head, bend right] & GFG \arrow[d, "G\varepsilon", Rightarrow] \\
                                                                                                     & F                                          &                                                                                                       & G                                        
\end{tikzcd}
\end{center}
satisfying the \emph{swallowtail} equations:
\begin{center}\label{swallowtail1}
\begin{tikzcd}
\id_{\cat C} \arrow[r, "\eta", Rightarrow] \arrow[d, "\eta"', Rightarrow] & GF \arrow[rdd, Rightarrow, bend left=5em, no head] \arrow[d, "GF\eta", Rightarrow] \arrow[ld, phantom, "\Arrowdl\eta_\eta" description]              &    \\
GF \arrow[rrd, Rightarrow, bend right, no head] \arrow[r, "\eta GF"', Rightarrow]  & GFGF \arrow[rd, "G\varepsilon F" description, Rightarrow] \arrow[d, phantom, "\Arrowdl tF" description] \arrow[r, phantom, "\Arrowdl Gs" description] & {} \\
                                                                          & {}                                                                                                                                  & GF
\end{tikzcd}
$=\id_{\eta},$\hspace{2em}
\begin{tikzcd}
\id_{\cat D}                                                                & FG \arrow[rdd, Rightarrow, bend left=5em, no head] \arrow[l, "\varepsilon"', Rightarrow] \arrow[ld, phantom, "\Arrowdl{\varepsilon_\varepsilon}" description]                                  &                                                         \\
FG \arrow[rrd, Rightarrow, bend right, no head] \arrow[u, "\varepsilon", Rightarrow] & FGFG \arrow[u, "\varepsilon FG"', Rightarrow] \arrow[l, "FG\varepsilon", Rightarrow] \arrow[d, phantom, "\Arrowdl Ft" description] \arrow[r, phantom, "\Arrowdl sG" description] & {}                                                      \\
                                                                            & {}                                                                                                                                                            & FG \arrow[lu, "F\eta G" description, Rightarrow]
\end{tikzcd}
$=\id_\varepsilon$
\end{center}
\end{itemize}
This data is called a \emph{pseudoadjunction} whenever $s$ and $t$ are invertible. To have triangulators satisfying these equations corresponds to having a coherent structure of adjoint functors between the hom-categories $\cat D(Fc,d)\myrightleftarrows{}\cat C(c,Gd)$, and clearly to have a pseudoadjunction corresponds to these adjunctions being adjoint equivalences.

\subsubsection{Tensor-Hom}\label{subsect tensor-hom}
The special case of a pseudoadjunction defining a right closed structure on a monoidal bicategory $\cat V$ have been briefly treated in Section 7.2 of \cite{garner2015enriched}. Their version is however chosen to be \emph{incoherent} ($s$ and $t$ are only required to exist). For us, a \emph{right closed monoidal bicategory} $\cat V$ is one such that the pseudofunctor $-\otimes A\colon\cat V\to\cat V$ forms a pseudoadjunction $-\otimes A\dashv[A,-]$ with a given pseudofunctor $[A,-]\colon\cat V\to\cat V$ for every object $A$ in $\cat V$.
\begin{rmk}\label{parametric family rmk}
It is customary (at least in the usual 1-categorical setting) to extend this family of pseudoadjunctions to what is usually called a \emph{parametric family} of pseudoadjunctions, meaning one for which the family of equivalences $\cat V(B\otimes A,C)\simeq\cat V(B,[A,C])$ also come with the structure of a pseudonatural equivalence $\cat V(B\otimes-,C)\Rightarrow\cat V(B,[-,C])$ for every pair of objects $B,C$. This depends on how the pseudofunctor $[-,-]\colon\cat V\times\cat V\to\cat V$ is defined in its first variable on morphisms. The bicategorical Yoneda Lemma tells then how to do so, by imposing commutativity of the square
\begin{center}
\begin{tikzcd}
{\mathcal V(B\otimes A',C)} \arrow[d, "{\mathcal V(B\otimes f,C)}"'] \arrow[r, "\cong"] & {\mathcal V(B,[A',C])} \arrow[d, dashed] \\
{\mathcal V(B\otimes A,C)} \arrow[r, "\cong"]                                           & {\mathcal V(B,[A,C]).}                   
\end{tikzcd}
\end{center}
\end{rmk}

\subsubsection{Naturality of unit and counit}
The structure of pseudonatural transformation $\eta_f$ and $\varepsilon_f$ for the (parametric family of) pseudoadjunctions $-\otimes A\dashv[A,-]$, are 2-cells of the following types, for every 1-morphism $f\colon B\to C$ in $\cat V$,

\noindent
\begin{minipage}{.1\textwidth}
\begin{equation}\label{epsilon-eta string rule}
\phantom{c}
\end{equation}
\end{minipage}
\begin{minipage}{.9\textwidth}
\begin{center}
\begin{tikzpicture}[scale=.5]
		\node [style=none] (0) at (-7, -1.5) {};
		\node [style=none] (1) at (-7, 1.5) {};
		\node [style=none] (2) at (-2, 1.5) {};
		\node [style=none] (3) at (-2, -1.5) {};
		\node [style=none] (4) at (-1.5, -1.5) {$f$};
		\node [style=none] (5) at (-1.5, 1.5) {$\eta$};
		\node [style=none] (6) at (-7.5, -1.5) {$\eta$};
		\node [style=none] (7) at (-7.75, 1.5) {$[1,f1]$};
		\node [style=none] (8) at (7, -1.5) {};
		\node [style=none] (9) at (7, 1.5) {};
		\node [style=none] (10) at (2, 1.5) {};
		\node [style=none] (11) at (2, -1.5) {};
		\node [style=none] (12) at (1.5, -1.5) {$\varepsilon$};
		\node [style=none] (13) at (1.5, 1.5) {$f$};
		\node [style=none] (14) at (7.75, -1.5) {$[1,f]1.$};
		\node [style=none] (15) at (7.5, 1.5) {$\varepsilon$};
		\node [style=none] (16) at (-4.5, -2.25) {$B$};
		\node [style=none] (17) at (-2.75, 0) {$C$};
		\node [style=none] (18) at (-6.5, 0) {$[A,BA]$};
		\node [style=none] (19) at (-4.5, 2.25) {$[A,CA]$};
		\node [style=none] (20) at (2.5, 0) {$B$};
		\node [style=none] (21) at (4.5, 2.25) {$C$};
		\node [style=none] (22) at (4.5, -2.25) {$[A,B]A$};
		\node [style=none] (23) at (6.25, 0) {$[A,C]A$};
		
		\draw [in=180, out=0] (1.center) to (3.center);
		\draw [in=0, out=-180] (2.center) to (0.center);
		\draw [in=0, out=180] (9.center) to (11.center);
		\draw [in=180, out=0] (10.center) to (8.center);
\end{tikzpicture}
\end{center}
\end{minipage}

The naturality axiom for this structure then says that for any 2-cell $\gamma\colon f\Rightarrow g$ in a closed monoidal bicategory $\cat V$, we have, from \eqref{naturality axiom},

\noindent
\begin{minipage}{.1\textwidth}
\begin{equation}\label{naturality eta}
\phantom{c}
\end{equation}
\end{minipage}
\begin{minipage}{.9\textwidth}
\begin{center}
\begin{tikzpicture}[scale=.5]
		\node [style=none] (0) at (-1.5, -1) {$\eta$};
		\node [style=none] (1) at (-1, 1) {$[1,g1]$};
		\node [style=none] (2) at (1.5, -1) {$f$};
		\node [style=none] (3) at (1.5, 1) {$\eta$};
		\node [style=none] (4) at (0, 0) {$=$};
		\node [style=none] (5) at (-2, -1) {};
		\node [style=none] (6) at (-2, 1) {};
		\node [style=none] (7) at (2, -1) {};
		\node [style=none] (8) at (2, 1) {};
		\node [style=none] (9) at (4.5, -1) {};
		\node [style=none] (10) at (4.5, 1) {};
		\node [style=squared] (11) at (6.5, 1) {$[1,\gamma1]$};
		\node [style=none] (12) at (8.5, 1) {};
		\node [style=none] (13) at (-4.5, -1) {};
		\node [style=none] (14) at (-4.5, 1) {};
		\node [style=squared] (15) at (-6.5, -1) {$\gamma$};
		\node [style=none] (16) at (-8.5, -1) {};
		\node [style=none] (17) at (-8.5, 1) {};
		\node [style=none] (18) at (-9, -1) {$f$};
		\node [style=none] (19) at (-9, 1) {$\eta$};
		\node [style=none] (20) at (8.5, -1) {};
		\node [style=none] (21) at (9, -1) {$\eta$};
		\node [style=none] (22) at (9.5, 1) {$[1,g1]$};
		\node [style=none] (23) at (-4.5, -1.5) {$g$};
		\node [style=none] (24) at (4.5, 1.5) {$[1,f1]$};
		
		\draw [in=0, out=180] (5.center) to (14.center);
		\draw (14.center) to (17.center);
		\draw (16.center) to (15);
		\draw (15) to (13.center);
		\draw [in=180, out=0] (13.center) to (6.center);
		\draw (10.center) to (11);
		\draw (11) to (12.center);
		\draw [in=-180, out=0] (7.center) to (10.center);
		\draw [in=-180, out=0] (8.center) to (9.center);
		\draw (9.center) to (20.center);
\end{tikzpicture}

\end{center}
\end{minipage}

and

\noindent
\begin{minipage}{.1\textwidth}
\begin{equation}\label{naturality counit}
\phantom{c}
\end{equation}
\end{minipage}
\begin{minipage}{.9\textwidth}
\begin{center}
\begin{tikzpicture}[scale=.5]
		\node [style=none] (0) at (-1.5, 1) {$\varepsilon$};
		\node [style=none] (1) at (-1, -1) {$[1,g]1$};
		\node [style=none] (2) at (1.5, 1) {$f$};
		\node [style=none] (3) at (1.5, -1) {$\varepsilon$};
		\node [style=none] (4) at (0, 0) {$=$};
		\node [style=none] (5) at (-2, 1) {};
		\node [style=none] (6) at (-2, -1) {};
		\node [style=none] (7) at (2, 1) {};
		\node [style=none] (8) at (2, -1) {};
		\node [style=none] (9) at (4.5, 1) {};
		\node [style=none] (10) at (4.5, -1) {};
		\node [style=squared] (11) at (6.5, -1) {$[1,\gamma]1$};
		\node [style=none] (12) at (8.5, -1) {};
		\node [style=none] (13) at (-4.5, 1) {};
		\node [style=none] (14) at (-4.5, -1) {};
		\node [style=squared] (15) at (-6.5, 1) {$\gamma$};
		\node [style=none] (16) at (-8.5, 1) {};
		\node [style=none] (17) at (-8.5, -1) {};
		\node [style=none] (18) at (-9, 1) {$f$};
		\node [style=none] (19) at (-9, -1) {$\varepsilon$};
		\node [style=none] (20) at (8.5, 1) {};
		\node [style=none] (21) at (9, 1) {$\varepsilon$};
		\node [style=none] (22) at (9.5, -1) {$[1,g]1.$};
		\node [style=none] (23) at (-4.5, 1.5) {$g$};
		\node [style=none] (24) at (4.5, -1.5) {$[1,f]1$};
		
		\draw [in=0, out=-180] (5.center) to (14.center);
		\draw (14.center) to (17.center);
		\draw (16.center) to (15);
		\draw (15) to (13.center);
		\draw [in=-180, out=0] (13.center) to (6.center);
		\draw (10.center) to (11);
		\draw (11) to (12.center);
		\draw [in=180, out=0] (7.center) to (10.center);
		\draw [in=180, out=0] (8.center) to (9.center);
		\draw (9.center) to (20.center);
\end{tikzpicture}
\end{center}
\end{minipage}

\subsubsection{Triangulators}
The string-diagrammatic form for the structure of a pseudoadjunction is

\noindent
\begin{minipage}{.1\textwidth}
\begin{equation}\label{triangulators}
\phantom{c}
\end{equation}
\end{minipage}
\begin{minipage}{.9\textwidth}
\begin{center}
\begin{tikzpicture}[scale=.5]
		\node [style=squared] (0) at (0, 0) {$s$};
		\node [style=none] (1) at (3, 2) {};
		\node [style=none] (2) at (3, -2) {};
		\node [style=none] (3) at (3.5, 2) {$\varepsilon$};
		\node [style=none] (4) at (3.5, -2) {$\eta 1$};
		\node [style=none] (5) at (2, 0) {$[A,BA]A$};
		\node [style=none] (6) at (-1.5, 0) {$BA$};
		
		\draw [in=75, out=-180] (1.center) to (0);
		\draw [in=-180, out=-75] (0) to (2.center);
\end{tikzpicture}
\hspace{1em}\raisebox{3em}{\text{,}}\hspace{3em}
\begin{tikzpicture}[scale=.5]
		\node [style=squared] (0) at (0, 0) {$t$};
		\node [style=none] (1) at (-3, 1.75) {};
		\node [style=none] (2) at (-3, -2) {};
		\node [style=none] (5) at (2, 0) {$[A,B]$};
		\node [style=none] (6) at (-2.5, 0) {$[A,[A,B]A]$};
		\node [style=none] (7) at (-3.75, 1.75) {$[1,\varepsilon]$};
		\node [style=none] (8) at (-3.5, -2) {$\eta$};
		
		\draw [in=105, out=0] (1.center) to (0);
		\draw [in=0, out=-105] (0) to (2.center);
\end{tikzpicture}
\end{center}
\end{minipage}

and the swallowtail equations then translate to

\noindent
\begin{minipage}{.1\textwidth}
\begin{equation}\label{swallowtail1}
\phantom{c}
\end{equation}
\end{minipage}
\begin{minipage}{.9\textwidth}
\begin{center}
\begin{tikzpicture}[scale=.5]
		\node [style=squared] (0) at (-7.75, 1.5) {$[1,s]$};
		\node [style=squared] (1) at (-3.75, 1.5) {$t$};
		\node [style=none] (2) at (-5.75, 0) {};
		\node [style=none] (3) at (-8.75, -1) {};
		\node [style=none] (4) at (-2.75, -1) {};
		\node [style=none] (5) at (-5.75, 3.2) {$[1,\varepsilon]$};
		\node [style=none] (6) at (-4, 0.25) {$\eta$};
		\node [style=none] (7) at (-7.5, 0.25) {$[1,\eta 1]$};
		\node [style=none] (8) at (-9.25, -1) {$\eta$};
		\node [style=none] (9) at (-2.25, -1) {$\eta$};
		\node [style=none] (10) at (0, 0) {$=$};
		\node [style=none] (11) at (2, 0) {$\eta$};
		\node [style=none] (12) at (2.5, 0) {};
		\node [style=none] (13) at (6, 0) {};
		\node [style=none] (14) at (6.5, 0) {$\eta$};
		
		\draw [in=105, out=75, looseness=0.75] (0) to (1);
		\draw [in=45, out=-105] (1) to (2.center);
		\draw [in=-135, out=0, looseness=0.75] (3.center) to (2.center);
		\draw [in=135, out=-75] (0) to (2.center);
		\draw [in=180, out=-45, looseness=0.75] (2.center) to (4.center);
		\draw (12.center) to (13.center);
\end{tikzpicture}
\end{center}
\end{minipage}

\noindent
\begin{minipage}{.1\textwidth}
\begin{equation}\label{swallowtail2}
\phantom{c}
\end{equation}
\end{minipage}
\begin{minipage}{.9\textwidth}
\begin{center}
\begin{tikzpicture}[scale=.5]
		\node [style=squared] (0) at (-7.75, -1.5) {$s$};
		\node [style=squared] (1) at (-3.75, -1.5) {$t1$};
		\node [style=none] (2) at (-5.75, 0) {};
		\node [style=none] (3) at (-8.75, 1) {};
		\node [style=none] (4) at (-2.75, 1) {};
		\node [style=none] (5) at (-5.75, -3.2) {$\eta 1$};
		\node [style=none] (6) at (-4, -0.25) {$[1,\varepsilon]1$};
		\node [style=none] (7) at (-7.5, -0.25) {$\varepsilon$};
		\node [style=none] (8) at (-9.25, 1) {$\varepsilon$};
		\node [style=none] (9) at (-2.25, 1) {$\varepsilon$};
		\node [style=none] (10) at (0, 0) {$=$};
		\node [style=none] (11) at (2, 0) {$\varepsilon$};
		\node [style=none] (12) at (2.5, 0) {};
		\node [style=none] (13) at (6, 0) {};
		\node [style=none] (14) at (6.5, 0) {$\varepsilon$};
		
		\draw [in=-105, out=-75, looseness=0.75] (0) to (1);
		\draw [in=-45, out=105] (1) to (2.center);
		\draw [in=135, out=0, looseness=0.75] (3.center) to (2.center);
		\draw [in=-135, out=75] (0) to (2.center);
		\draw [in=-180, out=45, looseness=0.75] (2.center) to (4.center);
		\draw (12.center) to (13.center);
\end{tikzpicture}

\end{center}
\end{minipage}

Moreover, there is the general set of rules provided by $s$ and $t$ being modifications. For that we refer to the general case \eqref{modification axiom}.

\section{Building $\cat V$-bicategories}\label{section Vbicat}

Enriched bicategories have been introduced by Garner--Shulman in \cite{garner2015enriched}. In their setting, no use of strictification was assumed in the definition. For our purposes, it is convenient to directly consider enrichments over Gray monoids and - subsequently, for the constructions requiring a braiding - over semi-strict braided monoidal bicategories. The underlying reasoning is that these assumptions don't really restrain the generality of our results and constructions, thanks to Theorem \ref{strictification gurski}.

\subsection{The tricategory of $\cat V$-bicategories}

As proved in Section 4 of \cite{garner2015enriched}, $\cat V$-bicategories, $\cat V$-pseudofunctors, $\cat V$-pseudonatural transformations and $\cat V$-modifications assemble into a tricategory. Subsequently we recall terminology, structure and axioms for each kind of cell of this tricategory.

\subsubsection{$\cat V$-bicategories}

\begin{defn}[\cite{garner2015enriched} Definition 3.1]
Let $\cat V$ be a monoidal bicategory. A \emph{$\cat V$-bicategory} $\cat C$ is the data of
\begin{itemize}[leftmargin=*]
\item A class of objects $\cat C_0$
\item For every pair of objects $c,d$ in $\cat C_0$ an object $\cat C(c,d)$ in $\cat V$, called \emph{hom-object}, together with composition and unit 1-morphisms, for every tiple of objects $c,d,e$
\begin{align*}
&m_{c,d,e}\colon\cat{C}(d,e)\otimes\cat C (c,d)\longrightarrow\cat C (c,e)\\
&u_c\colon \mathbb{1}\longrightarrow \cat C (c,c)
\end{align*}
\item 2-isomorphisms in $\cat V$ (called \emph{associator} and \emph{left} and \emph{right unitors})
\begin{center}
\begin{tikzcd}
	{\mathcal C(e,f)\otimes \mathcal C(d,e)\otimes \mathcal C(c,d)} & {\mathcal C(e,f)\otimes \mathcal C(c,e)} \\
	{ \mathcal C(d,f)\otimes\mathcal C(c,d)} & {\mathcal C(c,f)}
	\arrow["{1\otimes m}", from=1-1, to=1-2]
	\arrow["{m\otimes 1}"', from=1-1, to=2-1]
	\arrow["m", from=1-2, to=2-2]
	\arrow[phantom, "\alpha\Arrowur"{description}, shift right=2, shorten <=25pt, shorten >=25pt, Rightarrow, from=2-1, to=1-2]
	\arrow["m"', from=2-1, to=2-2]
\end{tikzcd}
\end{center}
\begin{center}
\begin{tikzcd}
	{\mathbb 1\otimes\mathcal C(c,d)} &&& {\mathcal C(c,d)\otimes\mathbb 1} \\
	{\mathcal C(d,d)\otimes\mathcal C(c,d)} && {\mathcal C(c,d)} & {\mathcal C(c,d)\otimes\mathcal C(c,c)} && {\mathcal C(c,d)}
	\arrow["{u\otimes 1}"', from=1-1, to=2-1]
	\arrow[""{name=0, anchor=center, inner sep=0}, curve={height=-12pt}, equals, from=1-1, to=2-3]
	\arrow["{1\otimes u}"', from=1-4, to=2-4]
	\arrow[""{name=1, anchor=center, inner sep=0}, curve={height=-12pt}, equals, from=1-4, to=2-6]
	\arrow["m"', from=2-1, to=2-3]
	\arrow["m"', from=2-4, to=2-6]
	\arrow[phantom,"\lambda\Arrowur"{description}, shift right=2, shorten <=16pt, shorten >=16pt, Rightarrow, from=2-1, to=0]
	\arrow[phantom, "\rho\Arrowur"{description}, shift right=2, shorten <=16pt, shorten >=16pt, Rightarrow, from=2-4, to=1]
\end{tikzcd}
\end{center}
\end{itemize}
subject to the identity and associativity coherence axioms below
\vspace{1em}

\noindent
\begin{minipage}{.1\textwidth}
\begin{equation}\label{IC Vbicat}
\phantom{c}
\tag{IC}
\end{equation}
\end{minipage}
\begin{minipage}{.9\textwidth}
\begin{center}
\begin{tikzpicture}[scale=.5]
		\node [style=none] (0) at (0, 0) {$=$};
		\node [style=fixed size node] (1) at (-6.75, 1) {$\alpha$};
		\node [style=fixed size node] (2) at (-3.75, -1) {$1\lambda$};
		\node [style=fixed size node] (3) at (3.75, -1) {$\rho 1$};
		\node [style=none] (4) at (-8.75, 2) {};
		\node [style=none] (5) at (-8.75, 0) {};
		\node [style=none] (6) at (-8.75, -2) {};
		\node [style=none] (7) at (-1.75, 2) {};
		\node [style=none] (8) at (1.5, 2) {$m$};
		\node [style=none] (9) at (1.5, 0) {$m1$};
		\node [style=none] (10) at (1.5, -2) {$1u1$};
		\node [style=none] (11) at (8.75, 2) {};
		\node [style=none] (12) at (-4.75, 2) {};
		\node [style=none] (13) at (-9.25, 2) {$m$};
		\node [style=none] (14) at (-9.25, 0) {$m1$};
		\node [style=none] (15) at (-9.25, -2) {$1u1$};
		\node [style=none] (16) at (-1.25, 2) {$m$};
		\node [style=none] (17) at (2, 2) {};
		\node [style=none] (18) at (2, 0) {};
		\node [style=none] (19) at (2, -2) {};
		\node [style=none] (20) at (-4.75, 0.25) {$1m$};
		\node [style=none] (21) at (9.25, 2) {$m$};
		
		\draw [in=120, out=0] (4.center) to (1);
		\draw [in=0, out=-120] (1) to (5.center);
		\draw [in=60, out=-180, looseness=1.25] (12.center) to (1);
		\draw [in=180, out=0] (12.center) to (7.center);
		\draw [in=120, out=-60] (1) to (2);
		\draw [in=0, out=-135, looseness=0.75] (2) to (6.center);
		\draw [in=-120, out=0] (19.center) to (3);
		\draw [in=0, out=120] (3) to (18.center);
		\draw (17.center) to (11.center);
\end{tikzpicture}
\end{center}
\end{minipage}
\vspace{1em}

\noindent
\begin{minipage}{.1\textwidth}
\begin{equation}\label{AC Vbicat}
\phantom{c}
\tag{AC}
\end{equation}
\end{minipage}
\begin{minipage}{.9\textwidth}
\begin{center}
\begin{tikzpicture}[scale=.5]
		\node [style=fixed size node] (0) at (-5.5, 1) {$\alpha$};
		\node [style=fixed size node] (1) at (-7.5, -1) {$\alpha 1$};
		\node [style=fixed size node] (2) at (-3.5, -1) {$1\alpha$};
		\node [style=fixed size node] (4) at (4, 1) {$\alpha$};
		\node [style=none] (5) at (-9, 2) {};
		\node [style=none] (6) at (-9, 0) {};
		\node [style=none] (7) at (-9, -2) {};
		\node [style=none] (8) at (-9.5, 2) {$m$};
		\node [style=none] (9) at (-9.5, 0) {$m1$};
		\node [style=none] (10) at (-9.5, -2) {$m11$};
		\node [style=none] (11) at (-2, 2) {};
		\node [style=none] (12) at (-2, 0) {};
		\node [style=none] (13) at (-2, -2) {};
		\node [style=none] (14) at (-1.5, 2) {$m$};
		\node [style=none] (15) at (-1.5, 0) {$1m$};
		\node [style=none] (16) at (-1.5, -2) {$11m$};
		\node [style=none] (17) at (0, 0) {$=$};
		\node [style=none] (21) at (2, 2) {};
		\node [style=none] (22) at (2, 0) {};
		\node [style=none] (23) at (2, -2) {};
		\node [style=none] (24) at (1.5, 2) {$m$};
		\node [style=none] (25) at (1.5, 0) {$m1$};
		\node [style=none] (26) at (1.5, -2) {$m11$};
		\node [style=none] (27) at (9, 2) {};
		\node [style=none] (28) at (9, 0) {};
		\node [style=none] (29) at (9, -2) {};
		\node [style=none] (30) at (9.5, 2) {$m$};
		\node [style=none] (31) at (9.5, 0) {$1m$};
		\node [style=none] (32) at (9.5, -2) {$11m$};
		\node [style=fixed size node] (33) at (7, 1) {$\alpha$};
		\node [style=none] (34) at (-6.75, 0.25) {$m1$};
		\node [style=none] (35) at (-4.25, 0.25) {$1m$};
		\node [style=none] (36) at (-5.5, -2.25) {$1m1$};
		\node [style=none] (37) at (5.5, 2) {$m$};
		\node [style=none] (38) at (6.75, -0.25) {$m1$};
		\node [style=none] (39) at (4.25, -0.25) {$1m$};
		
		\draw [in=120, out=0] (6.center) to (1);
		\draw [in=0, out=-120] (1) to (7.center);
		\draw [in=-135, out=-45, looseness=0.75] (1) to (2);
		\draw [in=-120, out=60] (1) to (0);
		\draw [in=0, out=135, looseness=0.75] (0) to (5.center);
		\draw [in=120, out=-60] (0) to (2);
		\draw [in=-60, out=-180] (13.center) to (2);
		\draw [in=-180, out=60] (2) to (12.center);
		\draw [in=180, out=45, looseness=0.75] (0) to (11.center);
		\draw [in=135, out=0] (21.center) to (4);
		\draw [in=0, out=-135] (4) to (22.center);
		\draw [in=135, out=45] (4) to (33);
		\draw [in=-180, out=-60, looseness=0.75] (4) to (29.center);
		\draw [in=-120, out=0, looseness=0.75] (23.center) to (33);
		\draw [in=180, out=-45] (33) to (28.center);
		\draw [in=-180, out=45] (33) to (27.center);
\end{tikzpicture}
\end{center}
\end{minipage}
\end{defn}

\begin{ex}
The unit $\cat V$-bicategory $\cat I$ has one object $\ast$ and the only hom-object $\cat I(\ast,\ast)=\mathbb 1$. Since we deal in fact with a Gray monoid, multiplication and unit can be chosen to be identical, and so are $\alpha,\lambda,\rho$. This makes all coherence axioms trivially satisfied.
\end{ex}

\begin{rmk}\label{underlying bicat}
For any $\cat V$-bicategory $\cat C$ there is an associated underlying bicategory $\cat C^\flat$ having the same objects and hom-categories $\cat C^\flat(c,d)\coloneqq \cat V(\mathbb 1,\cat C(c,d))$. The left and right unitors and the associator for $\cat C^\flat$ are easily build from $\lambda,\rho$ and $\alpha$ respectively.
\end{rmk}

\subsubsection{$\cat V$-pseudofunctors}

\begin{defn}\label{def pseudofunctor}
Let $\cat C,\cat D$ be $\cat V$-bicategories. A $\cat V$\emph{-pseudofunctor} $F\colon\cat C\to\cat D$ consists of
\begin{itemize}
\item an object $F(a)$ in $\cat D_0$ for every object $a$ in $\cat C_0$
\item For every pair of objects $(a,b)$ of $\cat C$, a morphism in $\cat V$
\[F_{a,b}\colon\cat C(a,b)\longrightarrow\cat D(Fa,Fb)\]
\item Two 2-cells in $\cat V$

\begin{minipage}{.05\textwidth}
\begin{equation}\label{fun e un}
\phantom{c}
\end{equation}
\end{minipage}
\begin{minipage}{.95\textwidth}
\begin{center}
\begin{tikzpicture}[scale=.5]
		\node [style=none] (0) at (-4.25, -1.75) {$FF$};
		\node [style=none] (1) at (-4.25, 1.75) {$m$};
		\node [style=none] (2) at (-3.75, -1.75) {};
		\node [style=none] (3) at (-3.75, 1.75) {};
		\node [style=none] (4) at (3.75, -1.75) {};
		\node [style=none] (5) at (3.75, 1.75) {};
		\node [style=none] (6) at (4.25, -1.75) {$m$};
		\node [style=none] (7) at (4.25, 1.75) {$F$};
		\node [style=none] (8) at (0, -2.25) {$\mathcal C(b,c)\mathcal C(a,b)$};
		\node [style=none] (9) at (3.5, 0) {$\mathcal C(a,c)$};
		\node [style=none] (10) at (-3.5, 0) {$\mathcal D(Fb,Fc)\mathcal D(Fa,Fb)$};
		\node [style=none] (11) at (0, 2) {$\mathcal D(Fa,Fc)$};

		\draw [in=-180, out=0] (2.center) to (5.center);
		\draw [in=180, out=0] (3.center) to (4.center);
\end{tikzpicture}
\hspace{1em}\raisebox{2em}{,}\hspace{2em}
\begin{tikzpicture}[scale=.5]
		\node [style=none] (0) at (0, 0) {};
		\node [style=none] (1) at (-3, 0) {};
		\node [style=none] (2) at (3, 2) {};
		\node [style=none] (3) at (3, -2) {};
		\node [style=none] (4) at (-3.5, 0) {$u$};
		\node [style=none] (5) at (3.5, 2) {$F$};
		\node [style=none] (6) at (3.5, -2) {$u$};
		\node [style=none] (7) at (-1.25, -1.75) {$\mathbb 1$};
		\node [style=none] (8) at (-1.25, 1.75) {$\mathcal D(Fa,Fa)$};
		\node [style=none] (9) at (2.25, 0) {$\mathcal C(a,a)$};

		\draw (1.center) to (0.center);
		\draw [in=-180, out=60] (0.center) to (2.center);
		\draw [in=180, out=-75] (0.center) to (3.center);
\end{tikzpicture}

\end{center}
\end{minipage}
\end{itemize}
subject to the following axioms:

\noindent
\begin{minipage}{.1\textwidth}
\begin{equation}\label{psfun ax1}
\phantom{c}
\end{equation}
\end{minipage}
\begin{minipage}{.9\textwidth}
\begin{center}
\begin{tikzpicture}[scale=.5]
		\node [style=none] (3) at (-8.25, -2.25) {};
		\node [style=none] (4) at (-11.25, 0) {};
		\node [style=none] (5) at (-11.25, -2.25) {};
		\node [style=none] (15) at (-11.25, 2.25) {};
		\node [style=none] (18) at (-12, -2.25) {$FFF$};
		\node [style=none] (19) at (-11.75, 2.25) {$m$};
		\node [style=none] (20) at (-11.75, 0) {$m1$};
		\node [style=none] (21) at (-6.75, -2.75) {$m1$};
		\node [style=none] (22) at (-8.25, 0.5) {$FF$};
		\node [style=none] (23) at (-5.25, 2.25) {};
		\node [style=none] (24) at (-5.25, 0) {};
		\node [style=none] (25) at (-8.25, 2.25) {};
		\node [style=none] (27) at (-5.25, -0.5) {$m$};
		\node [style=none] (28) at (-5.25, -2.25) {};
		\node [style=fixed size node] (29) at (-3.5, -1.25) {$\alpha$};
		\node [style=none] (30) at (-2, 0) {};
		\node [style=none] (31) at (-2, -2.25) {};
		\node [style=none] (32) at (-2, 2.25) {};
		\node [style=none] (33) at (0, 0) {$=$};
		\node [style=none] (34) at (-1.5, 0) {$m$};
		\node [style=none] (35) at (-1.5, -2.25) {$1m$};
		\node [style=none] (36) at (-1.5, 2.25) {$F$};
		\node [style=none] (42) at (1.25, -2.25) {$FFF$};
		\node [style=none] (43) at (1.5, 2.25) {$m$};
		\node [style=none] (44) at (1.5, 0) {$m1$};
		\node [style=none] (68) at (8.25, 2.25) {};
		\node [style=none] (69) at (11.25, 0) {};
		\node [style=none] (70) at (11.25, 2.25) {};
		\node [style=none] (71) at (11.25, -2.25) {};
		\node [style=none] (75) at (5.5, 0.5) {$m1$};
		\node [style=none] (76) at (8.25, -0.5) {$FF$};
		\node [style=none] (77) at (5.25, -2.25) {};
		\node [style=none] (78) at (5.25, 0) {};
		\node [style=none] (79) at (8.25, -2.25) {};
		\node [style=none] (80) at (6.75, 2.75) {$m$};
		\node [style=none] (81) at (5.25, 2.25) {};
		\node [style=fixed size node] (82) at (3.5, 1.25) {$\alpha$};
		\node [style=none] (83) at (2, 0) {};
		\node [style=none] (84) at (2, 2.25) {};
		\node [style=none] (85) at (2, -2.25) {};
		\node [style=none] (86) at (11.75, 0) {$m$};
		\node [style=none] (87) at (11.75, -2.25) {$1m$};
		\node [style=none] (88) at (11.75, 2.25) {$F$};
		\draw [in=180, out=0, looseness=0.75] (4.center) to (3.center);
		\draw [in=165, out=0, looseness=0.75] (25.center) to (24.center);
		\draw [in=180, out=0] (15.center) to (25.center);
		\draw [in=180, out=0] (3.center) to (28.center);
		\draw [in=120, out=-15] (24.center) to (29);
		\draw [in=0, out=-120] (29) to (28.center);
		\draw [in=-60, out=180] (31.center) to (29);
		\draw [in=-180, out=60] (29) to (30.center);
		\draw [in=360, out=180] (32.center) to (23.center);
		\draw [in=0, out=180, looseness=0.50] (23.center) to (5.center);
		\draw [in=0, out=-180, looseness=0.75] (69.center) to (68.center);
		\draw [in=-15, out=-180, looseness=0.75] (79.center) to (78.center);
		\draw [in=0, out=-180] (71.center) to (79.center);
		\draw [in=0, out=-180] (68.center) to (81.center);
		\draw [in=-60, out=165] (78.center) to (82);
		\draw [in=-180, out=60] (82) to (81.center);
		\draw [in=120, out=0] (84.center) to (82);
		\draw [in=0, out=-120] (82) to (83.center);
		\draw [in=180, out=0] (85.center) to (77.center);
		\draw [in=-180, out=0, looseness=0.50] (77.center) to (70.center);
\end{tikzpicture}
\end{center}
\end{minipage}

\noindent
\begin{minipage}{.1\textwidth}
\begin{equation}\label{psfun ax2}
\phantom{c}
\end{equation}
\end{minipage}
\begin{minipage}{.9\textwidth}
\begin{center}
\begin{tikzpicture}[scale=.5]
		\node [style=none] (0) at (1.5, 2) {$m$};
		\node [style=none] (1) at (2, 2) {};
		\node [style=none] (2) at (1.5, 0) {$u1$};
		\node [style=none] (3) at (2, 0) {};
		\node [style=none] (4) at (1.5, -2) {$1F$};
		\node [style=none] (5) at (2, -2) {};
		\node [style=none] (6) at (10, -2) {};
		\node [style=none] (7) at (10.5, -2) {$F$};
		\node [style=fixed size node] (8) at (3.5, 1) {$\lambda$};
		\node [style=none] (9) at (-10.5, 2) {$m$};
		\node [style=none] (10) at (-10, 2) {};
		\node [style=none] (11) at (-10.5, 0) {$u1$};
		\node [style=none] (12) at (-10, 0) {};
		\node [style=none] (13) at (-10.5, -2) {$1F$};
		\node [style=none] (14) at (-10, -2) {};
		\node [style=none] (15) at (-2, 2) {};
		\node [style=none] (16) at (-1.5, 2) {$F$};
		\node [style=none] (17) at (0, 0) {$=$};
		\node [style=none] (18) at (-8.75, 0) {};
		\node [style=none] (20) at (-7.25, 1) {};
		\node [style=none] (21) at (-8.5, -2) {};
		\node [style=none] (22) at (-5.75, -1.75) {};
		\node [style=fixed size node] (23) at (-2.75, -0.5) {$\lambda$};
		\node [style=none] (24) at (-6, 2) {};
		\node [style=none] (25) at (-5.75, 0) {};
		\node [style=none] (26) at (-2.75, 0.5) {$m$};
		\node [style=none] (27) at (-7.25, 1.25) {$F1$};
		\node [style=none] (28) at (-4.75, -2.25) {$u1$};
		\node [style=none] (29) at (-4.75, -0.25) {$FF$};
		\node [style=none] (30) at (-8.5, -1) {$u1$};

		\draw [in=180, out=0] (5.center) to (6.center);
		\draw [in=120, out=0] (1.center) to (8);
		\draw [in=0, out=-120] (8) to (3.center);
		\draw (12.center) to (18.center);
		\draw [in=-180, out=60] (18.center) to (20.center);
		\draw [in=180, out=0] (14.center) to (21.center);
		\draw [in=180, out=-60, looseness=0.75] (18.center) to (22.center);
		\draw (10.center) to (24.center);
		\draw [in=105, out=0] (24.center) to (23);
		\draw [in=0, out=-120, looseness=0.75] (23) to (22.center);
		\draw [in=240, out=0] (21.center) to (25.center);
		\draw [in=180, out=0] (25.center) to (15.center);
		\draw [in=105, out=0] (20.center) to (25.center);
\end{tikzpicture}
\end{center}
\end{minipage}

\noindent
\begin{minipage}{.1\textwidth}
\begin{equation}\label{psfun ax3}
\phantom{c}
\end{equation}
\end{minipage}
\begin{minipage}{.9\textwidth}
\begin{center}
\begin{tikzpicture}[scale=.5]
		\node [style=none] (0) at (1.5, 2) {$m$};
		\node [style=none] (1) at (2, 2) {};
		\node [style=none] (2) at (1.5, 0) {$1u$};
		\node [style=none] (3) at (2, 0) {};
		\node [style=none] (4) at (1.5, -2) {$F1$};
		\node [style=none] (5) at (2, -2) {};
		\node [style=none] (6) at (10, -2) {};
		\node [style=none] (7) at (10.5, -2) {$F$};
		\node [style=fixed size node] (8) at (3.5, 1) {$\rho$};
		\node [style=none] (9) at (-10.5, 2) {$m$};
		\node [style=none] (10) at (-10, 2) {};
		\node [style=none] (11) at (-10.5, 0) {$1u$};
		\node [style=none] (12) at (-10, 0) {};
		\node [style=none] (13) at (-10.5, -2) {$F1$};
		\node [style=none] (14) at (-10, -2) {};
		\node [style=none] (15) at (-2, 2) {};
		\node [style=none] (16) at (-1.5, 2) {$F$};
		\node [style=none] (17) at (0, 0) {$=$};
		\node [style=none] (18) at (-8.75, 0) {};
		\node [style=none] (20) at (-7.25, 1) {};
		\node [style=none] (21) at (-8.5, -2) {};
		\node [style=none] (22) at (-5.75, -1.75) {};
		\node [style=fixed size node] (23) at (-2.75, -0.5) {$\rho$};		
		\node [style=none] (30) at (-8.5, -1) {$1u$};
		\node [style=none] (24) at (-6, 2) {};
		\node [style=none] (25) at (-5.75, 0) {};
		\node [style=none] (26) at (-2.75, 0.5) {$m$};
		\node [style=none] (27) at (-7.25, 1.25) {$1F$};
		\node [style=none] (28) at (-4.75, -2.25) {$1u$};
		\node [style=none] (29) at (-4.75, -0.25) {$FF$};

		\draw [in=180, out=0] (5.center) to (6.center);
		\draw [in=120, out=0] (1.center) to (8);
		\draw [in=0, out=-120] (8) to (3.center);
		\draw (12.center) to (18.center);
		\draw [in=-180, out=60] (18.center) to (20.center);
		\draw [in=180, out=0] (14.center) to (21.center);
		\draw [in=180, out=-60, looseness=0.75] (18.center) to (22.center);
		\draw (10.center) to (24.center);
		\draw [in=105, out=0] (24.center) to (23);
		\draw [in=0, out=-120, looseness=0.75] (23) to (22.center);
		\draw [in=240, out=0] (21.center) to (25.center);
		\draw [in=180, out=0] (25.center) to (15.center);
		\draw [in=105, out=0] (20.center) to (25.center);
\end{tikzpicture}

\end{center}
\end{minipage}
\end{defn}

\begin{rmk}
A non-trivial example of $\cat V$-pseudofunctor will be provided at Section \ref{subsect hom pseudo}.
\end{rmk}

\subsubsection{$\cat V$-pseudonatural transformations}\label{subsect vpsnat}

Let $\cat D$ be a $\cat V$-bicategory, for $\cat V$ a monoidal bicategory. We are going to introduce some 2-cells that will be needed for Definition \ref{def Vpsnat}. First, for every morphism $f\colon \mathbb 1\to \cat D(x,y)$ in $\cat V$, we are going to use notations $f^\ast$ and $f_\ast$ to denote the composites
\[\cat D(y,z)\overset{1f}\longrightarrow\cat D(y,z)\cat D(x,y)\overset{m}\longrightarrow\cat D(x,z)\hspace{2em}\text{and}\hspace{2em}\cat D(z,x)\overset{f1}\longrightarrow\cat D(x,y)\cat D(z,x)\overset{m}\longrightarrow\cat D(z,y)\]
respectively. Then, there are four 2-isomorphisms exhibiting pseudonaturality in each variable of unit and multiplication:

\begin{minipage}{.05\textwidth}
\begin{equation}\label{pseudonaturality of m}
\phantom{c}
\end{equation}
\end{minipage}
\begin{minipage}{.95\textwidth}
\begin{center}
\begin{tikzpicture}[scale=.5]
		\node [style=none] (0) at (-3, -1.5) {};
		\node [style=none] (1) at (-3, 1.5) {};
		\node [style=squared] (2) at (0, 0) {$u_f$};a
		\node [style=none] (3) at (3.25, 1.5) {};
		\node [style=none] (4) at (3.25, -1.5) {};
		\node [style=none] (5) at (3.75, -1.5) {$u$};
		\node [style=none] (6) at (3.75, 1.5) {$f_\ast$};
		\node [style=none] (7) at (-3.5, -1.5) {$u$};
		\node [style=none] (8) at (-3.5, 1.5) {$f^\ast$};
		\node [style=none] (9) at (0, -2) {$\mathbb 1$};
		\node [style=none] (10) at (-2.5, 0) {$\mathcal D(y,y)$};
		\node [style=none] (11) at (2.5, 0) {$\mathcal D(x,x)$};
		\node [style=none] (12) at (0, 2) {$\mathcal D(x,y)$};

		\draw [in=120, out=0] (1.center) to (2);
		\draw [in=-180, out=60] (2) to (3.center);
		\draw [in=-60, out=180] (4.center) to (2);
		\draw [in=0, out=-120] (2) to (0.center);
\end{tikzpicture}
\hspace{1em}\raisebox{3em}{,}\hspace{2em}
\begin{tikzpicture}[scale=.5]
		\node [style=none] (0) at (-3.25, -1.75) {};
		\node [style=none] (1) at (-3.25, 1.75) {};
		\node [style=squared] (2) at (0, 0) {$m_f$};
		\node [style=none] (3) at (3.75, 1.75) {};
		\node [style=none] (4) at (3.75, -1.75) {};
		\node [style=none] (5) at (4.25, -1.75) {$f^\ast1$};
		\node [style=none] (6) at (4.25, 1.75) {$m$};
		\node [style=none] (7) at (-3.75, -1.75) {$1f_\ast$};
		\node [style=none] (8) at (-3.75, 1.75) {$m$};
		\node [style=none] (9) at (0, -2.25) {$\mathcal D(y,z)\mathcal D(w,x)$};
		\node [style=none] (10) at (-3.25, 0) {$\mathcal D(y,z)\mathcal D(w,y)$};
		\node [style=none] (11) at (3.25, 0) {$\mathcal D(x,z)\mathcal D(w,x)$};
		\node [style=none] (12) at (0, 2.25) {$\mathcal D(w,z)$};

		\draw [in=120, out=0] (1.center) to (2);
		\draw [in=-180, out=60] (2) to (3.center);
		\draw [in=-60, out=180] (4.center) to (2);
		\draw [in=0, out=-120] (2) to (0.center);
\end{tikzpicture}
\end{center}
\end{minipage}
along with

\begin{minipage}{.05\textwidth}
\begin{equation}\label{m crossing}
\phantom{c}
\end{equation}
\end{minipage}
\begin{minipage}{.95\textwidth}
\begin{center}
\begin{tikzpicture}[scale=.5]
		\node [style=none] (0) at (-2, -1.25) {};
		\node [style=none] (1) at (-2, 1.25) {};
		\node [style=none] (2) at (-2.5, -1.25) {$1f^\ast$};
		\node [style=none] (3) at (-2.5, 1.25) {$m$};
		\node [style=none] (4) at (2.5, -1.25) {$m$};
		\node [style=none] (5) at (2.5, 1.25) {$f^\ast$};
		\node [style=none] (6) at (2, -1.25) {};
		\node [style=none] (7) at (2, 1.25) {};

		\draw [in=180, out=0] (1.center) to (6.center);
		\draw [in=-180, out=0] (0.center) to (7.center);
\end{tikzpicture}
\hspace{1em}\raisebox{2em}{,}\hspace{2em}
\begin{tikzpicture}[scale=.5]
		\node [style=none] (0) at (-2, -1.25) {};
		\node [style=none] (1) at (-2, 1.25) {};
		\node [style=none] (2) at (-2.5, -1.25) {$m$};
		\node [style=none] (3) at (-2.5, 1.25) {$f_\ast$};
		\node [style=none] (4) at (2.5, -1.25) {$f_\ast 1$};
		\node [style=none] (5) at (2.5, 1.25) {$m$};
		\node [style=none] (6) at (2, -1.25) {};
		\node [style=none] (7) at (2, 1.25) {};

		\draw [in=180, out=0] (1.center) to (6.center);
		\draw [in=-180, out=0] (0.center) to (7.center);
\end{tikzpicture}
\end{center}
\end{minipage}
These are defined by
\begin{center}
\begin{tikzpicture}[scale=.5]
		\node [style=none] (0) at (-3.75, 1.25) {};
		\node [style=none] (1) at (-3.75, 0) {};
		\node [style=none] (2) at (-3.75, -1.25) {};
		\node [style=none] (3) at (-4.5, 1.25) {$m$};
		\node [style=none] (4) at (-4.5, 0) {$1f$};
		\node [style=none] (5) at (-4.5, -1.25) {$u$};
		\node [style=none] (6) at (-2.75, -1.25) {};
		\node [style=fixed size node] (7) at (-0.75, 0.5) {$\lambda$};
		\node [style=fixed size node] (8) at (0.75, 0.5) {$\rho$};
		\node [style=none] (9) at (0, -1.25) {};
		\node [style=none] (10) at (3.75, 1.25) {};
		\node [style=none] (11) at (3.75, 0) {};
		\node [style=none] (12) at (3.75, -1.25) {};
		\node [style=none] (13) at (4.5, 1.25) {$m$};
		\node [style=none] (14) at (4.5, 0) {$f1$};
		\node [style=none] (15) at (4.5, -1.25) {$u$};
		\node [style=none] (16) at (2.75, -1.25) {};
		\node [style=none] (17) at (-1.25, -1.5) {$1f$};
		\node [style=none] (18) at (-2.75, -1.75) {$u1$};
		\node [style=none] (19) at (-1.75, 0) {$u1$};
		\node [style=none] (20) at (0, -0.75) {$f$};
		\node [style=none] (21) at (1.25, -1.5) {$f1$};
		\node [style=none] (22) at (2.75, -1.75) {$1u$};
		\node [style=none] (23) at (1.75, 0) {$1u$};
		\node [style=none] (24) at (-6.5, 0) {$u_f =$};
		
		\draw [in=180, out=0] (2.center) to (6.center);
		\draw [in=-120, out=0, looseness=0.75] (6.center) to (7);
		\draw [in=0, out=135] (7) to (0.center);
		\draw [in=-180, out=0] (1.center) to (9.center);
		\draw [in=0, out=-180] (11.center) to (9.center);
		\draw [in=360, out=180] (12.center) to (16.center);
		\draw [in=-75, out=180, looseness=0.75] (16.center) to (8);
		\draw [in=180, out=60] (8) to (10.center);
\end{tikzpicture}
\hspace{1em}\raisebox{3em}{,}\hspace{2em}
\begin{tikzpicture}[scale=.5]
		\node [style=none] (0) at (-1.75, 1.25) {};
		\node [style=none] (1) at (-1.75, 0) {};
		\node [style=none] (2) at (-1.75, -1.25) {};
		\node [style=none] (3) at (-2.5, 1.25) {$m$};
		\node [style=none] (4) at (-2.5, 0) {$1m$};
		\node [style=none] (5) at (-2.5, -1.25) {$1f1$};
		\node [style=none] (6) at (2.5, 1.25) {$m$};
		\node [style=none] (7) at (2.5, 0) {$m1$};
		\node [style=none] (8) at (2.5, -1.25) {$1f1$};
		\node [style=none] (9) at (1.75, 1.25) {};
		\node [style=none] (10) at (1.75, 0) {};
		\node [style=none] (11) at (1.75, -1.25) {};
		\node [style=fixed size node] (12) at (0, 0.5) {$\alpha$};
		\node [style=none] (13) at (-4.5, 0) {$m_f =$};

		\draw (2.center) to (11.center);
		\draw [in=120, out=0] (0.center) to (12);
		\draw [in=180, out=-60] (12) to (10.center);
		\draw [in=60, out=-180] (9.center) to (12);
		\draw [in=0, out=-120] (12) to (1.center);
\end{tikzpicture}

\end{center}
while the crossings \eqref{m crossing} are respectively given by the two 2-cells
\begin{center}
\begin{tikzpicture}[scale=.5]
		\node [style=none] (2) at (3.5, 0) {$1f$};
		\node [style=none] (3) at (2.75, 0) {};
		\node [style=none] (4) at (3.5, 1.25) {$m$};
		\node [style=none] (5) at (2.75, 1.25) {};
		\node [style=none] (6) at (3.5, -1.25) {$m$};
		\node [style=none] (7) at (1.75, -1.25) {};
		\node [style=none] (8) at (-1, -1.25) {};
		\node [style=none] (10) at (-1, -1.25) {};
		\node [style=none] (11) at (-2.75, -1.25) {};
		\node [style=none] (12) at (-2.75, 0) {};
		\node [style=fixed size node] (14) at (-1, 0.75) {$\alpha$};
		\node [style=none] (15) at (-2.75, 1.25) {};
		\node [style=none] (16) at (-3.5, 1.25) {$m$};
		\node [style=none] (17) at (-3.5, -1.25) {$11f$};
		\node [style=none] (19) at (2.75, -1.25) {};
		\node [style=none] (20) at (-3.5, 0) {$1m$};
		\node [style=none] (21) at (1.75, -1.75) {$m1$};
		\node [style=none] (22) at (0.25, 0.25) {$m1$};

		\draw [in=0, out=180] (3.center) to (8.center);
		\draw [in=0, out=180] (10.center) to (11.center);
		\draw [in=150, out=0] (15.center) to (14);
		\draw [in=-180, out=30] (14) to (5.center);
		\draw [in=-45, out=180, looseness=0.75] (7.center) to (14);
		\draw [in=0, out=-135, looseness=0.75] (14) to (12.center);
		\draw [in=0, out=180] (19.center) to (7.center);
\end{tikzpicture}
\hspace{2em}\raisebox{3em}{and}\hspace{2em}
\begin{tikzpicture}[scale=.5]
		\node [style=none] (2) at (-3.5, 0) {$f1$};
		\node [style=none] (3) at (-2.75, 0) {};
		\node [style=none] (4) at (-3.5, 1.25) {$m$};
		\node [style=none] (5) at (-2.75, 1.25) {};
		\node [style=none] (6) at (-3.5, -1.25) {$m$};
		\node [style=none] (7) at (-1.75, -1.25) {};
		\node [style=none] (8) at (1, -1.25) {};
		\node [style=none] (10) at (1, -1.25) {};
		\node [style=none] (11) at (2.75, -1.25) {};
		\node [style=none] (12) at (2.75, 0) {};
		\node [style=fixed size node] (14) at (1, 0.75) {$\alpha$};
		\node [style=none] (15) at (2.75, 1.25) {};
		\node [style=none] (16) at (3.5, 1.25) {$m$};
		\node [style=none] (17) at (3.5, -1.25) {$f11$};
		\node [style=none] (19) at (-2.75, -1.25) {};
		\node [style=none] (20) at (3.5, 0) {$m1$};
		\node [style=none] (21) at (-1.75, -1.75) {$1m$};
		\node [style=none] (22) at (-0.25, 0.25) {$1m$};

		\draw [in=180, out=0] (3.center) to (8.center);
		\draw [in=-180, out=0] (10.center) to (11.center);
		\draw [in=-330, out=180] (15.center) to (14);
		\draw [in=360, out=150] (14) to (5.center);
		\draw [in=225, out=0, looseness=0.75] (7.center) to (14);
		\draw [in=-180, out=-45, looseness=0.75] (14) to (12.center);
		\draw [in=180, out=0] (19.center) to (7.center);
\end{tikzpicture}
\end{center}

\begin{defn}\label{def Vpsnat} A $\cat V$\emph{-pseudonatural transformation} $t\colon F\Rightarrow G$ between $\cat V$-pseudofunctors $F,G\colon\cat C\to\cat D$ consists of, for every pair of objects $a,b$ in $\cat B$
\begin{itemize}
\item 1-cells in $\cat V$
\[t_a\colon \mathbb 1\longrightarrow\cat D(Fa,Ga)\]
\item 2-cells in $\cat V$
\begin{center}
\begin{tikzpicture}[scale=.5]
		\node [style=none] (0) at (-3, -1.5) {};
		\node [style=none] (1) at (-3, 1.5) {};
		\node [style=fixed size node] (2) at (0, 0) {$t_{ab}$};
		\node [style=none] (3) at (3.25, 1.5) {};
		\node [style=none] (4) at (3.25, -1.5) {};
		\node [style=none] (5) at (3.75, -1.5) {$F$};
		\node [style=none] (6) at (3.75, 1.5) {$t_\ast$};
		\node [style=none] (7) at (-3.5, -1.5) {$G$};
		\node [style=none] (8) at (-3.5, 1.5) {$t^\ast$};
		\node [style=none] (9) at (0, -2) {{$\mathcal C(a,b)$}};
		\node [style=none] (10) at (-2.5, 0) {$\mathcal D(Ga,Gb)$};
		\node [style=none] (11) at (2.5, 0) {$\mathcal D(Fa,Fb)$};
		\node [style=none] (12) at (0, 2) {$\mathcal D(Fa,Gb)$};

		\draw [in=120, out=0] (1.center) to (2);
		\draw [in=-180, out=60] (2) to (3.center);
		\draw [in=-60, out=180] (4.center) to (2);
		\draw [in=0, out=-120] (2) to (0.center);
\end{tikzpicture}
\end{center}
\end{itemize}
subject to the following axioms of \emph{unitality} and \emph{functoriality}:
\vspace{1em}

\noindent
\begin{minipage}{.1\textwidth}
\begin{equation}\label{psnat unitality}
\phantom{c}
\tag{TU}
\end{equation}
\end{minipage}
\begin{minipage}{.9\textwidth}
\begin{center}
\begin{tikzpicture}[scale=.5]
		\node [style=none] (0) at (-7.5, 1.5) {};
		\node [style=none] (1) at (-7.5, -1.25) {};
		\node [style=none] (2) at (-6.5, -1.25) {};
		\node [style=none] (3) at (-2, 1.5) {};
		\node [style=none] (4) at (-3, -1.25) {};
		\node [style=none] (5) at (-2, -1.25) {};
		\node [style=fixed size node] (6) at (-4.75, 0.25) {$t_{aa}$};
		\node [style=none] (7) at (-4.75, -2) {};
		\node [style=none] (8) at (-1.5, 1.5) {$t_\ast$};
		\node [style=none] (9) at (-1.5, -1.25) {$u$};
		\node [style=none] (10) at (-8, 1.5) {$t^\ast$};
		\node [style=none] (11) at (-8, -1.25) {$u$};
		\node [style=none] (12) at (0, 0) {$=$};
		\node [style=none] (13) at (2, 1.5) {};
		\node [style=none] (14) at (2, -1.25) {};
		\node [style=none] (16) at (6.5, 1.5) {};
		\node [style=none] (18) at (6.5, -1.25) {};
		\node [style=none] (21) at (7, 1.5) {$t_\ast$};
		\node [style=none] (22) at (7, -1.25) {$u$};
		\node [style=none] (23) at (1.5, 1.5) {$t^\ast$};
		\node [style=none] (24) at (1.5, -1.25) {$u$};
		\node [style=squared] (25) at (4.25, 0) {$u_t$};
		\node [style=none] (26) at (-6.25, -0.5) {$G$};
		\node [style=none] (27) at (-3.25, -0.5) {$F$};
		\node [style=none] (28) at (-4.75, -2.25) {$u$};

		\draw [in=105, out=0, looseness=0.75] (0.center) to (6);
		\draw [in=360, out=180] (2.center) to (1.center);
		\draw [in=-120, out=75] (2.center) to (6);
		\draw [in=-75, out=180] (7.center) to (2.center);
		\draw [in=-120, out=0] (7.center) to (4.center);
		\draw [in=180, out=0] (4.center) to (5.center);
		\draw [in=105, out=-60] (6) to (4.center);
		\draw [in=-180, out=75, looseness=0.75] (6) to (3.center);
		\draw [in=120, out=0] (13.center) to (25);
		\draw [in=0, out=-120] (25) to (14.center);
		\draw [in=-180, out=-60] (25) to (18.center);
		\draw [in=-180, out=60] (25) to (16.center);
\end{tikzpicture}
\end{center}
\end{minipage}

\noindent
\begin{minipage}{.1\textwidth}
\begin{equation}\label{psnat functoriality}
\phantom{c}
\tag{TF}
\end{equation}
\end{minipage}
\begin{minipage}{.9\textwidth}
\begin{center}
\begin{tikzpicture}[scale=.5]
		\node [style=none] (0) at (0, 0) {$=$};
		\node [style=none] (1) at (-11.5, 0) {$m$};
		\node [style=none] (2) at (-11.5, 2) {$t^\ast$};
		\node [style=none] (3) at (-11.5, -2) {$GG$};
		\node [style=none] (4) at (-11, 2) {};
		\node [style=none] (5) at (-11, 0) {};
		\node [style=none] (6) at (-11, -2) {};
		\node [style=none] (10) at (1.5, 2) {$t^\ast$};
		\node [style=none] (11) at (1.5, 0) {$m$};
		\node [style=none] (12) at (1.5, -2) {$GG$};
		\node [style=none] (13) at (2, 2) {};
		\node [style=none] (14) at (2, 0) {};
		\node [style=none] (15) at (2, -2) {};
		\node [style=none] (16) at (-2, 2) {};
		\node [style=none] (17) at (-2, 0) {};
		\node [style=none] (18) at (-2, -2) {};
		\node [style=none] (19) at (11.5, 2) {$t_\ast$};
		\node [style=none] (20) at (11.5, 0) {$m$};
		\node [style=none] (21) at (11.5, -2) {$FF$};
		\node [style=none] (22) at (11, 2) {};
		\node [style=none] (23) at (11, 0) {};
		\node [style=none] (24) at (11, -2) {};
		\node [style=none] (25) at (-6.5, -2) {};
		\node [style=fixed size node] (26) at (-6.5, 1) {$t_{ac}$};
		\node [style=fixed size node] (27) at (4.5, -1) {$1t_{ab}$};
		\node [style=fixed size node] (28) at (8.5, -1) {$t_{bc}1$};
		\node [style=squared] (29) at (6.5, 1) {$m_t$};
		\node [style=none] (30) at (5.25, 2) {};
		\node [style=none] (31) at (7.75, 2) {};
		\node [style=none] (32) at (4.25, 0.25) {$1t^\ast$};
		\node [style=none] (33) at (5.25, 2.25) {$m$};
		\node [style=none] (34) at (7.75, 2.25) {$m$};
		\node [style=none] (35) at (5.75, -0.25) {$1t_\ast$};
		\node [style=none] (36) at (7.25, -0.25) {$t^\ast 1$};
		\node [style=none] (37) at (8.75, 0.25) {$t_\ast 1$};
		\node [style=none] (38) at (6.5, -2.25) {$GF$};
		\node [style=none] (39) at (-1.5, 2) {$t_\ast$};
		\node [style=none] (40) at (-1.5, 0) {$m$};
		\node [style=none] (41) at (-1.5, -2) {$FF$};
		\node [style=none] (42) at (-8, -0.25) {$G$};
		\node [style=none] (43) at (-5, -0.25) {$F$};

		\draw [in=135, out=0, looseness=0.75] (4.center) to (26);
		\draw [in=-180, out=0] (5.center) to (25.center);
		\draw [in=180, out=0] (25.center) to (17.center);
		\draw [in=-60, out=-180, looseness=0.75] (18.center) to (26);
		\draw [in=0, out=-120, looseness=0.75] (26) to (6.center);
		\draw [in=-180, out=45, looseness=0.75] (26) to (16.center);
		\draw [in=120, out=0, looseness=0.75] (13.center) to (27);
		\draw [in=60, out=-120] (29) to (27);
		\draw [in=120, out=-60] (29) to (28);
		\draw [in=60, out=-180, looseness=0.75] (22.center) to (28);
		\draw [in=180, out=-60] (28) to (24.center);
		\draw [in=255, out=-75] (27) to (28);
		\draw [in=0, out=-120] (27) to (15.center);
		\draw [in=-180, out=0] (14.center) to (30.center);
		\draw [in=120, out=0] (30.center) to (29);
		\draw [in=-180, out=60] (29) to (31.center);
		\draw [in=-180, out=0] (31.center) to (23.center);
\end{tikzpicture}

\end{center}
\end{minipage}
\end{defn}

\subsubsection{$\cat V$-modifications}

\begin{defn}
A $\cat V$-modification $M\colon t\Rrightarrow s$ between $\cat V$-pseudonatural transformations is a family of 2-cells $M_a\colon t_a\Rightarrow s_a$ in $\cat V$ such that
\begin{center}
\begin{tikzpicture}[scale=.5]
		\node [style=none] (0) at (-8, 1.25) {$t^\ast$};
		\node [style=none] (1) at (-8, -1.25) {G};
		\node [style=none] (2) at (-7.5, 1.25) {};
		\node [style=none] (3) at (-7.5, -1.25) {};
		\node [style=none] (4) at (-2, 1.25) {};
		\node [style=none] (5) at (-2, -1.25) {};
		\node [style=none] (6) at (-1.5, 1.25) {$s_\ast$};
		\node [style=none] (7) at (-1.5, -1.25) {$F$};
		\node [style=none] (16) at (0, 0) {$=$};
		\node [style=fixed size node] (17) at (-6.25, 0) {$t_{ab}$};
		\node [style=fixed size node] (18) at (-3.25, 1.25) {${M_b}_\ast$};
		\node [style=none] (19) at (-4.5, 1.25) {};
		\node [style=none] (24) at (-4.5, -1.25) {};
		\node [style=none] (25) at (-5, 1.5) {$t_\ast$};
		\node [style=none] (26) at (8, -1.25) {$F.$};
		\node [style=none] (27) at (8, 1.25) {$s_\ast$};
		\node [style=none] (28) at (7.5, -1.25) {};
		\node [style=none] (29) at (7.5, 1.25) {};
		\node [style=none] (30) at (2, -1.25) {};
		\node [style=none] (31) at (2, 1.25) {};
		\node [style=none] (32) at (1.5, -1.25) {$G$};
		\node [style=none] (33) at (1.5, 1.25) {$t^\ast$};
		\node [style=fixed size node] (34) at (6.25, 0) {$s_{ab}$};
		\node [style=fixed size node] (35) at (3.25, 1.25) {$M_a^\ast$};
		\node [style=none] (36) at (4.5, -1.25) {};
		\node [style=none] (37) at (4.5, 1.25) {};
		\node [style=none] (38) at (5, 1.5) {$s^\ast$};

		\draw [in=360, out=180] (4.center) to (18);
		\draw [in=360, out=180] (18) to (19.center);
		\draw [in=60, out=-180] (19.center) to (17);
		\draw [in=0, out=-105] (17) to (3.center);
		\draw [in=0, out=105] (17) to (2.center);
		\draw [in=360, out=180] (5.center) to (24.center);
		\draw [in=-60, out=-180] (24.center) to (17);
		\draw [in=-120, out=0] (36.center) to (34);
		\draw [in=-180, out=75] (34) to (29.center);
		\draw [in=-180, out=-75] (34) to (28.center);
		\draw [in=120, out=0] (37.center) to (34);
		\draw (31.center) to (35);
		\draw (35) to (37.center);
		\draw (30.center) to (36.center);
\end{tikzpicture}
\end{center}
The definition of $(-)^\ast$ and $(-)_\ast$ applied to a 2-cell is simply given by tensoring the 2-cell with 1 (to the right and to the left respectively) and then whiskering with $m$.
\end{defn}

\subsection{The opposite $\cat V$-bicategory}\label{section opposite}

In this section we assume $\cat V$ to be a monoidal bicategory. Moreover, without loss of generality, we will deal with its semi-strict replacement, which is a semi-strict braided monoidal bicategory (Definition \ref{def braided gray mon}). Under this assumption, we define the structure of the opposite $\cat V$-bicategory of a $\cat V$-bicategory $\cat C$, and we prove that it defines indeed a $\cat V$-bicategory, in the sense that this new structure will itself satisfy the coherence axioms \eqref{IC Vbicat} and \eqref{AC Vbicat}.

\begin{defn}\label{def op}
    Let $\cat C$ be a $\cat V$-bicategory, for $\cat V$ a braided monoidal bicategory. Let $\cat C\op$ be the data of
    \begin{itemize}
        \item The class of objects $({\cat C\op})_0\coloneqq\cat C_0$ of $\cat C$.
        \item For every pair of objects $a,b$, the object $\cat C\op(a,b)\coloneqq\cat C(b,a)$ in $\cat V$, together with, for every triple of objects $c,d,e$
        \begin{align*}
    \underline{m}\colon\cat C\op(d,e)\otimes\cat C\op(c,d)\longrightarrow\cat C\op(c,e)
    \end{align*}
    defined as the composite
    \begin{align*}
    \cat C(e,d)\otimes\cat C(d,c)\overset{\beta}{\longrightarrow}\cat C(d,c)\otimes\cat C(e,d)\overset{m}{\longrightarrow}\cat C(e,c),
\end{align*}
and the unit \[\underline u_c=u_c\colon \mathbb 1\longrightarrow\cat C(c,c)=\cat C\op(c,c).\]
        \item  The associator
\hspace*{-2cm}
    \begin{center}

	\end{center}
and the two unitors
\begin{center}
	
%
\end{center}
where the unlabeled 2-cells are identities provided by axiom \eqref{ax1 semi-strict braided} for a semi-strict braided monoidal bicategory (Definition \ref{def braided gray mon}).
\end{itemize}
\end{defn}

\begin{thm}\label{thmCop}
    The data of Definition \ref{def op} define $\cat C\op$ as a $\cat V$-bicategory.
\end{thm}
\begin{proof}
    The proof consists of a string-diagrammatic computation showing the identity and the associativity coherence axioms. Let us start with the identity coherence for $\cat C\op$:
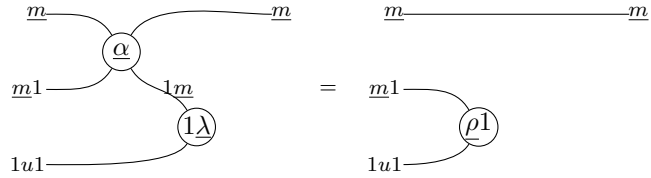
\begin{center}
%
\end{center}
The left-hand side of \ref{IC Vbicat} is given, expanding the definition of $\underline\alpha$ and of $1\lambda$, by the 2-cell
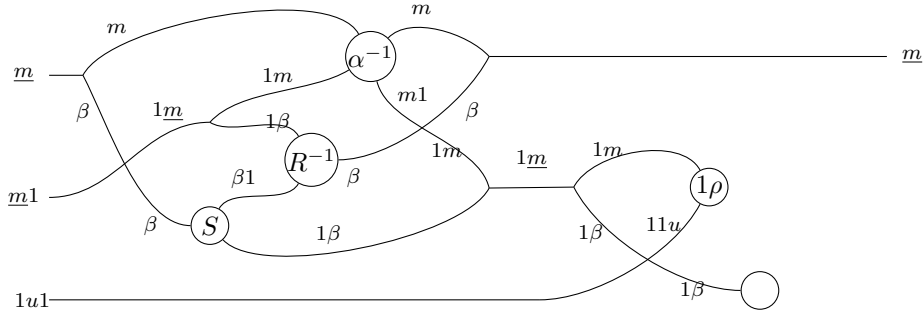
\begin{center}
%
\end{center}
Let's expand everywhere the definition of $\underline m=m\circ \beta$,
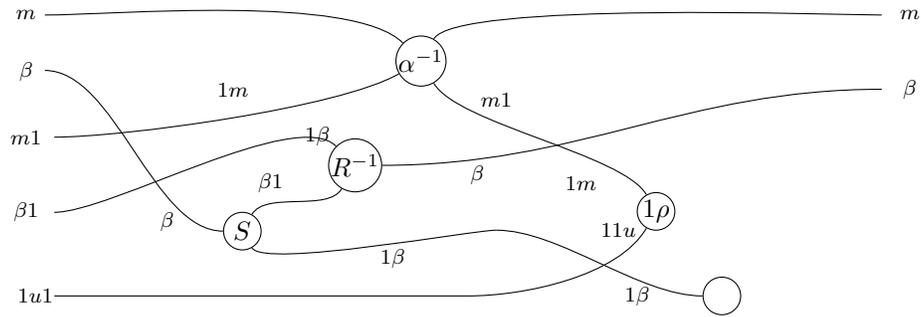
\begin{center}
%
\end{center}
Observe that we can use the naturality axiom \eqref{naturality beta} for the pseudonatural transformation $\beta$ with respect to the pair of 2-cells $(1,\rho).$ This gives
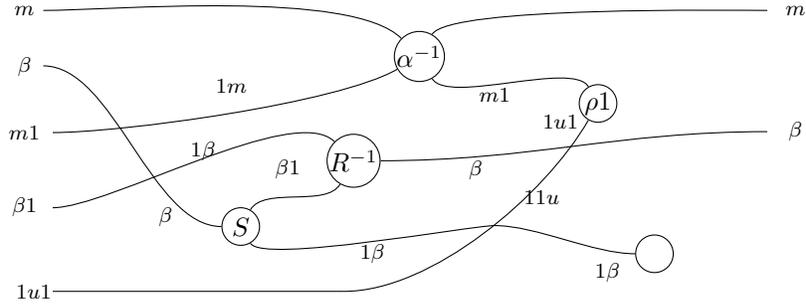
\begin{center}
%
\end{center}
Then, we are able to recognize at the top the left-hand side side of the identity coherence axiom \ref{IC Vbicat} for the $\cat V$-bicategory $\cat C$, which states
\begin{center}
%
\end{center}
Our diagram becomes then
\begin{center}
%

\end{center}
Now, thanks to the modification axiom \eqref{modification property R and S} for the composite constituted by $R^{-1}$ and $S$, we can slide the morphism $1u1$ over the latter, finding
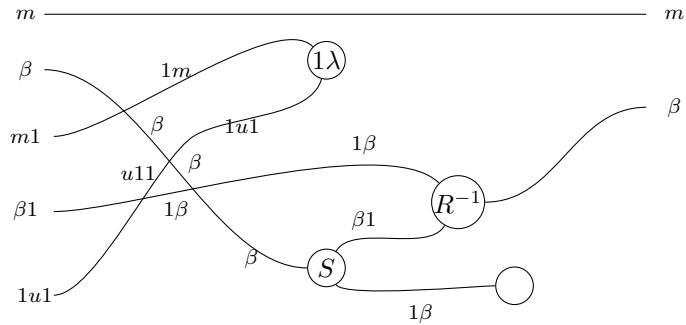
\begin{center}
%

\end{center}
Now, by naturality of $\beta$, we can slide the latter to the right of the 2-cell $1\lambda$:
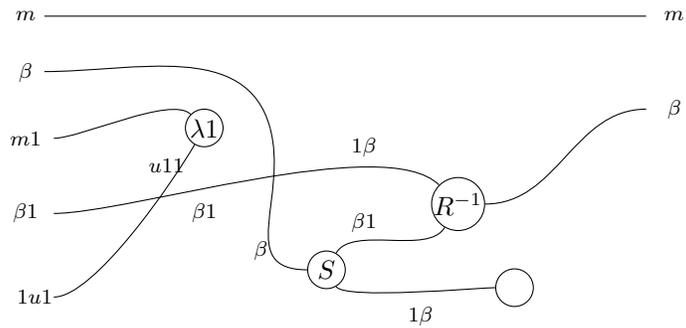
\begin{center}
%

\end{center}
We are now able to link $\beta$ and its inverse (after an adequate stretching) and find
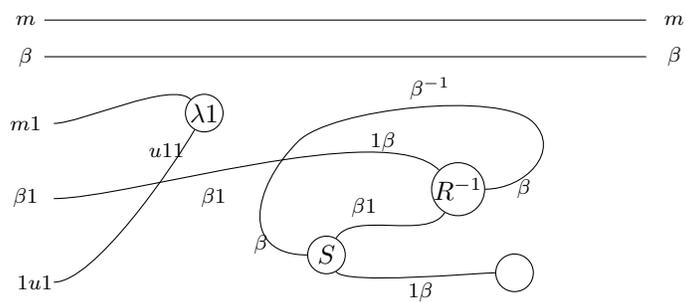
\begin{center}
%
\end{center}
The right-hand side of \eqref{IC Vbicat} is almost there. In only suffices to observe that the 2-cell

\noindent
\begin{minipage}{.1\textwidth}
\begin{equation}\label{pezzo di IC op}
\phantom{c}
\end{equation}
\end{minipage}
\begin{minipage}{.9\textwidth}
\begin{center}
%
\end{center}
\end{minipage}
is just the identity. This easily follows from axioms \eqref{ax2 semi-strict braided} and \eqref{ax3 semi-strict braided}, and concludes the proof of the identity coherence.

Let us now come to the associativity coherence \eqref{AC Vbicat}. The right-hand side looks like
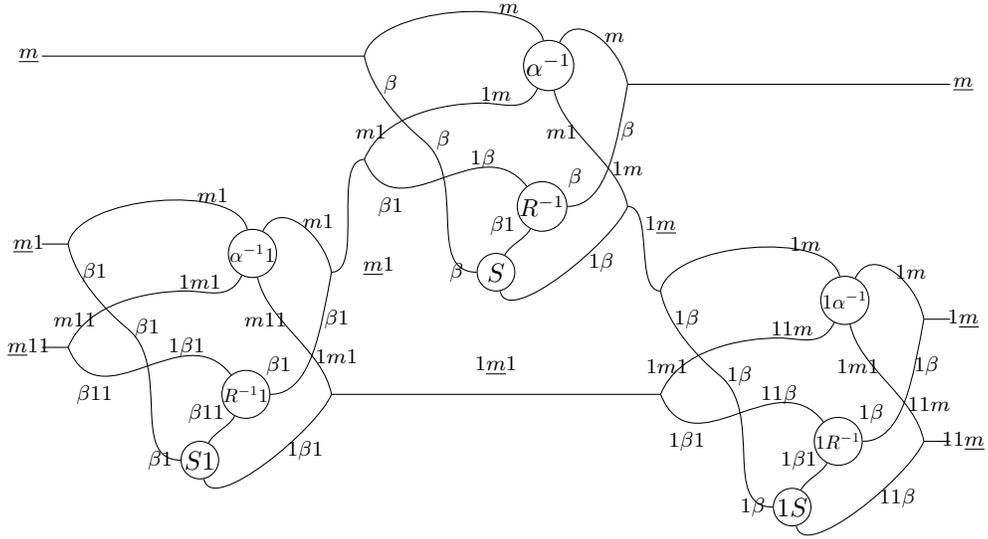
\begin{center}
%

\end{center}
First, as for \eqref{IC Vbicat}, let us decompose the morphisms $\underline m$:
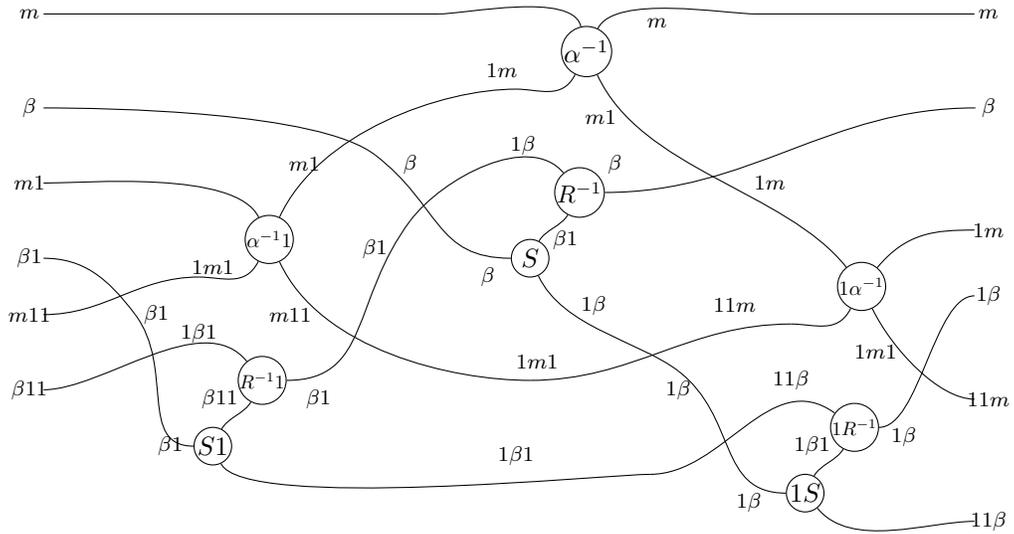
\begin{center}
%

\end{center}
The next step consists in sliding down the 1-morphism components of $\beta$ which are on top of the two cells $\alpha^{-1}1$ and $1\alpha^{-1}$, by the naturality axiom \eqref{naturality beta} for the pseudonatural transformation $\beta$:

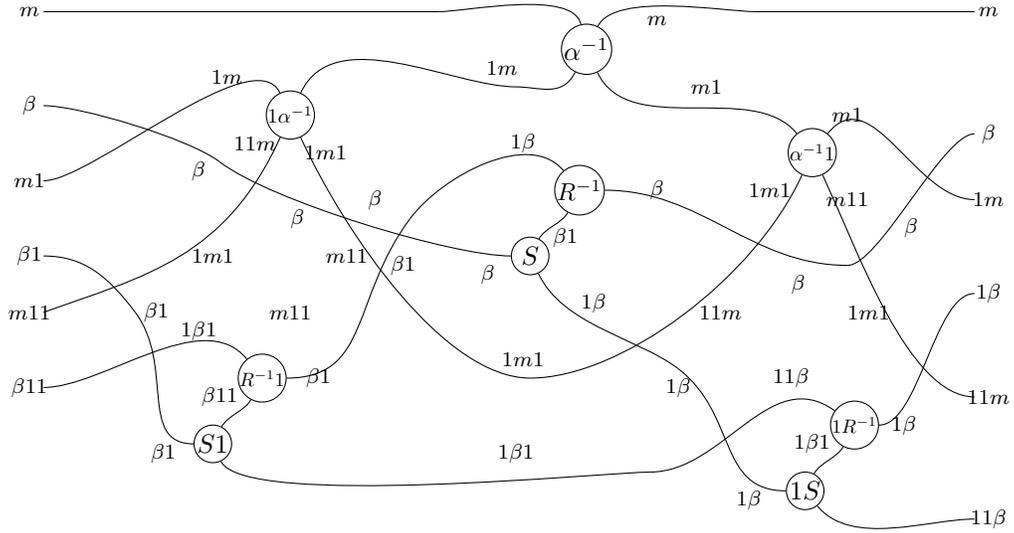
\begin{center}
%

\end{center}
Then, by the modification property \eqref{modification property R and S} of the composite 2-cell given by $S$ and $R^{-1}$, we can slide the morphism $1m1$ linking $1\alpha^{-1}$ to $\alpha^{-1}1$ above it. This allows us to recognize the axiom \eqref{AC Vbicat} for the $\cat V$-bicategory $\cat C$:
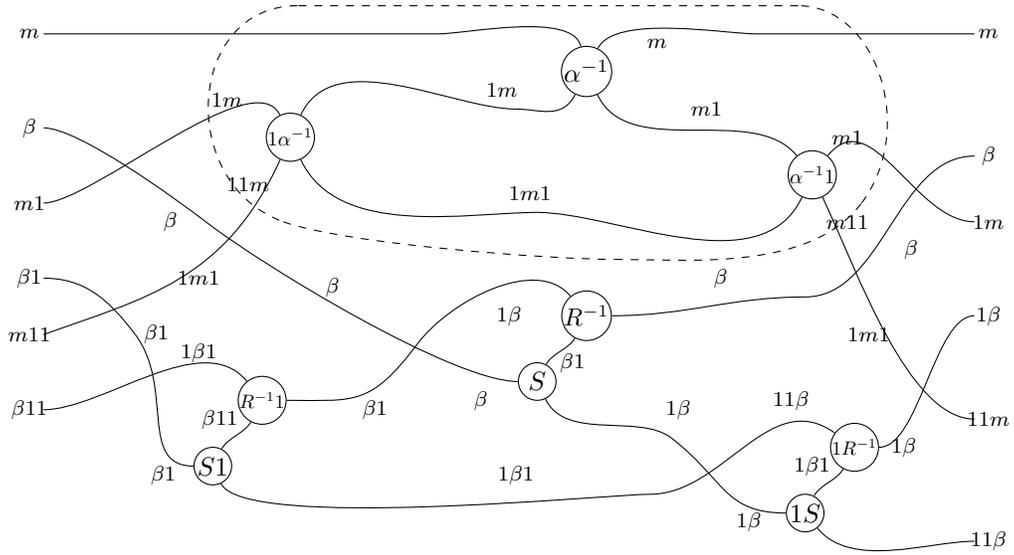
\begin{center}
%

\end{center}
Thus, by the associativity coherence for $\cat C$ we get
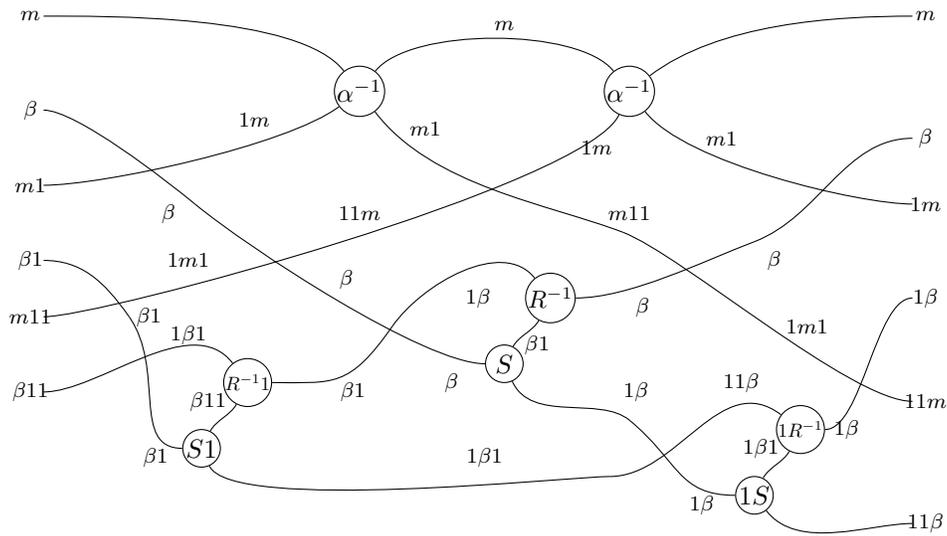
\begin{center}
%

\end{center}
Then let us slide, by pseudonaturality, the morphism $\beta$ entering in $S$ under the 2-cell composite of $R^{-1}1$ and $S1$:
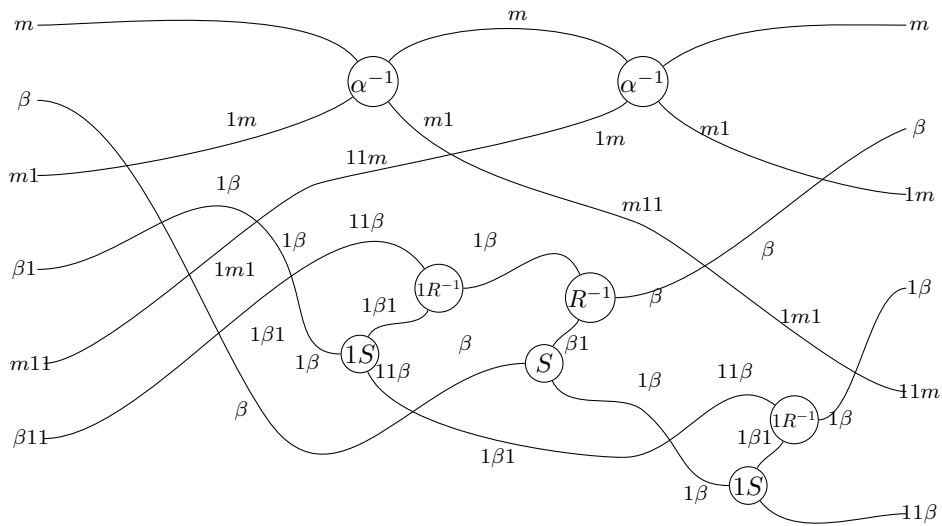
\begin{center}
%

\end{center}
Again, since $S$ is a modification, we can make it pass under the 1-morphism $1\beta 1$ below. This allows us to apply axioms \eqref{BA1} and \eqref{BA2}, in the dashed areas:
\begin{center}
%

\end{center}
Thus we apply the two braiding axioms and we get the following.

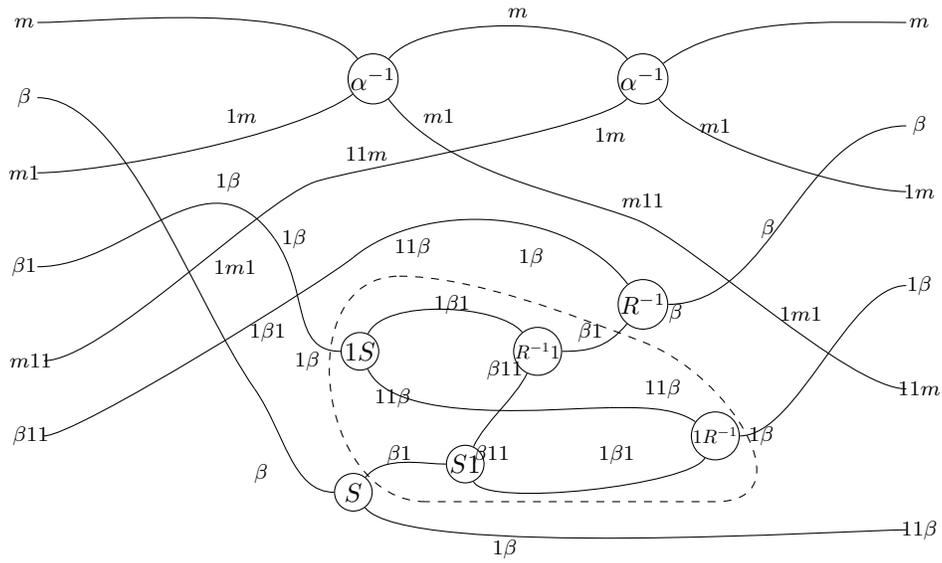
\begin{center}
%

\end{center}
Here, we see that axiom \eqref{BA4} applies to the dashed areas, and the 2-cell now looks like
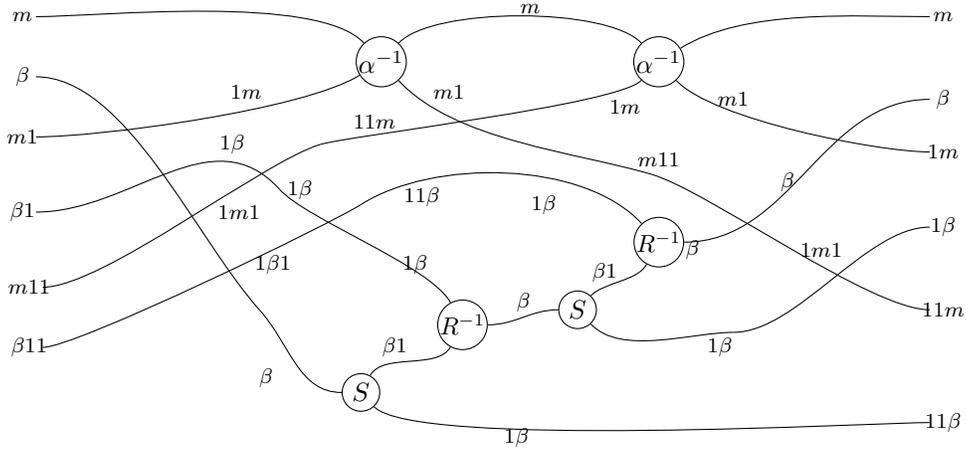
\begin{center}
%

\end{center}
Then, it suffices, thanks to a modification axiom, to slide the pair of morphisms $m11 \circ \beta 11$ under the composite of $R^{-1}$ after $S$.
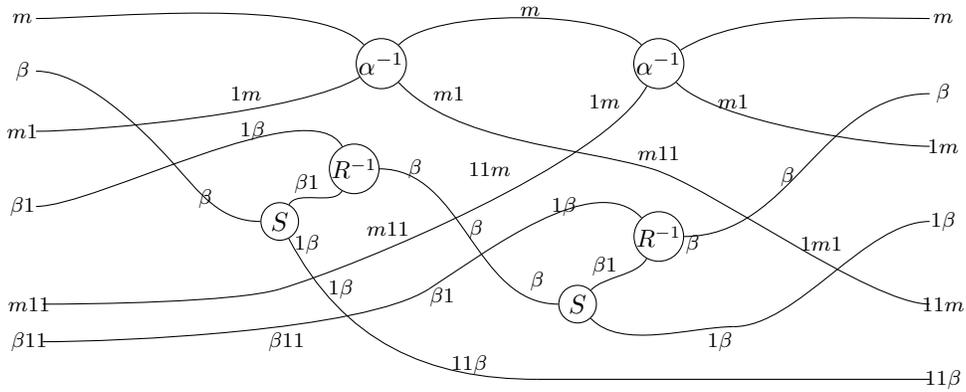
\begin{center}
%

\end{center}
Similarly, we can use the same modification axiom in order to slide the target $11m$ under the other composite $R^{-1}$ after $S$. This allows us to recognize the two cells $\underline\alpha$ forming the right-hand side of \eqref{AC Vbicat}:
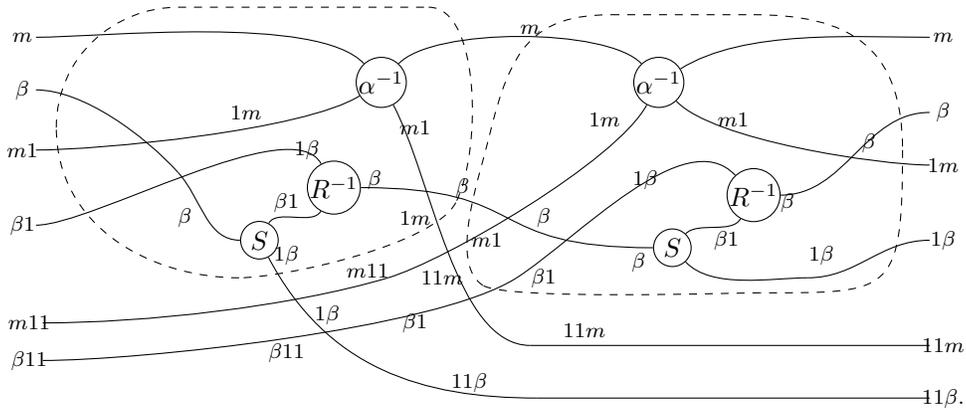
\begin{center}
%

\end{center}
This concludes the proof of the theorem.
\end{proof}

\begin{rmk}\label{oponfunctors}
If $F\colon \cat C\to\cat D$ is a $\cat V$-pseudofunctor, for $\cat V$ a braided monoidal bicategory, then the opposite construction is functorial in the sense that we can induce a $\cat V$-pseudofunctor $F\op\colon\cat C\to \cat D$, defined by the same assignment on objects, while defining each hom-morphism just as $F\op_{a,b}\coloneqq F_{b,a}\colon \cat C\op(a,b)\to\cat D\op(Fa,Fb)$. The higher structure of $\cat V$-pseudofunctor is also trivially deduced from the one of $F$, as well as the $\cat V$-pseudofunctor axioms.
\end{rmk}

\subsection{The tensor product of $\cat V$-bicategories}\label{section tensor product}

The second construction that we present is this article is that of the tensor product of $\cat V$-bicategories. The definition of the structure is subsequently given.

\begin{defn}\label{def tensor product}
Le $\cat{V}$ be a braided monoidal bicategory (which, as before, is assumed to be semi-strict without loss of generality) and $\cat B,\cat C$ be two $\cat V$-bicategories. Their product $\cat B\otimes\cat C$ is defined by the following data
\begin{itemize}
\item As objects the class product $\cat B_0\times\cat C_0$.
\item For every pair of objects $(b,c),(b',c')$, the hom-object $\cat B\otimes\cat C((b,c),(b',c'))$ is given by the monoidal product $\cat B(b,b')\otimes_\cat V\cat C(c,c')$. The composition is
\begin{center}
\begin{tikzcd}
	{\underline{m}\colon\mathcal B(b',b'')\mathcal C(c',c'')\mathcal B(b,b')\mathcal C(c,c')} & {\mathcal B(b',b'')\mathcal B(b,b')\mathcal C(c',c'')\mathcal C(c,c')} & {\mathcal B(b,b'')\mathcal C(c,c'')}
	\arrow["{1\beta 1}", from=1-1, to=1-2]
	\arrow["mm", from=1-2, to=1-3]
\end{tikzcd}
\end{center}
and the unit of an object $(b,c)$ is
\end{itemize}
\[\underline u\colon \mathbb 1=\mathbb{1}\otimes \mathbb 1\overset{u\otimes u}{\longrightarrow}\cat B(b,b)\otimes\cat C(c,c).\]
\begin{itemize}
\item The associator and the unitors are given by the following string diagrams:
\end{itemize} 
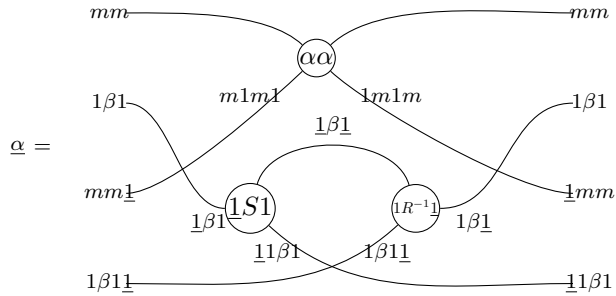
\begin{center}
\begin{tikzpicture}[scale=.4, xscale=1.1]
		\node [none] (0) at (-6.225, 4.5) {$mm$};
		\node [none] (1) at (-6.225, 1.5) {$1\beta 1$};
		\node [none] (2) at (-6.225, -1.5) {$mm\underline1$};
		\node [none] (3) at (-6.225, -4.5) {$1\beta1\underline 1$};
		\node [none] (4) at (-5.725, 4.5) {};
		\node [none] (5) at (-5.725, 1.5) {};
		\node [none] (6) at (-5.725, -1.5) {};
		\node [none] (7) at (-5.725, -4.5) {};
		\node [fixed size node] (8) at (0.025, 3) {$\alpha\alpha$};
		\node [fixed size node] (9) at (-1.975, -2) {\adjustbox{valign=m, max width=.6cm, max height=.6cm}{$\underline 1S1$}};
		\node [fixed size node] (10) at (3.025, -2) {\adjustbox{valign=m, max width=.6cm, max height=.6cm}{$1R^{-1}\underline 1$}};
		\node [none] (11) at (8.275, -4.5) {$\underline 11\beta 1$};
		\node [none] (12) at (7.775, -4.5) {};
		\node [none] (13) at (8.275, 1.5) {$1\beta 1$};
		\node [none] (14) at (8.275, 4.5) {$mm$};
		\node [none] (15) at (7.775, 1.5) {};
		\node [none] (16) at (7.775, 4.5) {};
		\node [none] (17) at (7.775, -1.5) {};
		\node [none] (18) at (8.275, -1.5) {$\underline 1mm$};
		\node [none] (19) at (0.525, 0.75) {$\underline 1\beta \underline 1$};
		\node [none] (20) at (2.15, -3.5) {$1\beta 1\underline 1$};
		\node [none] (21) at (-1.15, -3.5) {$\underline 1 1\beta 1$};
		\node [none] (22) at (2.275, 1.75) {$1m1m$};
		\node [none] (23) at (-3.225, -2.5) {$\underline 1\beta 1$};
		\node [none] (24) at (-1.975, 1.75) {$m1m1$};
		\node [none] (25) at (4.775, -2.5) {$1\beta \underline 1$};
		\node [none] (26) at (-9, 0) {$\underline \alpha$};
		\node [none] (27) at (-8.25, 0) {$=$};
		
		\draw [in=135, out=0, looseness=0.75] (4.center) to (8);
		\draw [in=-135, out=0, looseness=0.50] (6.center) to (8);
		\draw [in=0, out=180, looseness=0.75] (9) to (5.center);
		\draw [in=-135, out=0] (7.center) to (10);
		\draw [in=-45, out=180] (12.center) to (9);
		\draw [in=105, out=75] (9) to (10);
		\draw [in=180, out=0] (10) to (15.center);
		\draw [in=-45, out=180, looseness=0.50] (17.center) to (8);
		\draw [in=180, out=45, looseness=0.75] (8) to (16.center);
\end{tikzpicture}
\end{center}
from $\underline m\circ \underline m \underline 1$ to $\underline m\circ \underline 1\underline m$, and
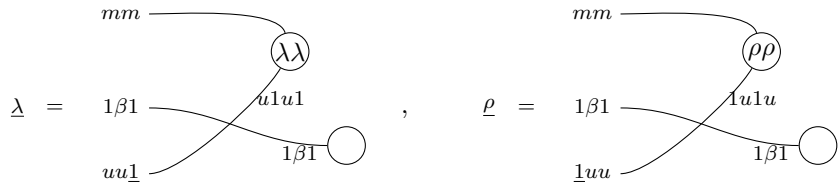
\begin{center}
\begin{tikzpicture}[scale=.5]
		\node [style=none] (2) at (-3, 2.5) {$mm$};
		\node [style=none] (3) at (-3, 0) {$1\beta 1$};
		\node [style=none] (4) at (-3, -1.75) { $uu\underline 1$};
		\node [style=none] (7) at (-2.25, 2.5) {};
		\node [style=none] (8) at (-2.25, 0) {};
		\node [style=none] (9) at (-2.25, -1.75) {};
		\node [style=fixed size node] (17) at (1.5, 1.5) {$\lambda\lambda$};
		\node [style=none] (34) at (1.25, 0.25) {$u1u1$};
		\node [style=fixed size node] (35) at (3, -1) {};
		\node [style=none] (36) at (1.75, -1.25) {$1\beta 1$};
		\node [style=none] (37) at (-5.75, 0) {$\underline \lambda$};
		\node [style=none] (38) at (-4.75, 0) {$=$};
		
		\draw [in=-120, out=0, looseness=0.50] (9.center) to (17);
		\draw [in=105, out=0, looseness=0.50] (7.center) to (17);
		\draw [in=-180, out=0] (8.center) to (35);
\end{tikzpicture}
\hspace{1em}\raisebox{3em}{\text{,}}\hspace{2em}
\begin{tikzpicture}[scale=.5]
		\node [style=none] (2) at (-3, 2.5) {$mm$};
		\node [style=none] (3) at (-3, 0) {$1\beta 1$};
		\node [style=none] (4) at (-3, -1.75) { $\underline 1uu$};
		\node [style=none] (7) at (-2.25, 2.5) {};
		\node [style=none] (8) at (-2.25, 0) {};
		\node [style=none] (9) at (-2.25, -1.75) {};
		\node [style=fixed size node] (17) at (1.5, 1.5) {$\rho\rho$};
		\node [style=none] (34) at (1.25, 0.25) {$1u1u$};
		\node [style=fixed size node] (35) at (3, -1) {};
		\node [style=none] (36) at (1.75, -1.25) {$1\beta 1$};
		\node [style=none] (37) at (-5.75, 0) {$\underline \rho$};
		\node [style=none] (38) at (-4.75, 0) {$=$};
		
		\draw [in=-120, out=0, looseness=0.50] (9.center) to (17);
		\draw [in=105, out=0, looseness=0.50] (7.center) to (17);
		\draw [in=-180, out=0] (8.center) to (35);
\end{tikzpicture}
\end{center}
\end{defn}

\begin{rmk}
In the above definition, the notation $\underline 1$ stands for the identity morphism of a hom-object in $\cat B\otimes\cat C$. Therefore it is nothing but a tensored pair of identities (11) for the hom-objects of $\cat B$ and $\cat C$. It is worth taking a moment to translate the strings defining the structural 2-cells for the tensor product of two $\cat V$-bicategories into pasting diagrams. This is a good opportunity to make the presence of the interchangers explicit:
\begin{itemize}
\item[$\underline{\alpha}$:] The associator is the following 2-cell, where we can immediately recognize $\alpha\alpha$, as well as the two crossings given by the structural morphisms $\beta_m$ and its inverse.
\end{itemize}
\begin{center}
\begin{tikzcd}[column sep=-1.5em]
	{\mathcal B(c,d)\mathcal C(z,w)\mathcal B(b,c)\mathcal C(y,z)\mathcal B(a,b)\mathcal C(x,y)} && {\mathcal B(c,d)\mathcal C(z,w)\mathcal B(a,c)\mathcal C(x,z)} \\
	& {\mathcal B(c,d)\mathcal C(z,w)\mathcal B(b,c)\mathcal B(a,b)\mathcal C(y,z)\mathcal C(x,y)} \\
	{\mathcal B(c,d)\mathcal B(b,c)\mathcal C(z,w)\mathcal C(y,z)\mathcal B(a,b)\mathcal C(x,y)} && {\mathcal B(c,d)\mathcal B(a,c)\mathcal C(z,w)\mathcal C(x,z)} \\
	& {\mathcal B(c,d)\mathcal B(b,c)\mathcal B(a,b)\mathcal C(z,w)\mathcal C(y,z)\mathcal C(x,y)} \\
	{\mathcal B(b,d)\mathcal C(y,w)\mathcal B(a,b)\mathcal C(x,y)} && {\mathcal B(a,d)\mathcal C(x,w)} \\
	& {\mathcal B(b,d)\mathcal B(a,b)\mathcal C(y,w)\mathcal C(x,y)}
	\arrow["{\underline 11\beta 1}", from=1-1, to=2-2]
	\arrow["{1\beta 1\underline 1}"', from=1-1, to=3-1]
	\arrow["{1\beta 1}", from=1-3, to=3-3]
	\arrow["{\underline 1mm}", from=2-2, to=1-3]
	\arrow["{1\beta \underline 1}", from=2-2, to=4-2]
	\arrow["{\cong (\ast)}"{description}, draw=none, from=3-1, to=2-2]
	\arrow["{\underline 1\beta 1}", from=3-1, to=4-2]
	\arrow["{mm\underline 1}"', from=3-1, to=5-1]
	\arrow["mm", from=3-3, to=5-3]
	\arrow["{\Arrowur\beta_m}"{description}, draw=none, from=4-2, to=1-3]
	\arrow["1m1m", from=4-2, to=3-3]
	\arrow["m1m1", from=4-2, to=6-2]
	\arrow["{\Arrowur\beta_m^{-1}}"{description}, draw=none, from=5-1, to=4-2]
	\arrow["{1\beta 1}"', from=5-1, to=6-2]
	\arrow["{\Arrowur\alpha\alpha}"{description}, draw=none, from=6-2, to=3-3]
	\arrow["mm"', from=6-2, to=5-3]
\end{tikzcd}
\end{center}
The isomorphism $(\ast)$ is in fact given by the following (we allow ourselves to rename each hom-object with a single letter, since no multiplication is involved now):
\begin{center}
\begin{tikzcd}
	XAYBZC && XAYZBC \\
	& XYAZBC \\
	XYABZC && XYZABC
	\arrow["{\underline 11\beta 1}", from=1-1, to=1-3]
	\arrow["{\adorn{\cong}{1\Sigma_{\beta,\beta} 1}}"{description}, draw=none, from=1-1, to=2-2]
	\arrow["{1\beta 1\underline 1}"', from=1-1, to=3-1]
	\arrow["{1\beta 1\underline 1}"', from=1-3, to=2-2]
	\arrow[""{name=0, anchor=center, inner sep=0}, "{1\beta \underline 1}", from=1-3, to=3-3]
	\arrow["{\underline 1\beta\underline 1}", from=2-2, to=3-3]
	\arrow["{\underline 11\beta 1}", from=3-1, to=2-2]
	\arrow[""{name=1, anchor=center, inner sep=0}, "{\underline 1\beta 1}"', from=3-1, to=3-3]
	\arrow["{\adorn{\Rightarrow}{1R^{-1}\underline 1}}"{description}, draw=none, from=2-2, to=0]
	\arrow["{\Arrowu \underline 1S1}"{description}, draw=none, from=1, to=2-2]
\end{tikzcd}
\end{center}
where we see the two cells $\underline 1S1$ and $1R^{-1}\underline 1$, as well as the crossing in between, which is provided by the interchanger at $\beta$.
\begin{itemize}
\item[$\underline \lambda,\underline\rho$:] The unitors, for example the left one (the right one is similar) is given by the following 2-cell
\[\begin{tikzcd}
	{\mathbb 1\mathcal B(b,b')\mathcal C(c,c')} & {\mathbb 1\mathcal B(b,b')\mathcal C(c,c')} \\
	{\mathbb 1\mathbb 1\mathcal B(b,b')\mathcal C(c,c')} & {\mathbb 1\mathcal B(b,b')\mathbb 1\mathcal C(c,c')} & {\mathcal B(b,b')\mathcal C(c,c')} \\
	{\mathcal B(b',b')\mathcal C(c',c')\mathcal B(b,b')\mathcal C(c,c')} & {\mathcal B(b',b')\mathcal B(b,b')\mathcal C(c',c')\mathcal C(c,c')}
	\arrow[equals, from=1-1, to=1-2]
	\arrow[equals, from=2-1, to=1-1]
	\arrow["{1\beta 1}", from=2-1, to=2-2]
	\arrow["uu1"', from=2-1, to=3-1]
	\arrow[equals, from=2-2, to=1-2]
	\arrow[""{name=0, anchor=center, inner sep=0}, equals, from=2-2, to=2-3]
	\arrow["u1u1"', from=2-2, to=3-2]
	\arrow["{\Arrowur\beta_u^{-1}}"{description}, draw=none, from=3-1, to=2-2]
	\arrow["{1\beta 1}"', from=3-1, to=3-2]
	\arrow["mm"', shift right=2, shorten <=17pt, from=3-2, to=2-3]
	\arrow["{\Arrowur\lambda\lambda}"{description}, shift right=2, draw=none, from=3-2, to=0]
\end{tikzcd}\]
The commutative square on top (the identical unlabeled 2-cell at the end of the string diagrammatic representation) is commutative by \eqref{ax1 semi-strict braided}.
\end{itemize}
\end{rmk}

\begin{thm}\label{thmBtensorC}
The data of Definition \ref{def tensor product} define $\cat B\otimes\cat C$ as a $\cat V$-bicategory.
\end{thm}

\begin{proof}
After expanding the definition of $\underline{\alpha}$, as well as of $\underline m$, the left-hand side of \eqref{AC Vbicat}
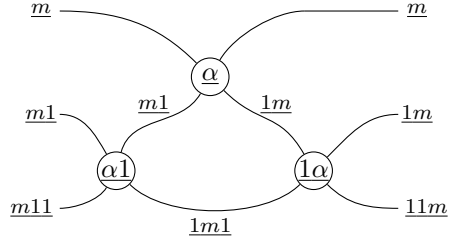
\begin{center}
\begin{tikzpicture}[scale=.5]
		\node [none] (0) at (-5.75, -0.25) {$\underline m\underline 1$};
		\node [none] (1) at (-5.75, 2.5) {$\underline m$};
		\node [none] (2) at (-5.25, -2.75) {};
		\node [none] (3) at (-5.25, -0.25) {};
		\node [none] (4) at (-5.25, 2.5) {};
		\node [fixed size node] (5) at (-3.75, -1.75) {$\underline\alpha \underline1$};
		\node [fixed size node] (7) at (-1.25, 0.75) {$\underline \alpha$};
		\node [fixed size node] (8) at (1.5, -1.75) {$\underline 1\underline\alpha$};
		\node [none] (9) at (1.25, 2.5) {};
		\node [none] (12) at (3.75, -2.75) {};
		\node [none] (14) at (3.75, 2.5) {};
		\node [none] (15) at (3.75, -0.25) {};
		\node [none] (16) at (4.5, -2.75) {$\underline1\underline1\underline m$};
		\node [none] (18) at (4.25, 2.5) {$\underline m$};
		\node [none] (19) at (4.25, -0.25) {$\underline 1\underline m$};
		\node [none] (20) at (-6, -2.75) {$\underline m\underline 1\underline 1$};
		\node [none] (21) at (-2.75, 0) {$\underline m\underline1$};
		\node [none] (22) at (-1.25, -3.25) {$\underline1\underline m\underline1$};
		\node [none] (25) at (0.5, 0) {$\underline1\underline m$};
		
		\draw [in=120, out=0] (3.center) to (5);
		\draw (9.center) to (14.center);
		\draw [in=-120, out=0, looseness=0.75] (2.center) to (5);
		\draw [in=-135, out=-45, looseness=0.75] (5) to (8);
		\draw [in=75, out=-120] (7) to (5);
		\draw [in=135, out=0] (4.center) to (7);
		\draw [in=120, out=-45, looseness=1.25] (7) to (8);
		\draw [in=-180, out=-45] (8) to (12.center);
		\draw [in=-180, out=60, looseness=0.75] (7) to (9.center);
		\draw [in=-180, out=45] (8) to (15.center);
\end{tikzpicture}

\end{center}
looks like the following:
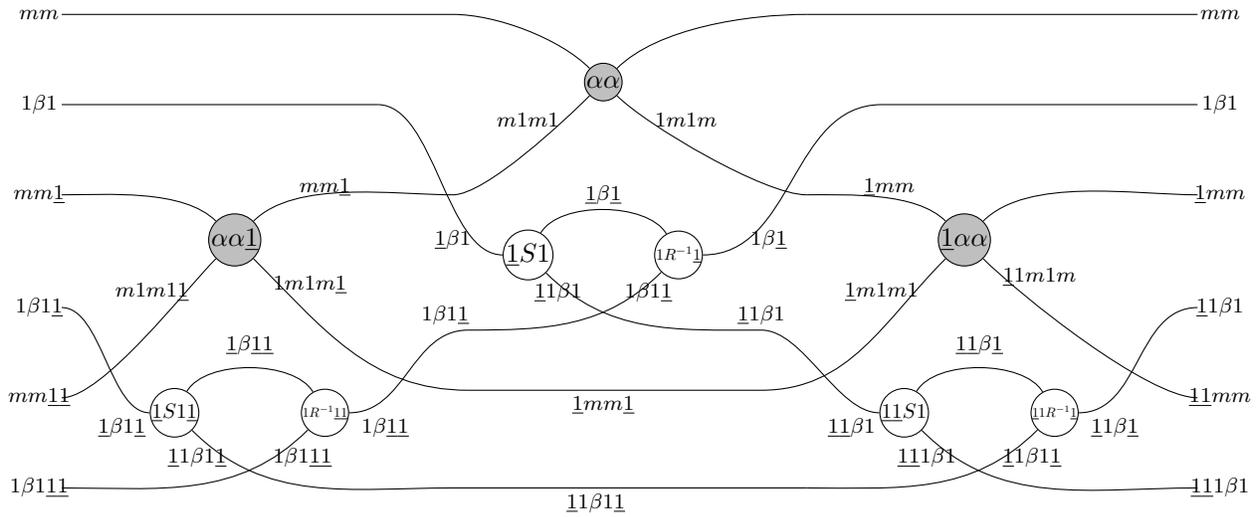
\begin{center}
\begin{tikzpicture}[scale=.4]
		\node [none] (0) at (-20, 7) {$mm$};
		\node [none] (1) at (-20, 4) {$1\beta 1$};
		\node [none] (2) at (-6.25, 7) {};
		\node [none] (3) at (-8.75, 4) {};
		\node [none] (4) at (-6.25, 1) {};
		\node [none] (5) at (-5.75, -3.5) {};
		\node [fixed size node, fill=gray!50] (6) at (-1.25, 4.75) {$\alpha\alpha$};
		\node [fixed size node] (7) at (-3.75, -1) {\adjustbox{valign=m, max width=.6cm, max height=.6cm}{$\underline 1S1$}};
		\node [fixed size node] (8) at (1.25, -1) {\adjustbox{valign=m, max width=.6cm, max height=.6cm}{$1R^{-1}\underline 1$}};
		\node [none] (9) at (4, -3) {$\underline 11\beta 1$};
		\node [none] (10) at (4, -3.5) {};
		\node [none] (11) at (19.25, 4) {$1\beta 1$};
		\node [none] (12) at (19.25, 7) {$mm$};
		\node [none] (13) at (8, 4) {};
		\node [none] (14) at (5.5, 7) {};
		\node [none] (15) at (5.5, 1) {};
		\node [none] (16) at (8.25, 1.25) {$\underline 1mm$};
		\node [none] (17) at (-1.25, 1) {$\underline 1\beta \underline 1$};
		\node [none] (18) at (0.25, -2.25) {$1\beta 1\underline 1$};
		\node [none] (19) at (-2.75, -2.25) {$\underline 1 1\beta 1$};
		\node [none] (20) at (1.5, 3.5) {$1m1m$};
		\node [none] (21) at (-6.25, -0.5) {$\underline 1\beta 1$};
		\node [none] (22) at (-3.75, 3.5) {$m1m1$};
		\node [none] (23) at (4.25, -0.5) {$1\beta \underline 1$};
		\node [none] (24) at (-19.25, 1) {};
		\node [none] (25) at (-19.25, -2.75) {};
		\node [none] (26) at (-19.25, -5.75) {};
		\node [none] (27) at (-19.25, -8.75) {};
		\node [fixed size node, fill=gray!50] (28) at (-13.5, -0.5) {$\alpha\alpha\underline 1$};
		\node [fixed size node] (29) at (-15.5, -6.25) {\adjustbox{valign=m, max width=.6cm, max height=.6cm}{$\underline 1S1\underline 1$}};
		\node [fixed size node] (30) at (-10.5, -6.25) {\adjustbox{valign=m, max width=.6cm, max height=.6cm}{$1R^{-1}\underline 1\underline 1$}};
		\node [none] (31) at (-6.25, -8.75) {};
		\node [none] (32) at (-5.75, -3.5) {};
		\node [none] (33) at (-6.25, 1) {};
		\node [none] (34) at (-5.75, -5.5) {};
		\node [none] (35) at (-13, -4) {$\underline 1\beta \underline 1\underline 1$};
		\node [none] (36) at (-11.25, -7.75) {$1\beta 1\underline 1\underline 1$};
		\node [none] (37) at (-14.75, -7.75) {$\underline 1 1\beta 1\underline 1$};
		\node [none] (38) at (-11, -2) {$1m1m\underline 1$};
		\node [none] (39) at (-17.25, -6.75) {$\underline 1\beta 1\underline 1$};
		\node [none] (40) at (-16.25, -2.25) {$m1m1\underline 1$};
		\node [none] (41) at (-8.5, -6.75) {$1\beta \underline 1\underline 1$};
		\node [none] (42) at (5.5, 1) {};
		\node [none] (43) at (4, -3.5) {};
		\node [none] (44) at (4, -5.5) {};
		\node [none] (45) at (5.5, -8.75) {};
		\node [fixed size node, fill=gray!50] (46) at (10.75, -0.5) {$\underline 1\alpha\alpha$};
		\node [fixed size node] (47) at (8.75, -6.25) {\adjustbox{valign=m, max width=.6cm, max height=.6cm}{$\underline 1\underline 1S1$}};
		\node [fixed size node] (48) at (13.75, -6.25) {\adjustbox{valign=m, max width=.6cm, max height=.6cm}{$\underline 11R^{-1}\underline 1$}};
		\node [none] (49) at (18.5, -8.75) {};
		\node [none] (50) at (18.5, -2.75) {};
		\node [none] (51) at (18.5, 1) {};
		\node [none] (52) at (18.5, -5.75) {};
		\node [none] (53) at (11.25, -4) {$\underline 1\underline 1\beta \underline 1$};
		\node [none] (54) at (13, -7.75) {$\underline 11\beta 1\underline 1$};
		\node [none] (55) at (9.5, -7.75) {$\underline 1\underline 1 1\beta 1$};
		\node [none] (56) at (13.25, -1.75) {$\underline 11m1m$};
		\node [none] (57) at (7, -6.75) {$\underline 1\underline 1\beta 1$};
		\node [none] (58) at (8, -2.25) {$\underline 1m1m1$};
		\node [none] (59) at (15.75, -6.75) {$\underline 11\beta \underline 1$};
		\node [none] (60) at (-19.25, 4) {};
		\node [none] (61) at (-19.25, 7) {};
		\node [none] (62) at (18.5, 4) {};
		\node [none] (63) at (18.5, 7) {};
		\node [none] (64) at (-20, -5.75) {$mm\underline 1\underline 1$};
		\node [none] (65) at (-20, -2.75) {$1\beta 1\underline 1$};
		\node [none] (66) at (-20, 1) {$mm\underline 1$};
		\node [none] (67) at (-20, -8.75) {$1\beta1\underline 1\underline 1$};
		\node [none] (69) at (-1.5, -9.25) {$\underline 11\beta 1\underline 1$};
		\node [none] (70) at (-10.5, 1.25) {$mm\underline 1$};
		\node [none] (71) at (19.25, 1) {$\underline 1 mm$};
		\node [none] (72) at (19.25, -2.75) {$\underline 11\beta 1$};
		\node [none] (73) at (19.25, -5.75) {$\underline 1\underline 1 mm$};
		\node [none] (74) at (19.25, -8.75) {$\underline 1\underline 11\beta 1$};
		\node [none] (75) at (-6.5, -3) {$1\beta 1\underline 1$};
		\node [none] (76) at (-1.25, -6) {$\underline 1 mm\underline 1$};

		\draw [in=135, out=0, looseness=0.75] (2.center) to (6);
		\draw [in=-135, out=0, looseness=0.50] (4.center) to (6);
		\draw [in=0, out=180, looseness=0.75] (7) to (3.center);
		\draw [in=-135, out=0] (5.center) to (8);
		\draw [in=-45, out=180] (10.center) to (7);
		\draw [in=120, out=60, looseness=0.75] (7) to (8);
		\draw [in=180, out=0] (8) to (13.center);
		\draw [in=-45, out=180, looseness=0.50] (15.center) to (6);
		\draw [in=180, out=45, looseness=0.75] (6) to (14.center);
		\draw [in=135, out=0, looseness=0.75] (24.center) to (28);
		\draw [in=-135, out=0, looseness=0.50] (26.center) to (28);
		\draw [in=0, out=180, looseness=0.75] (29) to (25.center);
		\draw [in=-135, out=0] (27.center) to (30);
		\draw [in=-45, out=180] (31.center) to (29);
		\draw [in=120, out=60, looseness=0.75] (29) to (30);
		\draw [in=180, out=0] (30) to (32.center);
		\draw [in=-45, out=180] (34.center) to (28);
		\draw [in=180, out=45, looseness=0.75] (28) to (33.center);
		\draw [in=135, out=0, looseness=0.75] (42.center) to (46);
		\draw [in=-135, out=0] (44.center) to (46);
		\draw [in=0, out=180, looseness=0.75] (47) to (43.center);
		\draw [in=-135, out=0] (45.center) to (48);
		\draw [in=-45, out=180] (49.center) to (47);
		\draw [in=120, out=60, looseness=0.75] (47) to (48);
		\draw [in=180, out=0] (48) to (50.center);
		\draw [in=-45, out=180, looseness=0.50] (52.center) to (46);
		\draw [in=180, out=45, looseness=0.75] (46) to (51.center);
		\draw (34.center) to (44.center);
		\draw (31.center) to (45.center);
		\draw (61.center) to (2.center);
		\draw (3.center) to (60.center);
		\draw (63.center) to (14.center);
		\draw (13.center) to (62.center);
\end{tikzpicture}

\end{center}
In the bottom area there is a tensored list of six hom-objects of $\cat B\otimes\cat C$. Since each hom-object for the tensor product bicategory is a tensored pair of hom-objects from the two tensored bicategories, we are in fact dealing with twelve hom-objects.

The idea is to assemble all of the instances of $\alpha$ (the 2-cells in light gray) in order to apply in parallel associativity coherence for $\cat B$ and $\cat C$. So, let us first slide the two central 2-cells $\underline 1S1$ and $1R^{-1}\underline 1$ below the morphism $\underline 1mm\underline 1$. We find, after a manipulation provided by naturality for $\beta$ and the modification axiom for the central composite of $\underline 1S1$ and $1R^{-1}\underline 1$:

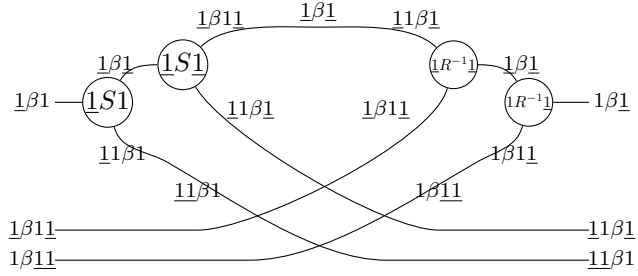
\begin{center}
\begin{tikzpicture}[scale=.4]
		\node [style=none] (0) at (-19.75, 5.25) {$mm$};
		\node [style=none] (1) at (-19.75, 2.25) {$1\beta 1$};
		\node [style=none] (4) at (-6, 5.25) {};
		\node [style=none] (6) at (-6, 2.25) {};
		\node [style=none] (7) at (-5, -8) {};
		\node [fixed size node] (8) at (-1, 3) {$\alpha\alpha$};
		\node [fixed size node] (9) at (-3, -5.5) {\adjustbox{valign=m, max width=.6cm, max height=.6cm}{$\underline 1S1$}};
		\node [fixed size node] (10) at (2, -5.5) {\adjustbox{valign=m, max width=.6cm, max height=.6cm}{$1R^{-1}\underline 1$}};
		\node [style=none] (12) at (4.75, -8) {};
		\node [style=none] (13) at (19.5, 2.25) {$1\beta 1$};
		\node [style=none] (14) at (19.5, 5.25) {$mm$};
		\node [style=none] (16) at (5.75, 5.25) {};
		\node [style=none] (17) at (5.25, 2.25) {};
		\node [style=none] (18) at (7.25, 2.75) {$\underline 1mm$};
		\node [style=none] (19) at (-0.5, -3.5) {$\underline 1\beta \underline 1$};
		\node [style=none] (20) at (1.5, -7) {$1\beta 1\underline 1$};
		\node [style=none] (21) at (-2.5, -7) {$\underline 1 1\beta 1$};
		\node [style=none] (22) at (1.25, 1.75) {$1m1m$};
		\node [style=none] (23) at (-5.25, -4.75) {$\underline 1\beta 1$};
		\node [style=none] (24) at (-3, 1.75) {$m1m1$};
		\node [style=none] (25) at (4.5, -4.75) {$1\beta \underline 1$};
		\node [style=none] (26) at (-19, -0.75) {};
		\node [style=none] (27) at (-19, -4.5) {};
		\node [style=none] (28) at (-19, -7.5) {};
		\node [style=none] (29) at (-19, -10.5) {};
		\node [fixed size node] (30) at (-9.75, 1.25) {\adjustbox{valign=m, max width=.6cm, max height=.6cm}{$\alpha1\alpha1$}};
		\node [fixed size node] (31) at (-15.25, -8){\adjustbox{valign=m, max width=.6cm, max height=.6cm} {$\underline 1S1\underline 1$}};
		\node [fixed size node] (32) at (-10.25, -8) {\adjustbox{valign=m, max width=.6cm, max height=.6cm}{$1R^{-1}\underline 1\underline 1$}};
		\node [style=none] (33) at (-6, -10.5) {};
		\node [style=none] (35) at (-6, 2.25) {};
		\node [style=none] (36) at (-6, 0) {};
		\node [style=none] (37) at (-12.75, -5.75) {$\underline 1\beta \underline 1\underline 1$};
		\node [style=none] (38) at (-10.75, -9.25) {$1\beta 1\underline 1\underline 1$};
		\node [style=none] (39) at (-14.5, -9.5) {$\underline 1 1\beta 1\underline 1$};
		\node [style=none] (41) at (-16.75, -8.5) {$\underline 1\beta 1\underline 1$};
		\node [style=none] (42) at (-12.25, -0.5) {$m\underline1m\underline 1$};
		\node [style=none] (43) at (-4.5, -8.5) {$1\beta \underline 1\underline 1$};
		\node [style=none] (44) at (5.25, 2.25) {};
		\node [style=none] (46) at (5, 0) {};
		\node [style=none] (47) at (5.75, -10.5) {};
		\node [fixed size node] (48) at (10.25, 1) {\adjustbox{valign=m, max width=.6cm, max height=.6cm}{$1\alpha1\alpha$}};
		\node [fixed size node] (49) at (9, -8) {\adjustbox{valign=m, max width=.6cm, max height=.6cm}{$\underline 1\underline 1S1$}};
		\node [fixed size node] (50) at (14, -8) {\adjustbox{valign=m, max width=.6cm, max height=.6cm}{$\underline 11R^{-1}\underline 1$}};
		\node [style=none] (51) at (18.75, -10.5) {};
		\node [style=none] (52) at (18.75, -4.5) {};
		\node [style=none] (53) at (18.75, -0.75) {};
		\node [style=none] (54) at (18.75, -7.5) {};
		\node [style=none] (55) at (11.5, -5.75) {$\underline 1\underline 1\beta \underline 1$};
		\node [style=none] (56) at (13.25, -9.5) {$\underline 11\beta 1\underline 1$};
		\node [style=none] (57) at (9.5, -9.25) {$\underline 1\underline 1 1\beta 1$};
		\node [style=none] (58) at (12.5, -0.25) {$\underline 1m\underline 1m$};
		\node [style=none] (59) at (4, -8.5) {$\underline 1\underline 1\beta 1$};
		\node [style=none] (61) at (15.75, -8.5) {$\underline 11\beta \underline 1$};
		\node [style=none] (62) at (-19, 2.25) {};
		\node [style=none] (63) at (-19, 5.25) {};
		\node [style=none] (64) at (18.75, 2.25) {};
		\node [style=none] (65) at (18.75, 5.25) {};
		\node [style=none] (66) at (-19.75, -7.5) {$mm\underline 1\underline 1$};
		\node [style=none] (67) at (-19.75, -4.5) {$1\beta 1\underline 1$};
		\node [style=none] (68) at (-19.75, -0.75) {$mm\underline 1$};
		\node [style=none] (69) at (-19.75, -10.5) {$1\beta1\underline 1\underline 1$};
		\node [style=none] (71) at (-1.25, -11.25) {$\underline 11\beta 1\underline 1$};
		\node [style=none] (72) at (-7, 3) {$mm\underline 1$};
		\node [style=none] (73) at (19.5, -0.75) {$\underline 1 mm$};
		\node [style=none] (74) at (19.5, -4.5) {$\underline 11\beta 1$};
		\node [style=none] (75) at (19.5, -7.5) {$\underline 1\underline 1 mm$};
		\node [style=none] (76) at (19.5, -10.5) {$\underline 1\underline 11\beta 1$};
		\node [style=none] (81) at (-10.25, -3.75) {};
		\node [style=none] (82) at (11, -3.5) {};
		\node [style=none] (83) at (-0.75, -0.5) {$1m\underline 1m1$};
		\node [style=none] (86) at (-12.5, 2.5) {$m1m1$};
		\node [style=none] (87) at (12.75, 2.25) {$1m1m$};
		\node [style=none] (88) at (14.25, -1.25) {$1\beta \underline 1$};
		\node [style=none] (89) at (-14.25, -1) {$\underline 1\beta 1$};
		\node [style=none] (96) at (-15.25, -4) {$m1m1\underline 1$};
		\node [style=none] (97) at (15.5, -4.25) {$\underline 11m1m$};
		\node [style=none] (98) at (-10.75, -2.5) {};
		\node [style=none] (99) at (-11.75, -3.5) {};
		\node [style=none] (100) at (9.75, -2.5) {};
		\node [style=none] (101) at (10.75, -3.5) {};
		\node [style=none] (102) at (10.75, -8.75) {};
		\node [style=none] (103) at (9.75, -9.75) {};
		\node [style=none] (104) at (-10.75, -9.75) {};
		\node [style=none] (105) at (-11.75, -8.75) {};
		\node [style=none] (106) at (-10.25, -1.5) {};
		\node [style=none] (107) at (-11.25, -0.5) {};
		\node [style=none] (108) at (-10.25, 5.75) {};
		\node [style=none] (109) at (-11.25, 4.75) {};
		\node [style=none] (110) at (10.25, 5.75) {};
		\node [style=none] (111) at (11.25, 4.75) {};
		\node [style=none] (112) at (11.25, -0.5) {};
		\node [style=none] (113) at (10.25, -1.5) {};

		\draw [in=135, out=0, looseness=0.75] (4.center) to (8);
		\draw [in=-135, out=0, looseness=0.50] (6.center) to (8);
		\draw [in=-135, out=0] (7.center) to (10);
		\draw [in=-45, out=180] (12.center) to (9);
		\draw [in=120, out=60, looseness=0.75] (9) to (10);
		\draw [in=-45, out=180, looseness=0.50] (17.center) to (8);
		\draw [in=180, out=45, looseness=0.75] (8) to (16.center);
		\draw [in=135, out=0, looseness=0.75] (26.center) to (30);
		\draw [in=-135, out=0, looseness=0.50] (28.center) to (30);
		\draw [in=0, out=180, looseness=0.75] (31) to (27.center);
		\draw [in=-135, out=0] (29.center) to (32);
		\draw [in=-45, out=180] (33.center) to (31);
		\draw [in=120, out=60, looseness=0.75] (31) to (32);
		\draw [in=-45, out=180, looseness=0.50] (36.center) to (30);
		\draw [in=180, out=45, looseness=0.75] (30) to (35.center);
		\draw [in=135, out=0, looseness=0.75] (44.center) to (48);
		\draw [in=-135, out=0, looseness=0.50] (46.center) to (48);
		\draw [in=-135, out=0] (47.center) to (50);
		\draw [in=-45, out=180] (51.center) to (49);
		\draw [in=120, out=60, looseness=0.75] (49) to (50);
		\draw [in=180, out=0] (50) to (52.center);
		\draw [in=-45, out=180, looseness=0.50] (54.center) to (48);
		\draw [in=180, out=45, looseness=0.75] (48) to (53.center);
		\draw [in=180, out=0] (36.center) to (46.center);
		\draw (33.center) to (47.center);
		\draw (63.center) to (4.center);
		\draw (65.center) to (16.center);
		\draw [in=-150, out=0, looseness=0.75] (10) to (82.center);
		\draw [in=180, out=30, looseness=0.75] (82.center) to (64.center);
		\draw [in=-30, out=180, looseness=0.75] (9) to (81.center);
		\draw [in=0, out=150, looseness=0.50] (81.center) to (62.center);
		\draw [style=dashed, in=180, out=-90] (105.center) to (104.center);
		\draw [style=dashed] (105.center) to (99.center);
		\draw [style=dashed, in=180, out=90] (99.center) to (98.center);
		\draw [style=dashed] (98.center) to (100.center);
		\draw [style=dashed, in=90, out=0] (100.center) to (101.center);
		\draw [style=dashed] (101.center) to (102.center);
		\draw [style=dashed, in=0, out=-90] (102.center) to (103.center);
		\draw [style=dashed] (103.center) to (104.center);
		\draw [style=dashed, in=90, out=-180] (108.center) to (109.center);
		\draw [style=dashed, in=180, out=0] (108.center) to (110.center);
		\draw [style=dashed, in=90, out=0] (110.center) to (111.center);
		\draw [style=dashed] (111.center) to (112.center);
		\draw [style=dashed, in=0, out=-90] (112.center) to (113.center);
		\draw [style=dashed] (113.center) to (106.center);
		\draw [style=dashed, in=270, out=-180] (106.center) to (107.center);
		\draw [style=dashed] (107.center) to (109.center);
		\draw [in=360, out=180] (7.center) to (32);
		\draw [in=180, out=0] (12.center) to (49);
\end{tikzpicture}

\end{center}
Now in the highlighted box on top we can apply the associativity coherence for $\cat B$ and $\cat C$. The dotted box at the bottom can instead be manipulated as follows. Firs, let us join identities (see Remark \ref{join identities}) in the two 2-cells $1R^{-1}\underline 1\underline 1$ and $\underline 1\underline 1 S1$. These become then respectively $1R1\underline 1$ and $\underline 1 1S1$:
\begin{center}
\begin{tikzpicture}[scale=.5]
		\node [fixed size node] (1) at (-2.5, 0.5) {\adjustbox{valign=m, max width=.6cm, max height=.6cm}{$\underline 1S1$}};
		\node [fixed size node] (2) at (2.5, 0.5) {\adjustbox{valign=m, max width=.6cm, max height=.6cm}{$1R^{-1}\underline 1$}};
		\node [style=none] (4) at (0, 2.5) {$\underline 1\beta \underline 1$};
		\node [style=none] (5) at (2, -1) {$1\beta 1\underline 1$};
		\node [style=none] (6) at (-2, -1) {$\underline 1 1\beta 1$};
		\node [style=none] (7) at (-7.75, 0.5) {$\underline 1\beta 1$};
		\node [style=none] (8) at (7.25, 0.5) {$1\beta \underline 1$};
		\node [fixed size node] (9) at (-4.25, -2) {\adjustbox{valign=m, max width=.6cm, max height=.6cm}{$1R^{-1} 1\underline 1$}};
		\node [style=none] (10) at (-7.75, -3) {$1\beta \underline 1\underline 1$};
		\node [style=none] (11) at (-3, -2.5) {$1\beta 1\underline 1$};
		\node [fixed size node] (12) at (4.25, -2) {\adjustbox{valign=m, max width=.6cm, max height=.6cm}{$\underline 1 1S1$}};
		\node [style=none] (14) at (7.25, -3) {$\underline 1\underline 1\beta 1$};
		\node [style=none] (15) at (-7, 0.5) {};
		\node [style=none] (16) at (-7, -3) {};
		\node [style=none] (17) at (-7, -1.25) {};
		\node [style=none] (18) at (6.5, 0.5) {};
		\node [style=none] (19) at (6.5, -1.25) {};
		\node [style=none] (20) at (6.5, -3) {};
		\node [style=none] (21) at (-7.75, -1.25) {$\underline 1\beta 1\underline 1$};
		\node [style=none] (22) at (7.25, -1.25) {$\underline 1 1\beta\underline 1$};
		\node [style=none] (23) at (3, -2.5) {$\underline 1 1\beta 1$};
		
		\draw [in=120, out=60, looseness=0.75] (1) to (2);
		\draw [in=360, out=180, looseness=0.75] (1) to (15.center);
		\draw [in=180, out=-45, looseness=0.75] (12) to (20.center);
		\draw [in=-180, out=30, looseness=0.75] (12) to (19.center);
		\draw [in=0, out=-180] (18.center) to (2);
		\draw [in=0, out=-135, looseness=0.75] (9) to (16.center);
		\draw [in=0, out=135, looseness=0.75] (9) to (17.center);
		\draw [in=-180, out=-60, looseness=0.75] (1) to (12);
		\draw [in=-120, out=0, looseness=0.75] (9) to (2);
\end{tikzpicture}
\end{center}
Then, by the modification axiom, this is equal to
\begin{center}
\begin{tikzpicture}[scale=.4]
		\node [fixed size node] (0) at (-7, 2) {\adjustbox{valign=m, max width=.6cm, max height=.6cm}{$\underline 1S1$}};
		\node [fixed size node] (1) at (7, 2) {\adjustbox{valign=m, max width=.6cm, max height=.6cm}{$1R^{-1}\underline 1$}};
		\node [style=none] (2) at (0, 3.75) {$\underline 1\beta \underline 1$};
		\node [style=none] (3) at (6.75, 0.5) {$1\beta 1\underline 1$};
		\node [style=none] (5) at (-9.5, 2) {$\underline 1\beta 1$};
		\node [style=none] (6) at (9.75, 2) {$1\beta \underline 1$};
		\node [fixed size node] (7) at (4.25, 0.25) {\adjustbox{valign=m, max width=.6cm, max height=.6cm}{$1R^{-1}1\underline 1$}};
		\node [style=none] (8) at (-9.5, -3.25) {$1\beta \underline 1\underline 1$};
		\node [fixed size node] (10) at (-4.25, 0.25) {\adjustbox{valign=m, max width=.6cm, max height=.6cm}{$\underline 1 1S1$}};
		\node [style=none] (11) at (9.75, -3.25) {$\underline 1\underline 1\beta 1$};
		\node [style=none] (12) at (-8.75, 2) {};
		\node [style=none] (13) at (-2.25, -3.25) {};
		\node [style=none] (14) at (-4, -2.25) {};
		\node [style=none] (15) at (9, 2) {};
		\node [style=none] (16) at (4, -2.25) {};
		\node [style=none] (17) at (2.25, -3.25) {};
		\node [style=none] (18) at (-9.5, -2.25) {$\underline 1\beta 1\underline 1$};
		\node [style=none] (19) at (9.75, -2.25) {$\underline 1 1\beta\underline 1$};
		\node [style=none] (20) at (0, 3.25) {};
		\node [style=none] (21) at (2.5, 1.25) {};
		\node [style=none] (22) at (-2.5, 1.25) {};
		\node [style=none] (24) at (9, -2.25) {};
		\node [style=none] (25) at (-8.75, -2.25) {};
		\node [style=none] (26) at (-8.75, -3.25) {};
		\node [style=none] (27) at (9, -3.25) {};
		\node [style=none] (28) at (2.25, 1.75) {$\underline 1 \beta 1\underline 1$};
		\node [style=none] (29) at (-2.25, 1.75) {$\underline 1 1\beta \underline 1$};
		\node [style=none] (30) at (-4, -1) {$\underline 1\underline 1\beta 1$};
		\node [style=none] (31) at (4, -1) {$1\beta\underline 1\underline1$};
		\node [style=none] (32) at (-6.75, 0.5) {$\underline 1 1\beta 1$};
		
		\draw [in=360, out=180, looseness=0.75] (0) to (12.center);
		\draw [in=180, out=-60, looseness=0.75] (10) to (17.center);
		\draw [in=0, out=-180] (15.center) to (1);
		\draw [in=0, out=-120, looseness=0.75] (7) to (13.center);
		\draw [in=-180, out=-60, looseness=0.75] (0) to (10);
		\draw [in=-120, out=0, looseness=0.75] (7) to (1);
		\draw [in=45, out=-180, looseness=0.75] (20.center) to (0);
		\draw [in=135, out=0, looseness=0.75] (20.center) to (1);
		\draw [in=0, out=135] (7) to (21.center);
		\draw [in=0, out=180, looseness=0.75] (21.center) to (14.center);
		\draw [in=-180, out=45] (10) to (22.center);
		\draw [in=-180, out=0] (22.center) to (16.center);
		\draw [in=360, out=180] (24.center) to (16.center);
		\draw [in=180, out=0] (25.center) to (14.center);
		\draw [in=180, out=0] (26.center) to (13.center);
		\draw [in=180, out=0] (17.center) to (27.center);
\end{tikzpicture}
\end{center}
We can hence apply braiding axioms \eqref{BA1} and \eqref{BA2} in order to get
\begin{center}
\begin{tikzpicture}[scale=.4]
		\node [fixed size node] (0) at (-7, 1.75) {\adjustbox{valign=m, max width=.6cm, max height=.6cm}{$\underline 1S1$}};
		\node [fixed size node] (1) at (7, 1.75) {\adjustbox{valign=m, max width=.6cm, max height=.6cm}{$1R^{-1}\underline 1$}};
		\node [style=none] (2) at (0, 4.75) {$\underline 1\beta \underline 1$};
		\node [style=none] (3) at (6.5, 0) {$1\beta 1\underline 1$};
		\node [style=none] (4) at (-9.5, 1.75) {$\underline 1\beta 1$};
		\node [style=none] (5) at (9.75, 1.75) {$1\beta \underline 1$};
		\node [style=none] (7) at (-9.5, -3.5) {$1\beta \underline 1\underline 1$};
		\node [style=none] (9) at (9.75, -3.5) {$\underline 1\underline 1\beta 1$};
		\node [style=none] (10) at (-8.75, 1.75) {};
		\node [style=none] (11) at (-2.25, -3.5) {};
		\node [style=none] (12) at (-4, -2.5) {};
		\node [style=none] (13) at (9, 1.75) {};
		\node [style=none] (14) at (4, -2.5) {};
		\node [style=none] (15) at (2.25, -3.5) {};
		\node [style=none] (16) at (-9.5, -2.5) {$\underline 1\beta 1\underline 1$};
		\node [style=none] (17) at (9.75, -2.5) {$\underline 1 1\beta\underline 1$};
		\node [style=none] (18) at (0, 4.25) {};
		\node [style=none] (21) at (9, -2.5) {};
		\node [style=none] (22) at (-8.75, -2.5) {};
		\node [style=none] (23) at (-8.75, -3.5) {};
		\node [style=none] (24) at (9, -3.5) {};
		\node [style=none] (25) at (2.25, 1.5) {$\underline 1 \beta 1\underline 1$};
		\node [style=none] (26) at (-2.25, 1.5) {$\underline 1 1\beta \underline 1$};
		\node [style=none] (27) at (-4, -1.25) {$\underline 1\underline 1\beta 1$};
		\node [style=none] (28) at (4, -1.25) {$1\beta\underline 1\underline1$};
		\node [style=none] (29) at (-6.5, 0) {$\underline 1 1 \beta 1$};
		\node [style=none] (30) at (-5, -0.25) {};
		\node [fixed size node] (31) at (-4.5, 3) {\adjustbox{valign=m, max width=.6cm, max height=.6cm}{$\underline 1S\underline 1$}};
		\node [fixed size node] (32) at (4.5, 3) {\adjustbox{valign=m, max width=.6cm, max height=.6cm}{$\underline 1 R^{-1}\underline 1$}};
		\node [style=none] (33) at (5.25, -0.25) {};
		\node [style=none] (34) at (-3.25, 4.5) {$\underline 1 \beta 1\underline 1$};
		\node [style=none] (35) at (3.25, 4.5) {$\underline 1 1\beta\underline 1$};
		\node [style=none] (36) at (6.75, 3) {$\underline 1\beta \underline 1$};
		\node [style=none] (37) at (-6.75, 3) {$\underline 1\beta \underline 1$};
		
		\draw [in=360, out=180, looseness=0.75] (0) to (10.center);
		\draw [in=0, out=-180] (13.center) to (1);
		\draw [in=360, out=180] (21.center) to (14.center);
		\draw [in=180, out=0] (22.center) to (12.center);
		\draw [in=180, out=0] (23.center) to (11.center);
		\draw [in=180, out=0] (15.center) to (24.center);
		\draw [in=150, out=-75] (0) to (30.center);
		\draw [in=180, out=-30, looseness=0.75] (30.center) to (15.center);
		\draw [in=-180, out=45, looseness=0.75] (31) to (18.center);
		\draw [in=60, out=-180] (31) to (0);
		\draw [in=0, out=120] (1) to (32);
		\draw [in=0, out=135, looseness=0.75] (32) to (18.center);
		\draw [in=30, out=-105] (1) to (33.center);
		\draw [in=0, out=-150, looseness=0.75] (33.center) to (11.center);
		\draw [in=-105, out=0, looseness=0.50] (12.center) to (32);
		\draw [in=180, out=-60, looseness=0.50] (31) to (14.center);
\end{tikzpicture}

\end{center}
Now, let us return to the whole 2-cell, and let us also reduce, by the same Remark (\ref{join identities}) used above, the two external 2-cells $\underline 1S1\underline 1$ at the beginning and $\underline 1 1R^{-1}\underline 1$ at the end of the composition to $\underline 1S\underline 1$ and $\underline 1 R^{-1}\underline 1$. The careful reader will easily be able to place the objects in the areas of the string diagram in order to see how this join of identities is coherent with the one used before. That means, the joining of identities could have in fact be done at once by replacing the composite $1R^{-1}\underline 1\underline 1\circ \underline 1 S 1\underline 1$ by $1R^{-1}\underline 1\circ \underline 1 S\underline 1$, and similarly on the right part of the 2-cell by replacing $\underline 1 1R^{-1}\underline 1\circ \underline 1\underline 1 S1$ by $1\underline 1 1S 1\circ \underline 1R\underline 1$.

The whole 2-cell becomes then, by reinserting,
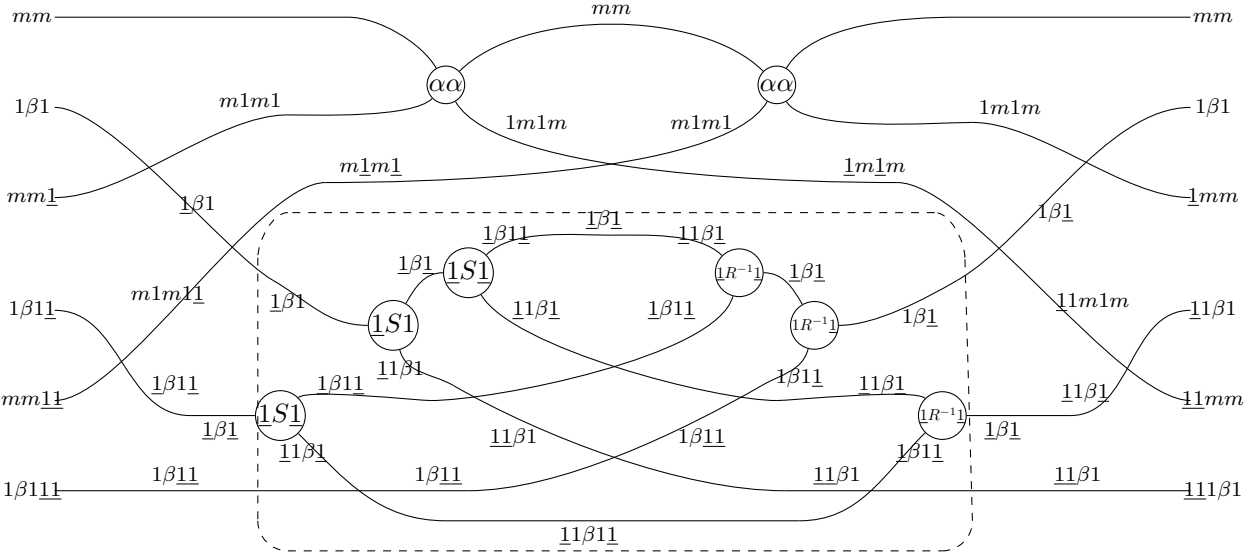
\begin{center}


\end{center}
Now, we can focus again on what is happening in the bottom highlighted part. By naturality of $\beta$ we can again slide 1-cells below the external $\underline 1S\underline 1$ and $\underline 1R^{-1}\underline 1$, and find
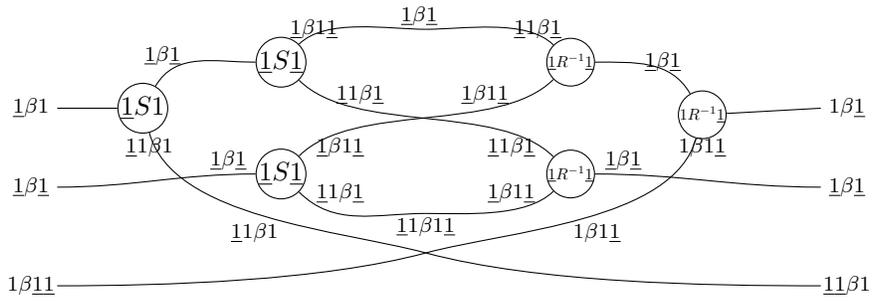
\begin{center}
%

\end{center}
Now again we will join some unitors. Specifically, we can write the 2-cell above as
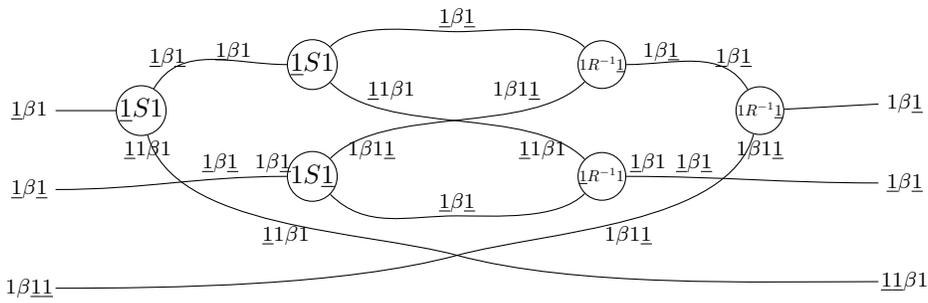
\begin{center}
%

\end{center}
and then apply axiom \eqref{BA4} in order to get
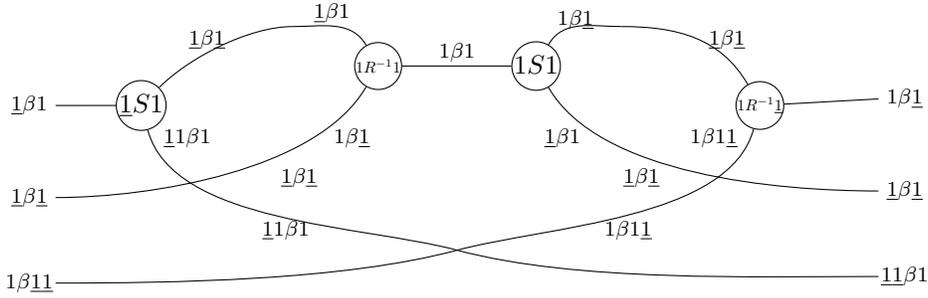
\begin{center}
%

\end{center}
Now, let us slide the unitors in between the two 2-cells $1R^{-1}1$ and $1S1$, in order to get
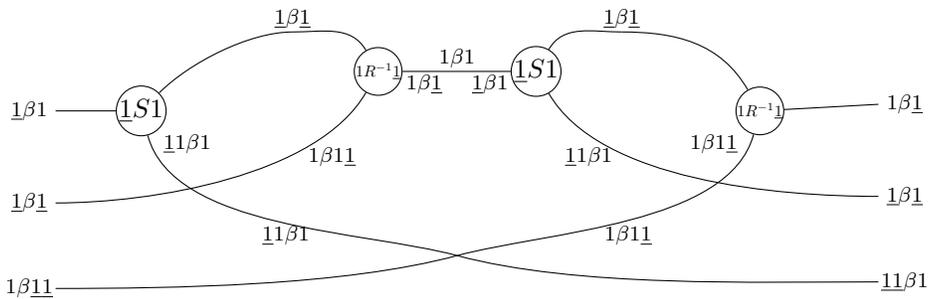
\begin{center}
%

\end{center}
Let's then replace this in the complete diagram above:
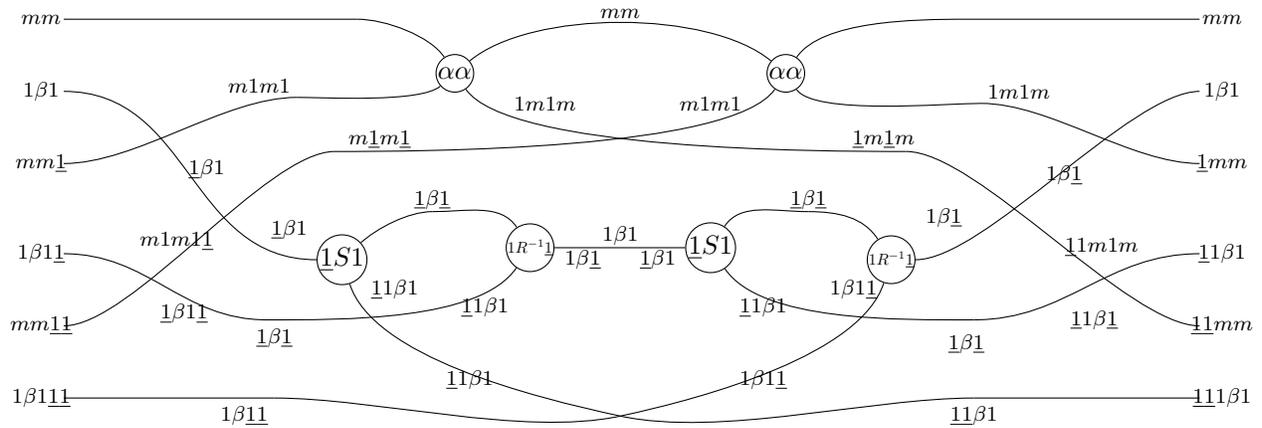
\begin{center}
%

\end{center}
We can observe that it suffices to slide - by the modification property \eqref{modification property R and S} as well as by the interchanger axiom \eqref{axiom 3 interchanger} - the two morphisms crossing each other from the two cells $\alpha\alpha$, in order to precisely get the right-hand side of \eqref{AC Vbicat}:
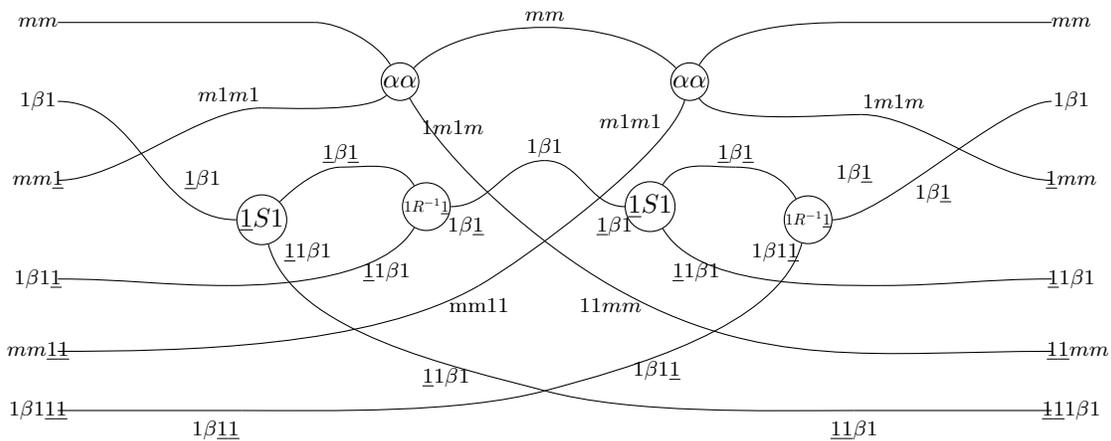
\begin{center}
%

\end{center}
For what concerns the identity coherence axiom, the proof is as follows. We start with the right-hand side of the axioms, that looks like
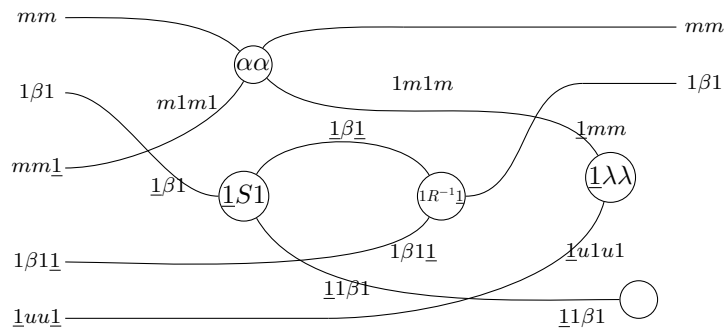
\begin{center}
%

\end{center}
Then, we can use the modification axiom \eqref{modification axiom} for $R^{-1}$ and $S$, jointly to the interchanger axiom \eqref{axiom 3 interchanger}, in order to slide the unit morphism on top of the composite of $1R^{-1}\underline 1$ and $\underline 1S1$. This gives
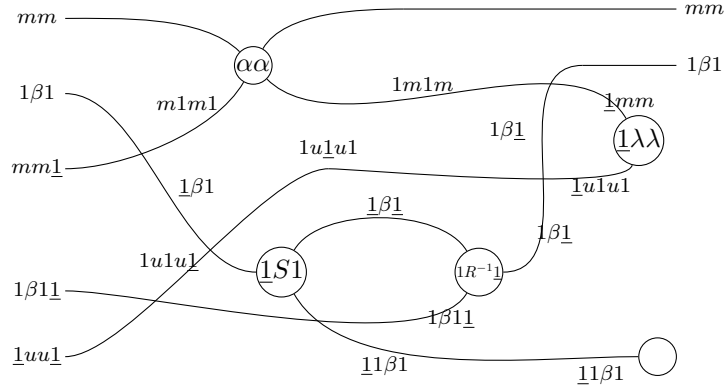
\begin{center}
%
\end{center}
Now, we see that we can slide the 2-cell $\underline 1\lambda\lambda$ on top of the braiding 1-morphism. This is the naturality of the pseudonatural transformation $\beta$, which gives
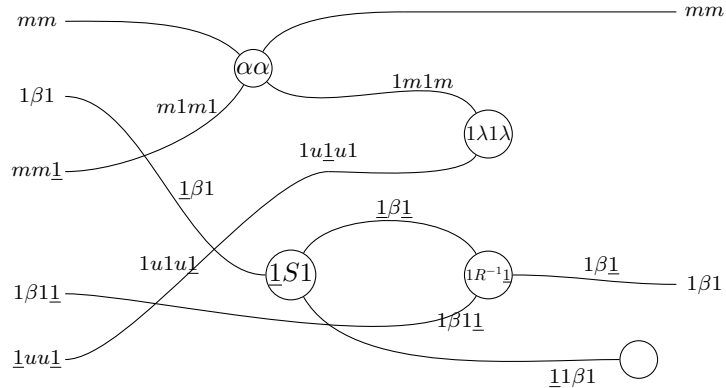
\begin{center}
%
\end{center}

This allows us to recognize the coherence axiom for $\mathcal{B}$ and $\mathcal C$ (as an identity given by the tensor product of the two axioms), and hence this becomes
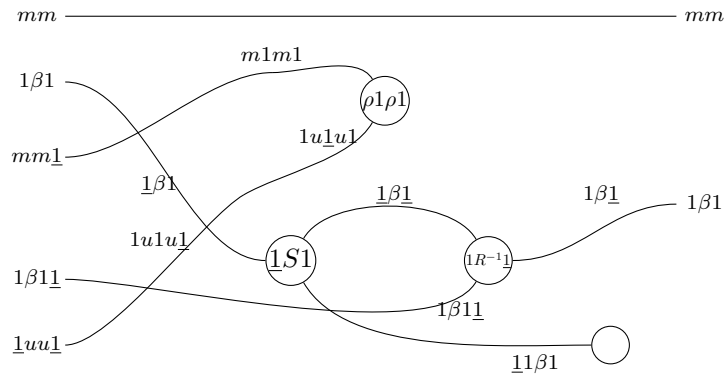
\begin{center}
%

\end{center}
Now, again by naturality we can slide $\rho 1\rho 1$ to the left turning it into the desired $\rho\rho\underline 1$. Then, we also can link the morphisms $1\beta 1$, finding
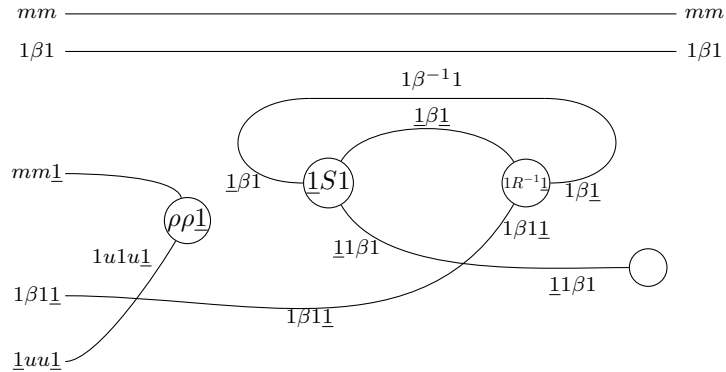
\begin{center}
%
\end{center}
Eventually, similarly to the proof of the identity coherence in Theorem \ref{thmCop}, it's easy to see that the semi-strictification axioms for the braiding \eqref{ax2 semi-strict braided} and \eqref{ax3 semi-strict braided}, together with the interchanger axiom \eqref{axiom 1 interchanger}, imply that the 2-cell
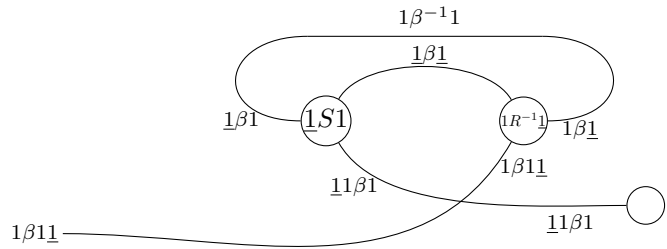
\begin{center}
%
\end{center}
is in fact the identity. Hence, the result of the computation is precisely the right-hand side of \eqref{IC Vbicat} for the tensor product $\cat V$-bicategory, that is
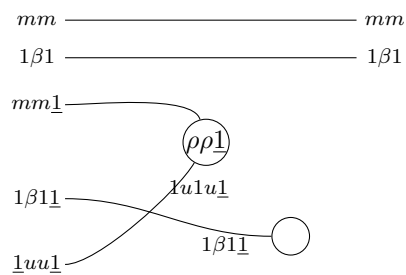
\begin{center}
%

\end{center}
concluding the proof.
\end{proof}

\begin{rmk}\label{Iop times I}
One can easily check that the unit $\cat V$-bicategory $\cat I$ works in fact as a unit for the tensor product of $\cat V$-bicategories. Moreover is it easily seen to be isomorphic to its opposite, and hence we have $\cat I\op\otimes\cat I\cong\cat I$. This observation will allow us in Section \ref{section bicoends} to talk about \emph{constant} $\cat V$-pseudofunctors $\cat I\op\otimes\cat I\to \cat D$ into a $\cat V$-bicategory $\cat D$.

The constant $\cat V$-pseudofunctor over an object $d$ in a $\cat V$-bicategory $\cat D$ is in fact the $\cat V$-pseudofunctor $\cat I\to\cat D$ given by mapping the unique object $\ast$ of $\cat I$ over the object $d$, together with $u_d\colon\mathbb 1=\cat I(\ast,\ast)\to \cat D(d,d)$.
\end{rmk}

\begin{rmk}\label{opdistributeovertensor}
It is an easy check to observe that the opposite distribute over the tensor product, in the sense that for $\cat V$-bicategories $\cat B,\cat C$ there is a $\cat V$-equivalence between $(\cat B\otimes\cat C)\op$ and $\cat B\op\otimes\cat C\op$.
\end{rmk}

\subsection{$\cat V$ as a $\cat V$-bicategory}\label{section closed monoidal}

The present section is the last one of this article dedicated to a proof of the coherence axioms for a $\cat V$-bicategory. The $\cat V$-bicategorical structure that we are going to prove is that of the monoidal bicategory $\cat V$ itself, under the assumption of being right closed (Section \ref{sec string adj}).

\begin{defn}\label{def V as V bicat}
Let $\cat V$ be a right closed monoidal bicategoryn as in Section \ref{subsect tensor-hom}. Then we can define
\begin{itemize}
\item Its class of objects $\cat V_0$
\item For every pair of objects $a,b$, the hom-object $\cat V(a,b)\coloneqq[a,b]$ in $\cat V$, together with, for every triple of object $a,b,c$, the multiplication
\[m\colon [b,c][a,b]\overset{\eta}\longrightarrow[a,[b,c][a,b]a]\overset{[1,1\varepsilon]}\longrightarrow[a,[b,c]b]\overset{[1,\varepsilon]}\longrightarrow[a,c],\]
and the unit
\[u\colon \mathbb{1}\overset{\eta}\longrightarrow[a,\mathbb{1}a]=[a,a].\]
\item The associator, given by the following string diagram
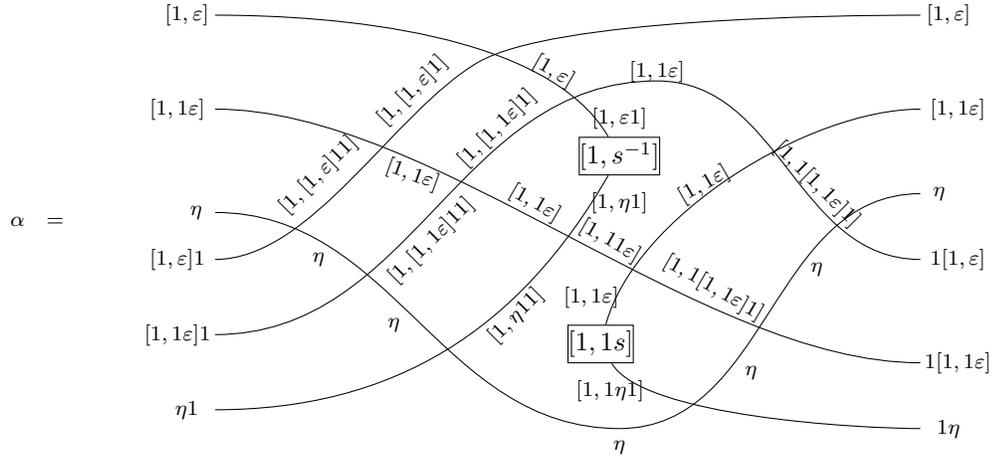
\begin{center}

\end{center}
and the two unitors
\begin{center}
%
\end{center}
\end{itemize}
\end{defn}

\begin{thm}\label{closed monoidal is self enriched}
The structure in Definition \ref{def V as V bicat} equips the right closed monoidal bicategory $\cat V$ with a structure of $\cat V$-bicategory.
\end{thm}

\begin{proof}
The left-hand side of the identity coherence \eqref{IC Vbicat} is given by the 2-cell
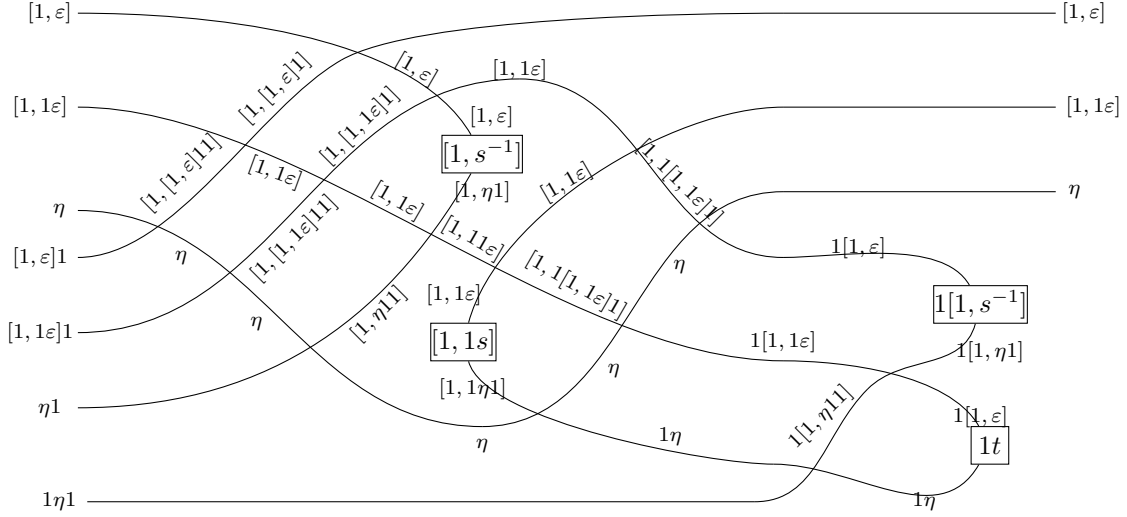
\begin{center}
%
\end{center}
The idea is then to cancel out the couple of 2-cells $[1,1s]$ and $1t$ by the swallowtail equation \eqref{swallowtail2}. In order to make the right-hand side of \eqref{swallowtail2} appear, we first let $\eta$ slide below the two 2-cell at the bottom-right - forming $1\lambda$ - by naturality. This, as in axiom \eqref{naturality axiom} will result in applying the target pseudofunctor $[1,-1]$ to the two 2-cells:
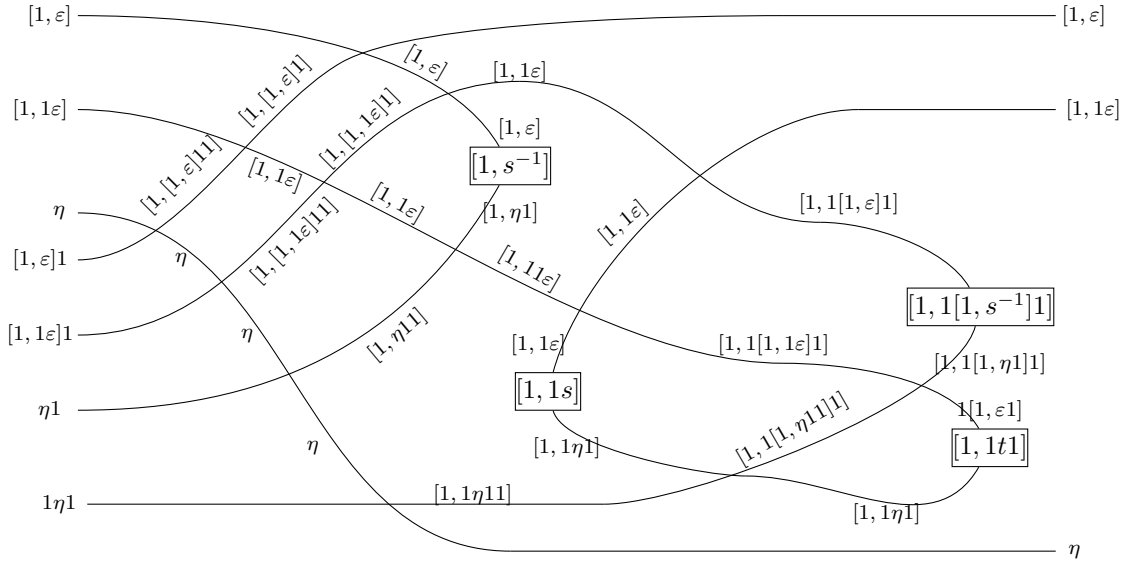
\begin{center}
%
\end{center}
Then, we are able to also slide the 2-cell $[1,1[1,s^{-1}]1]$, by naturality of $\varepsilon$, to the left of the morphism $[1,1\varepsilon]$. This will turn it into the 2-cell $[1,1s^{-1}]$. At the same time, the modification property for the 2-cell $[1,1s]$ allows us to slide past the 1-cell on its right, and then apply the swallowtail equation \eqref{swallowtail2}:
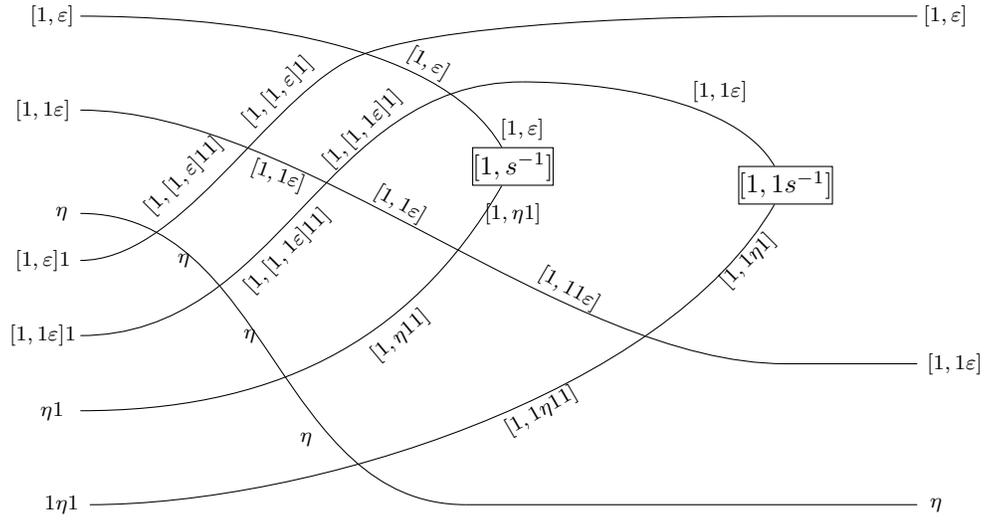
\begin{center}
%
\end{center}
Now, let us move the 2-cell $[1,1s^{-1}]$. The 1-cell $[1,1\eta 1]$ entering in it can be passed over the other 2-cell, $[1,s^{-1}]$, then we can slide $[1,1s^{-1}]$ at the left of the 1-cell $[1,\varepsilon]$ entering in $[1,s^{-1}]$, by naturality of $\varepsilon$. We find then:
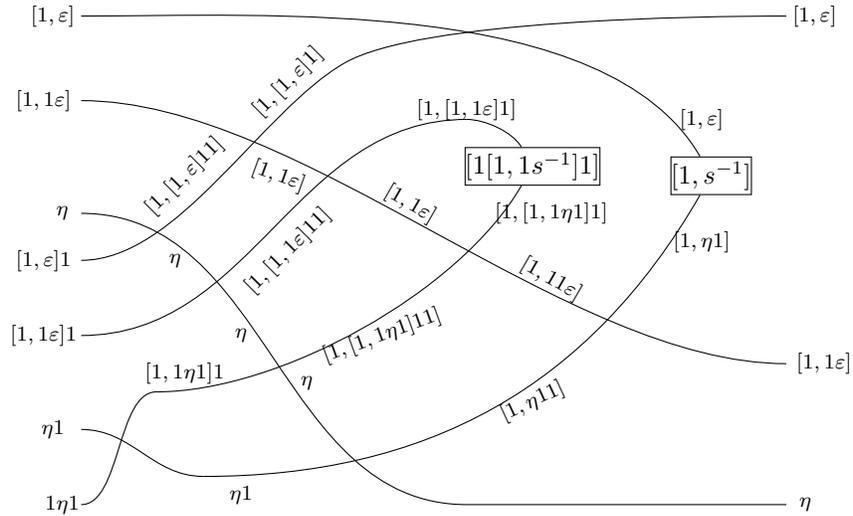
\begin{center}
%
\end{center}
Let's continue by shifting the same 2-cell under $[1,1\varepsilon]$ (trivial) and $\eta$ (by naturality of $\eta$):
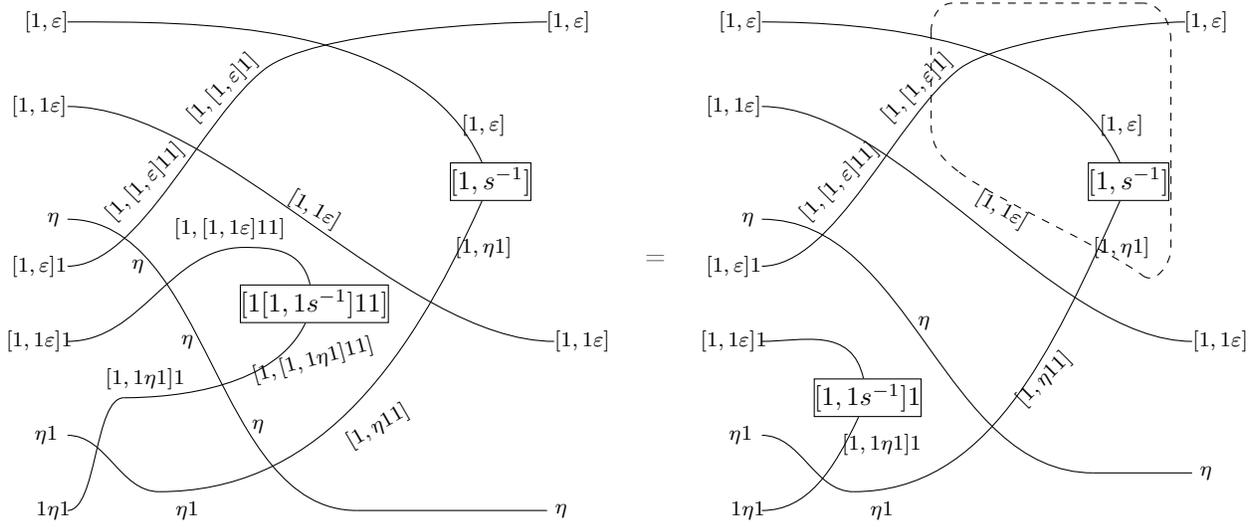
\begin{center}
%

\end{center}
Now in the highlighted part we recognize the possibility to apply the swallowtail equation \eqref{swallowtail2}. This gives eventually
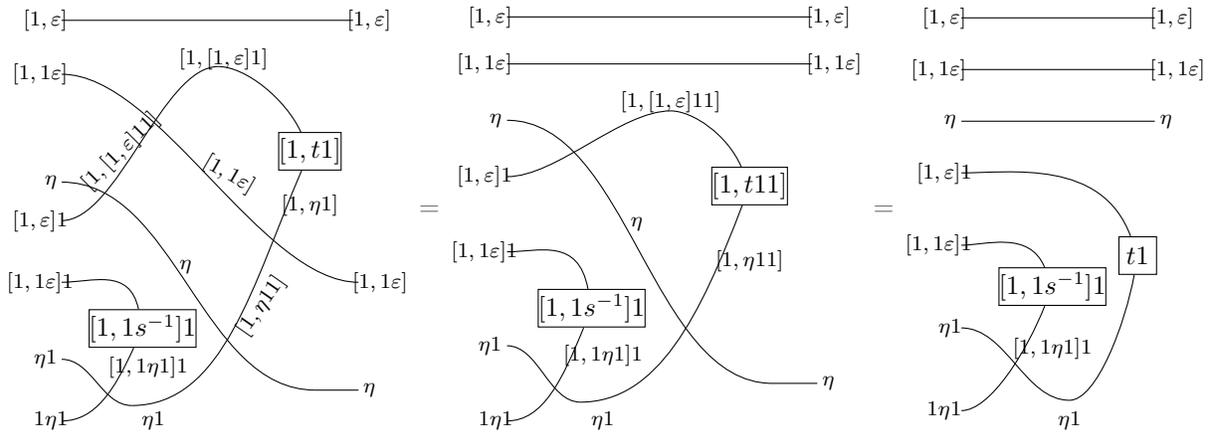
\begin{center}
%
\end{center}
which is the desired right-hand side of \eqref{IC Vbicat}.

Let us now come to the associativity coherence. The left-hand side is:
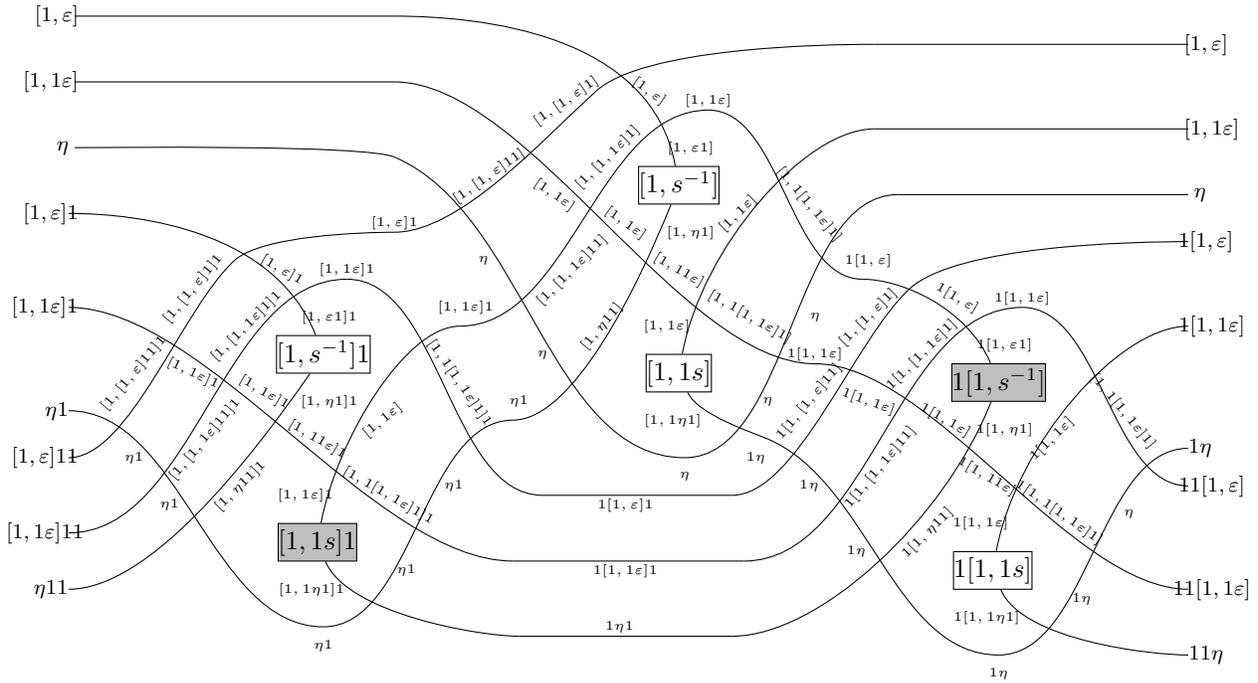
\begin{center}
%

\end{center}
The plan of the proof is to use the usual rules of naturality and modification (without need of swallowtail equations) in order to cancel out the two highlighted 2-cells. That means, we can slide the two highlighted 2-cells at the bottom of the diagram, and they will turn out to be the same one inverse to the other. Once this is done, one can appreciate how a rearrangement of what remains in the diagram provides yet the right-hand side of the associativity coherence:
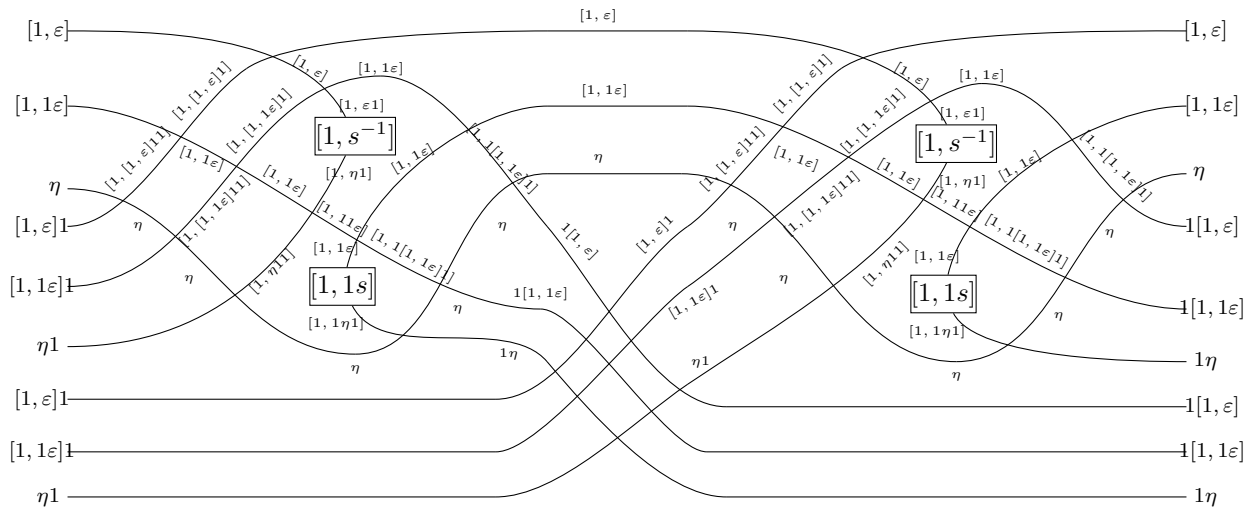
\begin{center}
%
\end{center}
The precise sequence of steps leading to it is left as a zen exercise to the willing reader.
\end{proof}

\subsection{The enriched hom-pseudofunctor}\label{subsect hom pseudo}

With the ingredients of the above three sections, namely the opposite $\cat V$-bicategory, the tensor product of two $\cat V$-bicategories and the self-enrichment of $\cat V$, we are able to define hom-pseudofunctors
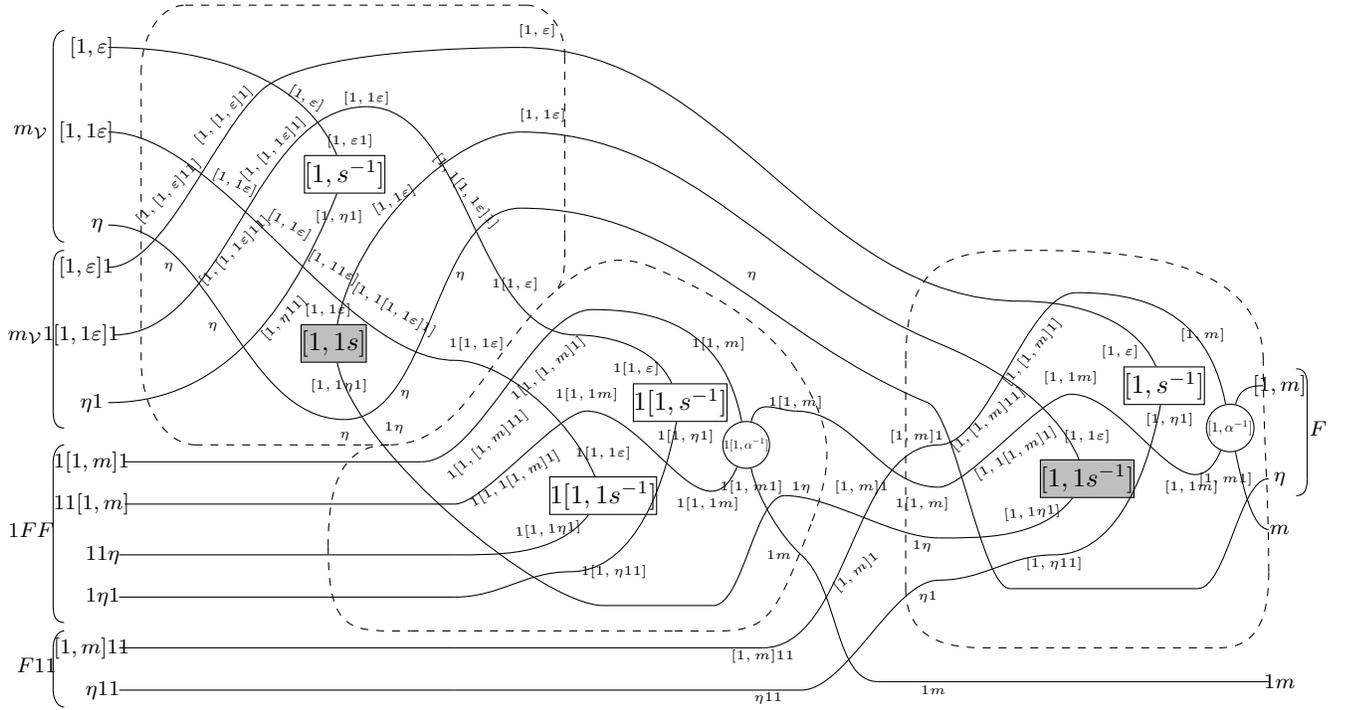
\[\cat C(-,-)\colon\cat C\op\otimes\cat C\longrightarrow\cat V\]
for every $\cat V$-bicategory $\cat C$ over a right closed braided monoidal $\cat V$. In this section we are going to provide the $\cat V$-pseudofunctor structure to $\cat C(-,-)$. To be honest, we are going to limit ourselves to providing the structure structure and the idea of a proof for $\cat C(a,-)\colon\cat C\to\cat V$, which is defined on hom-objects as
\begin{equation}\label{hom pseudofunctor}
\cat C(b,c)\overset{\eta}\longrightarrow[\cat C(a,b),\cat C(b,c)\cat C(a,b)]\overset{[1,m]}\longrightarrow[\cat C(a,b),\cat C(a,c)].
\end{equation}
Then, the case $\cat C(-,b)\colon\cat C\op\to\cat V$ is very similar to treat, though instances of $\beta$ appear, both in the definition of the multiplication for $\cat C\op$, and in the definition of the pseudofunctor on hom-objects:
\[\cat C(-,d)\colon \cat C(b,c)\overset{\eta}\longrightarrow[\cat C(c,d),\cat C(b,c)\cat C(c,d)]\overset{[1,\beta]}\longrightarrow[\cat C(c,d),\cat C(c,d)\cat C(b,c)]\overset{[1,m]}\longrightarrow[\cat C(c,d),\cat C(b,d)].\]
The more general case proving that $\cat C(-,-)$ is a $\cat V$-pseudofunctor is just incredibly more space-consuming, but not conceptually different, and can moreover be deduced from the existence of $\cat V$-pseudofunctors in each variable by generalizing an easy argument which can be found at \cite{closedcats}, Chapter III, Section 4.

\begin{thm}
Let $\cat C$ be a $\cat V$-bicategory, for $\cat V$ a right closed braided monoidal bicategory. Then, for every object $a$ in $\cat C$ the morphism \eqref{hom pseudofunctor}
is part of a $\cat V$-pseudofunctor $\cat C\to\cat V$.
\end{thm}

\begin{proof}
One has to define the higher structure of $F=\cat C(a,-)$, namely the two 2-cells \eqref{fun e un} of Definition \ref{def pseudofunctor}, and check the axioms. The two are given as

\begin{minipage}{.05\textwidth}
\begin{equation}\label{fun hom}
\phantom{c}
\end{equation}
\end{minipage}
\begin{minipage}{.95\textwidth}
\begin{center}
\begin{tikzpicture}[scale=.5]
		\node [style=none] (1) at (-4, 2.75) {};
		\node [style=none] (2) at (-4, 1.5) {};
		\node [style=none] (3) at (-4, 0.25) {};
		\node [style=none] (4) at (-4, -1) {};
		\node [style=none] (5) at (-4, -2.25) {};
		\node [style=none] (6) at (-4.75, 2.75) {$[1,\varepsilon]$};
		\node [style=none] (7) at (-4.75, 1.5) {$[1,1\varepsilon]$};
		\node [style=none] (8) at (-4.75, 0.25) {$\eta$};
		\node [style=none] (9) at (-5, -1) {$[1,m]1$};
		\node [style=none] (10) at (-5, -2.25) {$1[1,m]$};
		\node [style=none] (11) at (-4, -3.75) {};
		\node [style=none] (12) at (-4, -5) {};
		\node [style=none] (13) at (-4.75, -3.75) {$1\eta$};
		\node [style=none] (14) at (-4.75, -5) {$\eta 1$};
		\node [style=none] (15) at (-5.5, -0.5) {};
		\node [style=none] (16) at (-6, -1) {};
		\node [style=none] (17) at (-6, -5) {};
		\node [style=none] (18) at (-5.5, -5.5) {};
		\node [style=none] (19) at (-5.5, 3.25) {};
		\node [style=none] (20) at (-6, 2.75) {};
		\node [style=none] (21) at (-6, 0) {};
		\node [style=none] (22) at (-5.5, -0.5) {};
		\node [style=none] (23) at (-7, 1) {$m_\mathcal V$};
		\node [style=none] (24) at (-7, -2.75) {$FF$};
		\node [style=squared] (25) at (2.75, -2) {$[1,1s^{-1}]$};
		\node [style=squared] (26) at (6.25, 0.75) {$[1,s^{-1}]$};
		\node [style=none] (27) at (-0.75, -5.25) {};
		\node [style=none] (28) at (0.25, -3) {\rotatebox{5}{\tiny $[1,1\eta 1]$}};
		\node [style=none] (29) at (2.75, -0.75) {{\tiny $[1,1\varepsilon]$}};
		\node [style=none] (31) at (2, 0.5) {};
		\node [style=none] (32) at (-1.75, -0.25) {\rotatebox{40}{\tiny $[1,[1,m]11]$}};
		\node [style=none] (33) at (-0.5, -1.25) {\rotatebox{30}{\tiny $[1,1[1,m]1]$}};
		\node [style=none] (34) at (1, 2.75) {};
		\node [style=none] (35) at (2.25, 3.5) {};
		\node [style=none] (36) at (1.25, -4.25) {};
		\node [style=none] (37) at (2.25, 3.5) {};
		\node [style=fixed size node] (38) at (9.25, -0.5) {\adjustbox{valign=m, max width=.6cm, max height=.6cm}{$[1,\alpha^{-1}]$}};
		\node [style=none] (39) at (0.25, 1.75) {\rotatebox{40}{\tiny $[1,[1,m]1]$}};
		\node [style=none] (40) at (2.25, 0.9) {\tiny $[1,1m]$};
		\node [style=none] (41) at (10.75, 0.75) {};
		\node [style=none] (42) at (11, -3.5) {};
		\node [style=none] (43) at (11, -1.75) {};
		\node [style=none] (44) at (6, -5.25) {};
		\node [style=none] (45) at (1.25, -4.5) {\tiny $[1,\eta11]$};
		\node [style=none] (46) at (6.5, -0.25) {\tiny $[1,\eta1]$};
		\node [style=none] (47) at (7, -2) {\tiny $[1,1m]$};
		\node [style=none] (48) at (4.25, 1.75) {\tiny $[1,\varepsilon]$};
		\node [style=none] (49) at (8.25, 2.25) {\tiny $[1,m]$};
		\node [style=none] (50) at (11.5, 0.75) {$[1,m]$};
		\node [style=none] (51) at (11.5, -1.75) {$\eta$};
		\node [style=none] (52) at (11.5, -3.5) {$m$};
		\node [style=none] (53) at (11.75, -2.25) {};
		\node [style=none] (54) at (12.25, -1.75) {};
		\node [style=none] (55) at (12.25, 0.75) {};
		\node [style=none] (56) at (11.75, 1.25) {};
		\node [style=none] (57) at (12.75, -0.5) {$F$};
		\node [style=none] (58) at (9, -2) {\tiny $[1,m1]$};
		
		\draw [in=90, out=-180, looseness=1.25] (19.center) to (20.center);
		\draw (20.center) to (21.center);
		\draw [in=180, out=-90, looseness=1.25] (21.center) to (22.center);
		\draw [in=-270, out=180, looseness=1.25] (22.center) to (16.center);
		\draw (16.center) to (17.center);
		\draw [in=180, out=-90, looseness=1.25] (17.center) to (18.center);
		\draw [in=-135, out=0, looseness=0.75] (11.center) to (25);
		\draw [in=-180, out=0, looseness=0.50] (3.center) to (27.center);
		\draw [in=0, out=120, looseness=0.75] (25) to (2.center);
		\draw [in=0, out=-180, looseness=0.50] (31.center) to (5.center);
		\draw [in=180, out=0] (1.center) to (34.center);
		\draw [in=120, out=0, looseness=0.75] (34.center) to (26);
		\draw [in=180, out=0] (12.center) to (36.center);
		\draw [in=-105, out=0] (36.center) to (26);
		\draw [in=-180, out=0, looseness=0.75] (4.center) to (35.center);
		\draw [in=105, out=0] (37.center) to (38);
		\draw [in=0, out=-120] (38) to (31.center);
		\draw [in=-75, out=-180, looseness=0.75] (42.center) to (38);
		\draw [in=180, out=75] (38) to (41.center);
		\draw (27.center) to (44.center);
		\draw [in=-180, out=0, looseness=0.75] (44.center) to (43.center);
		\draw [in=90, out=0, looseness=1.25] (56.center) to (55.center);
		\draw (55.center) to (54.center);
		\draw [in=0, out=-90, looseness=1.25] (54.center) to (53.center);
\end{tikzpicture}
\end{center}
\end{minipage}

\noindent
and
\begin{center}
\begin{tikzpicture}[scale=.5]
		\node [style=none] (0) at (-4.25, -2.25) {$u_\mathcal V=\eta$};
		\node [style=none] (1) at (-3.25, -2.25) {};
		\node [style=fixed size node] (2) at (-2.25, 0.25) {\adjustbox{valign=m, max width=.6cm, max height=.6cm}{$[1,\lambda^{-1}]$}};
		\node [style=none] (3) at (3, -1.5) {};
		\node [style=none] (4) at (3, 1.5) {};
		\node [style=none] (5) at (3, 0) {};
		\node [style=none] (6) at (4, 1.5) {$[1,m]$};
		\node [style=none] (7) at (3.5, 0) {$\eta$};
		\node [style=none] (8) at (3.5, -1.5) {$u$};
		\node [style=none] (9) at (-2.25, -1) {$[1,u1]$};
		\node [style=none] (10) at (4.25, 2) {};
		\node [style=none] (11) at (4.75, 1.5) {};
		\node [style=none] (12) at (4.75, 0) {};
		\node [style=none] (13) at (4.25, -0.5) {};
		\node [style=none] (14) at (5.25, 0.75) {$F$};

		\draw [in=0, out=180] (5.center) to (1.center);
		\draw [in=-180, out=-60, looseness=0.75] (2) to (3.center);
		\draw [in=-180, out=60, looseness=0.75] (2) to (4.center);
		\draw [in=270, out=0, looseness=1.25] (13.center) to (12.center);
		\draw (12.center) to (11.center);
		\draw [in=360, out=90, looseness=1.25] (11.center) to (10.center);
\end{tikzpicture}
\end{center}
Now, start for example by considering the first axiom for a $\cat V$-pseudofunctor \eqref{psfun ax1}. In the right-hand side we have, in the dashed regions below, $\alpha$ for the $\cat V$-bicategory structure of $\cat V$ (Definition \ref{def V as V bicat}) and the crossing defined in \eqref{fun hom}.
\begin{center}
\begin{tikzpicture}[scale=.45,xscale=.65]
		\node [style=none] (0) at (-4.25, 10.75) {};
		\node [style=none] (1) at (14.75, 2) {};
		\node [style=none] (2) at (14.25, 0.25) {};
		\node [style=none] (3) at (14.75, -1) {};
		\node [style=none] (4) at (14.75, -2.25) {};
		\node [style=none] (5) at (6.25, 4) {\tiny $\eta$};
		\node [style=none] (6) at (13.75, -0.75) {\tiny $[1,m]1$};
		\node [style=none] (7) at (14, -2.75) {\tiny $1[1,m]$};
		\node [style=none] (8) at (14.75, -3.75) {};
		\node [style=none] (9) at (14.75, -5) {};
		\node [style=none] (10) at (14, -4) {\tiny $1\eta$};
		\node [style=none] (11) at (14.25, -5.5) {\tiny $\eta 1$};
		\node [style=squared, fill=gray!50] (12) at (21.5, -2) {$[1,1s^{-1}]$};
		\node [style=squared] (13) at (25, 0.75) {$[1,s^{-1}]$};
		\node [style=none] (14) at (18, -5.25) {};
		\node [style=none] (15) at (19, -3) {\rotatebox{5}{\tiny $[1,1\eta1]$}};
		\node [style=none] (16) at (21.5, -0.75) {{\tiny $[1,1\varepsilon]$}};
		\node [style=none] (17) at (20.75, 0.5) {};
		\node [style=none] (18) at (17, -0.25) {\rotatebox{40}{\tiny $[1,[1,m]11]$}};
		\node [style=none] (19) at (18.25, -1.25) {\rotatebox{30}{\tiny $[1,1[1,m]1]$}};
		\node [style=none] (20) at (18.5, 3.25) {};
		\node [style=none] (21) at (21, 3.5) {};
		\node [style=none] (22) at (20, -4.25) {};
		\node [style=none] (23) at (21, 3.5) {};
		\node [style=fixed size node] (24) at (28, -0.5) {\adjustbox{valign=m, max width=.6cm, max height=.6cm}{$[1,\alpha^{-1}]$}};
		\node [style=none] (25) at (19, 1.75) {\rotatebox{50}{\tiny $[1,[1,m]1]$}};
		\node [style=none] (26) at (20.75, 1) {\tiny $[1,1m]$};
		\node [style=none] (27) at (29.5, 0.75) {};
		\node [style=none] (28) at (29.75, -3.5) {};
		\node [style=none] (29) at (29.75, -2) {};
		\node [style=none] (30) at (26.5, -5.25) {};
		\node [style=none] (31) at (20, -4.5) {\tiny $[1,\eta11]$};
		\node [style=none] (32) at (25.25, -0.25) {\tiny $[1,\eta1]$};
		\node [style=none] (33) at (26.25, -2.25) {\tiny $[1,1m]$};
		\node [style=none] (34) at (23, 1.75) {\tiny $[1,\varepsilon]$};
		\node [style=none] (35) at (26.75, 2.25) {\tiny $[1,m]$};
		\node [style=none] (36) at (30.25, 0.75) {$[1,m]$};
		\node [style=none] (37) at (30.25, -2) {$\eta$};
		\node [style=none] (38) at (30.25, -3.5) {$m$};
		\node [style=none] (39) at (31, -2.5) {};
		\node [style=none] (40) at (31.5, -2) {};
		\node [style=none] (41) at (31.5, 0.75) {};
		\node [style=none] (42) at (31, 1.25) {};
		\node [style=none] (43) at (32, -0.5) {$F$};
		\node [style=none] (44) at (27.8, -2) {\tiny $[1,m1]$};
		\node [style=none] (45) at (-7.25, 1.5) {};
		\node [style=none] (47) at (-8.75, -1.5) {};
		\node [style=none] (48) at (-7.25, -2.75) {};
		\node [style=none] (49) at (-7.25, -4.25) {};
		\node [style=none] (50) at (-7.25, -5.5) {};
		\node [style=squared] (51) at (-0.5, -2.5) {$1[1,1s^{-1}]$};
		\node [style=squared] (52) at (3, 0.25) {$1[1,s^{-1}]$};
		\node [style=none] (53) at (-0.5, -5.75) {};
		\node [style=none] (54) at (-3, -3.5) {\rotatebox{5}{\tiny $1[1,1\eta1]$}};
		\node [style=none] (55) at (-0.5, -1.25) {{\tiny $1[1,1\varepsilon]$}};
		\node [style=none] (56) at (-1.25, 0) {};
		\node [style=none] (57) at (-5.75, -1) {\rotatebox{40}{\tiny $1[1,[1,m]11]$}};
		\node [style=none] (58) at (-4.5, -2) {\rotatebox{30}{\tiny $1[1,1[1,m]1]$}};
		\node [style=none] (59) at (-1.75, 2.25) {};
		\node [style=none] (60) at (-1, 3) {};
		\node [style=none] (61) at (-2, -4.75) {};
		\node [style=none] (62) at (-1, 3) {};
		\node [style=fixed size node] (63) at (6, -1) {\adjustbox{valign=m, max width=.6cm, max height=.6cm}{$1[1,\alpha^{-1}]$}};
		\node [style=none] (64) at (-3, 1.25) {\rotatebox{40}{\tiny $1[1,[1,m]1]$}};
		\node [style=none] (65) at (-1.25, 0.5) {\tiny $1[1,1m]$};
		\node [style=none] (66) at (8.25, 0) {};
		\node [style=none] (67) at (8.5, -4.25) {};
		\node [style=none] (68) at (7.75, -2.5) {};
		\node [style=none] (69) at (4.5, -5.75) {};
		\node [style=none] (70) at (0, -4.75) {\tiny $1[1,\eta11]$};
		\node [style=none] (71) at (3.25, -0.75) {\tiny $1[1,\eta1]$};
		\node [style=none] (72) at (4.25, -2.75) {\tiny $1[1,1m]$};
		\node [style=none] (73) at (1, 1.25) {\tiny $1[1,\varepsilon]$};
		\node [style=none] (74) at (4.75, 2) {\tiny $1[1,m]$};
		\node [style=none] (75) at (6.25, -2.25) {\tiny $1[1,m1]$};
		\node [style=none] (76) at (-7.25, -7) {};
		\node [style=none] (77) at (-7.25, -8.25) {};
		\node [style=none] (78) at (8.25, 0.25) {\tiny $1[1,m]$};
		\node [style=none] (79) at (8.5, -2.25) {\tiny $1\eta$};
		\node [style=none] (80) at (7.5, -4.25) {\tiny $1m$};
		\node [style=none] (81) at (6.75, -7) {};
		\node [style=none] (82) at (8.5, -8.25) {};
		\node [style=none] (83) at (12, -8) {};
		\node [style=none] (84) at (14.75, -1) {};
		\node [style=none] (85) at (14.75, -5) {};
		\node [style=none] (86) at (11, -4.75) {\rotatebox{40}{\tiny $[1,m]1$}};
		\node [style=none] (87) at (14.75, -3.75) {};
		\node [style=none] (88) at (14.5, -8.25) {\tiny $1m$};
		\node [style=none] (89) at (11.25, -2.25) {\tiny $[1,m]1$};
		\node [style=none] (90) at (29.75, -8) {};
		\node [style=none] (91) at (30.25, -8) {$1m$};
		\node [style=none] (92) at (-23.75, 10.75) {$[1,\varepsilon]$};
		\node [style=none] (93) at (-24, 8.25) {$[1,1\varepsilon]$};
		\node [style=none] (94) at (-23.5, 5.5) {$\eta$};
		\node [style=none] (95) at (-24, 4.25) {$[1,\varepsilon]1$};
		\node [style=none] (96) at (-24, 2.25) {$[1,1\varepsilon]1$};
		\node [style=none] (97) at (-23.75, 0.25) {$\eta 1$};
		\node [style=none] (98) at (-23, 10.75) {};
		\node [style=none] (99) at (-23, 8.25) {};
		\node [style=none] (100) at (-23, 5.5) {};
		\node [style=none] (101) at (-23, 4.25) {};
		\node [style=none] (102) at (-22.75, 2.25) {};
		\node [style=none] (103) at (-23, 0.25) {};
		\node [style=none] (104) at (-3.5, 11.25) {\tiny $[1,\varepsilon]$};
		\node [style=none] (105) at (-3.25, 8.75) {\tiny $[1,1\varepsilon]$};
		\node [style=none] (106) at (-4.5, 3.75) {\tiny $1[1,\varepsilon]$};
		\node [style=none] (107) at (-6.25, 2) {\tiny $1[1,1\varepsilon]$};
		\node [style=none] (108) at (-10, -0.5) {\tiny $1\eta $};
		\node [style=none] (109) at (-4.25, 10.75) {};
		\node [style=none] (110) at (-4.25, 8.25) {};
		\node [style=none] (111) at (-4.25, 6) {};
		\node [style=none] (112) at (-1.75, 2.25) {};
		\node [style=none] (113) at (-7.25, 1.5) {};
		\node [style=squared, fill=gray!50] (115) at (-12.75, 2) {$[1,1s]$};
		\node [style=squared] (116) at (-12.25, 7) {$[1,s^{-1}]$};
		\node [style=none] (117) at (-12.25, -0.25) {};
		\node [style=none] (118) at (-16, 9.5) {};
		\node [style=none] (119) at (-11.25, 9) {};
		\node [style=none] (120) at (-20.25, 4.25) {\tiny $\eta$};
		\node [style=none] (121) at (-18.25, 2.5) {\tiny $\eta$};
		\node [style=none] (122) at (-12.25, -0.75) {\tiny $\eta$};
		\node [style=none] (123) at (-9.5, 0.5) {\tiny $\eta$};
		\node [style=none] (124) at (-7, 4) {\tiny $\eta$};
		\node [style=none] (125) at (-12.5, 0.75) {\tiny $[1,1\eta 1]$};
		\node [style=none] (126) at (-13, 3) {\tiny $[1,1\varepsilon]$};
		\node [style=none] (127) at (-10, 6.25) {\rotatebox{30}{\tiny $[1,1\varepsilon]$}};
		\node [style=none] (128) at (-11.25, 9.25) {\tiny $[1,1\varepsilon]$};
		\node [style=none] (129) at (-10, 3) {\rotatebox{-30}{\tiny \tiny $[1,1[1,1\varepsilon]1]$}};
		\node [style=none] (130) at (-15, 2.75) {\rotatebox{45}{\tiny $[1,\eta11]$}};
		\node [style=none] (131) at (-20.25, 6.5) {\rotatebox{45}{\tiny $[1,[1,\varepsilon]11]$}};
		\node [style=none] (132) at (-12.75, 4.25) {\rotatebox{-30}{\tiny $[1,11\varepsilon]$}};
		\node [style=none] (133) at (-14.75, 5.5) {\rotatebox{-30}{\tiny $[1,1\varepsilon]$}};
		\node [style=none] (134) at (-17.25, 6.75) {\rotatebox{-25}{\tiny $[1,1\varepsilon]$}};
		\node [style=none] (135) at (-17.25, 4.75) {\rotatebox{45}{\tiny $[1,[1,1\varepsilon]11]$}};
		\node [style=none] (136) at (-12, 8) {\tiny $\tiny [1,\varepsilon1]$};
		\node [style=none] (137) at (-12.5, 5.75) {\tiny $[1,\eta1]$};
		\node [style=none] (138) at (-15.5, 7.75) {\rotatebox{45}{\tiny $[1,[1,1\varepsilon]1]$}};
		\node [style=none] (139) at (-6.75, 6.5) {\rotatebox{-50}{\tiny $[1, 1[1,1\varepsilon]1]$}};
		\node [style=none] (140) at (-17.75, 8.75) {\rotatebox{45}{\tiny $[1,[1,\varepsilon]1]$}};
		\node [style=none] (141) at (-14, 9.25) {\rotatebox{-30}{\tiny $[1,\varepsilon]$}};
		\node [style=none] (142) at (-22.25, -1.5) {};
		\node [style=none] (143) at (-22.25, -2.75) {};
		\node [style=none] (144) at (-22.5, -4.25) {};
		\node [style=none] (145) at (-22.5, -5.5) {};
		\node [style=none] (146) at (-23.75, -1.5) {$1[1,m]1$};
		\node [style=none] (147) at (-23.75, -2.75) {$11[1,m]$};
		\node [style=none] (148) at (-23.25, -4.25) {$11\eta$};
		\node [style=none] (149) at (-23.25, -5.5) {$1\eta 1$};
		\node [style=none] (150) at (-22.5, -7) {};
		\node [style=none] (151) at (-22.5, -8.25) {};
		\node [style=none] (152) at (-23.25, -8.25) {$\eta 11$};
		\node [style=none] (153) at (-23.75, -7) {$[1,m]11$};
		\node [style=none] (154) at (-25.5, -7) {};
		\node [style=none] (155) at (-25, -6.5) {};
		\node [style=none] (156) at (-25.5, -8.25) {};
		\node [style=none] (157) at (-25, -8.75) {};
		\node [style=none] (158) at (-25.5, -5.75) {};
		\node [style=none] (159) at (-26.25, -7.5) {$F11$};
		\node [style=none] (160) at (-25.5, -1.5) {};
		\node [style=none] (161) at (-25, -1) {};
		\node [style=none] (162) at (-26.5, -3.5) {$1FF$};
		\node [style=none] (163) at (-25, -0.5) {};
		\node [style=none] (164) at (-25.5, 0) {};
		\node [style=none] (165) at (-25.5, 4.25) {};
		\node [style=none] (166) at (-25, 4.75) {};
		\node [style=none] (167) at (-25, 5) {};
		\node [style=none] (168) at (-25.5, 5.5) {};
		\node [style=none] (169) at (-25.5, 10.75) {};
		\node [style=none] (170) at (-25, 11.25) {};
		\node [style=none] (171) at (-26.5, 8.25) {$m_\mathcal V$};
		\node [style=none] (172) at (-26.5, 2.25) {$m_\mathcal V 1$};
		\node [style=none] (173) at (-25, -6.25) {};
		\node [style=none] (174) at (-2.25, 5.25) {};
		\node [style=none] (175) at (-2.75, 3.5) {};
		\node [style=none] (176) at (-10.25, -1) {};
		\node [style=none] (177) at (-11.5, -1) {};
		\node [style=none] (178) at (-19, -1) {};
		\node [style=none] (179) at (-21.5, 1.5) {};
		\node [style=none] (180) at (-21.5, 10) {};
		\node [style=none] (181) at (-19.5, 12) {};
		\node [style=none] (182) at (-4, 12) {};
		\node [style=none] (183) at (-2.25, 10.5) {};
		\node [style=none] (184) at (-13, -4) {};
		\node [style=none] (185) at (-9.25, -6.5) {};
		\node [style=none] (186) at (2.75, -6.5) {};
		\node [style=none] (187) at (9, -3.25) {};
		\node [style=none] (188) at (8.25, 1.25) {};
		\node [style=none] (189) at (1.75, 4.25) {};
		\node [style=none] (190) at (6.75, -7.25) {\tiny $[1,m]11$};
		\node [style=none] (191) at (7, -8.5) {\tiny $\eta 11$};
		\node [style=none] (192) at (19, 4.75) {};
		\node [style=none] (193) at (13.25, 1.5) {};
		\node [style=none] (194) at (13.25, -5.25) {};
		\node [style=none] (195) at (16.75, -7) {};
		\node [style=none] (196) at (27.5, -7) {};
		\node [style=none] (197) at (29.75, -4.75) {};
		\node [style=none] (198) at (29.5, 1.5) {};

		\draw [in=-135, out=0, looseness=0.75] (8.center) to (12);
		\draw [in=-180, out=-15, looseness=0.50] (2.center) to (14.center);
		\draw [in=-15, out=120, looseness=0.75] (12) to (1.center);
		\draw [in=0, out=-180, looseness=0.50] (17.center) to (4.center);
		\draw [in=180, out=0] (0.center) to (20.center);
		\draw [in=120, out=0, looseness=0.75] (20.center) to (13);
		\draw [in=180, out=0] (9.center) to (22.center);
		\draw [in=-105, out=0] (22.center) to (13);
		\draw [in=-180, out=0, looseness=0.75] (3.center) to (21.center);
		\draw [in=105, out=0] (23.center) to (24);
		\draw [in=0, out=-120] (24) to (17.center);
		\draw [in=-75, out=-180, looseness=0.75] (28.center) to (24);
		\draw [in=180, out=75] (24) to (27.center);
		\draw [in=-180, out=0, looseness=0.75] (30.center) to (29.center);
		\draw [in=90, out=0, looseness=1.25] (42.center) to (41.center);
		\draw (41.center) to (40.center);
		\draw [in=0, out=-90, looseness=1.25] (40.center) to (39.center);
		\draw [in=360, out=180] (30.center) to (14.center);
		\draw [in=-135, out=0, looseness=0.75] (49.center) to (51);
		\draw [in=0, out=120, looseness=0.75] (51) to (45.center);
		\draw [in=0, out=-180, looseness=0.50] (56.center) to (48.center);
		\draw [in=120, out=0, looseness=0.75] (59.center) to (52);
		\draw [in=180, out=0] (50.center) to (61.center);
		\draw [in=-105, out=0] (61.center) to (52);
		\draw [in=-180, out=0, looseness=0.75] (47.center) to (60.center);
		\draw [in=105, out=0] (62.center) to (63);
		\draw [in=0, out=-120] (63) to (56.center);
		\draw [in=-75, out=150, looseness=0.75] (67.center) to (63);
		\draw [in=180, out=75] (63) to (66.center);
		\draw [in=-180, out=0, looseness=0.75] (69.center) to (68.center);
		\draw [in=360, out=180] (69.center) to (53.center);
		\draw (81.center) to (76.center);
		\draw (77.center) to (82.center);
		\draw [in=330, out=-180] (83.center) to (67.center);
		\draw [in=180, out=0] (81.center) to (84.center);
		\draw [in=0, out=-180, looseness=0.75] (85.center) to (82.center);
		\draw [in=-180, out=0, looseness=0.75] (68.center) to (87.center);
		\draw [in=0, out=180] (4.center) to (66.center);
		\draw [in=360, out=180] (90.center) to (83.center);
		\draw [in=-180, out=75, looseness=0.75] (115) to (110.center);
		\draw [in=-180, out=0, looseness=0.75] (99.center) to (113.center);
		\draw [in=0, out=-180] (117.center) to (100.center);
		\draw [in=-180, out=0] (117.center) to (111.center);
		\draw [in=0, out=-150, looseness=0.75] (118.center) to (101.center);
		\draw [in=0, out=-120] (116) to (103.center);
		\draw [in=0, out=180] (119.center) to (102.center);
		\draw [in=-180, out=0] (119.center) to (112.center);
		\draw [in=180, out=30, looseness=0.50] (118.center) to (109.center);
		\draw [in=120, out=0, looseness=0.75] (98.center) to (116);
		\draw [in=0, out=165] (1.center) to (110.center);
		\draw [in=360, out=165] (2.center) to (111.center);
		\draw (47.center) to (142.center);
		\draw (143.center) to (48.center);
		\draw (49.center) to (144.center);
		\draw (145.center) to (50.center);
		\draw (76.center) to (150.center);
		\draw (151.center) to (77.center);
		\draw [in=270, out=180] (157.center) to (156.center);
		\draw (156.center) to (154.center);
		\draw [in=180, out=90] (154.center) to (155.center);
		\draw [in=90, out=-180] (161.center) to (160.center);
		\draw (160.center) to (158.center);
		\draw [in=90, out=-180, looseness=1.25] (166.center) to (165.center);
		\draw (165.center) to (164.center);
		\draw [in=180, out=-90, looseness=1.25] (164.center) to (163.center);
		\draw [in=90, out=-180] (170.center) to (169.center);
		\draw (169.center) to (168.center);
		\draw [in=180, out=-90] (168.center) to (167.center);
		\draw [in=180, out=-90, looseness=1.25] (158.center) to (173.center);
		\draw [in=-75, out=180, looseness=0.50] (53.center) to (115);
		\draw [style=dashed] (181.center) to (182.center);
		\draw [style=dashed, in=90, out=0] (182.center) to (183.center);
		\draw [style=dashed] (183.center) to (174.center);
		\draw [style=dashed, in=30, out=-90] (174.center) to (175.center);
		\draw [style=dashed, in=0, out=-150] (175.center) to (176.center);
		\draw [style=dashed, in=0, out=-180] (176.center) to (177.center);
		\draw [style=dashed] (177.center) to (178.center);
		\draw [style=dashed, in=-90, out=-180] (178.center) to (179.center);
		\draw [style=dashed] (179.center) to (180.center);
		\draw [style=dashed, in=180, out=90] (180.center) to (181.center);
		\draw [style=dashed, in=90, out=180, looseness=1.25] (176.center) to (184.center);
		\draw [style=dashed, in=-180, out=-90, looseness=1.50] (184.center) to (185.center);
		\draw [style=dashed, in=180, out=0] (185.center) to (186.center);
		\draw [style=dashed, in=-120, out=0, looseness=1.50] (186.center) to (187.center);
		\draw [style=dashed, in=-45, out=60, looseness=1.25] (187.center) to (188.center);
		\draw [style=dashed, in=-15, out=135, looseness=0.75] (188.center) to (189.center);
		\draw [style=dashed, in=30, out=165] (189.center) to (175.center);
		\draw [style=dashed, in=360, out=90, looseness=0.75] (198.center) to (192.center);
		\draw [style=dashed, in=90, out=-180, looseness=1.25] (192.center) to (193.center);
		\draw [style=dashed] (193.center) to (194.center);
		\draw [style=dashed, in=180, out=-90] (194.center) to (195.center);
		\draw [style=dashed] (195.center) to (196.center);
		\draw [style=dashed, in=270, out=0, looseness=1.25] (196.center) to (197.center);
		\draw [style=dashed, in=270, out=90] (197.center) to (198.center);
\end{tikzpicture}
\end{center}
If we stare at the morphism above, we can argue that up to adjustments of naturality and modification property, we can cancel out the two 2-cells highlighted in light gray. Now let us first focus on the squared 2-cells, coming from the structure of pseudoadjunction. Once we remove the highlighted ones, we are able to observe how these 2-cells are in bijective correspondence with the analogous 2-cells in the left-hand side of \eqref{psfun ax1} developed below. Most importantly, we can make this correspondence consistent with all the strings involved. This means that with no further application of the swallowtail equations we can argue how the usual naturality and modification steps bring one to the other.

Let us expand such left-hand side, where we recognize the same two crossings \eqref{fun hom}, followed by $\alpha$ for $\cat C$:
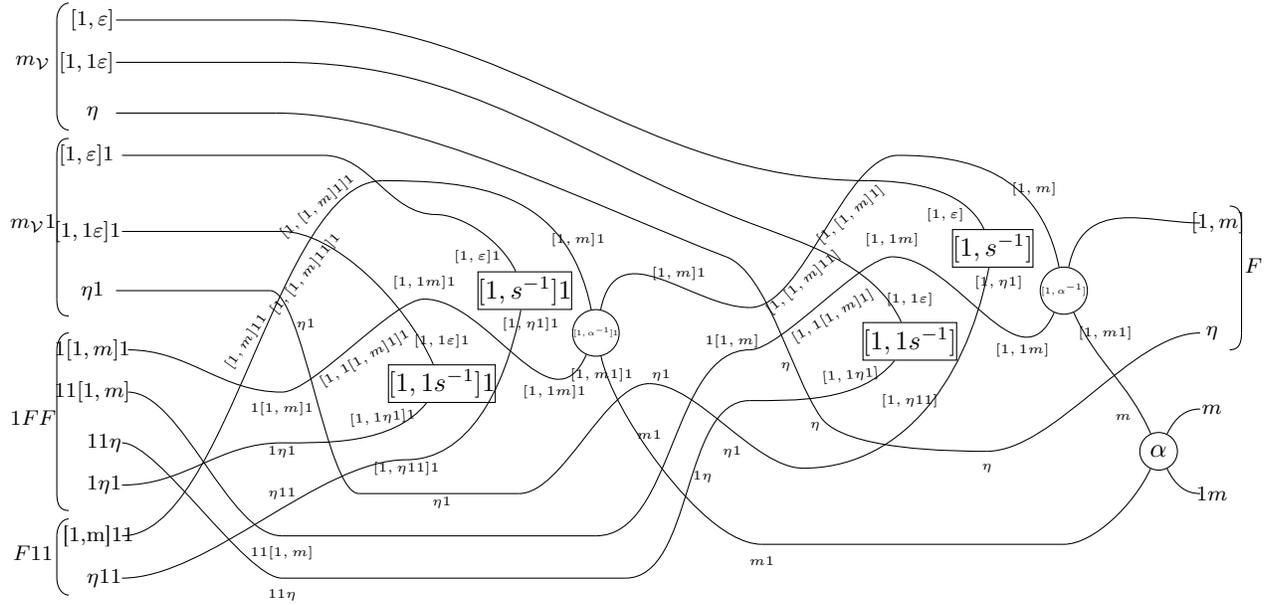
\begin{center}
\begin{tikzpicture}[scale=.45,xscale=.7]
		\node [style=none] (0) at (-41.5, 10) {$[1,\varepsilon]$};
		\node [style=none] (1) at (-41.75, 8.75) {$[1,1\varepsilon]$};
		\node [style=none] (2) at (-41.5, 7.25) {$\eta$};
		\node [style=none] (3) at (-41.75, 6) {$[1,\varepsilon]1$};
		\node [style=none] (4) at (-41.7, 3.75) {$[1,1\varepsilon]1$};
		\node [style=none] (5) at (-41.5, 2) {$\eta 1$};
		\node [style=none] (6) at (-33.5, 8.75) {};
		\node [style=none] (7) at (-33.75, 7.25) {};
		\node [style=none] (8) at (-33.5, 6) {};
		\node [style=none] (9) at (-33.5, 3.75) {};
		\node [style=none] (10) at (-34, 2) {};
		\node [style=none] (11) at (-33.5, 0.25) {};
		\node [style=none] (12) at (-33.5, -1) {};
		\node [style=none] (13) at (-33.5, -2.5) {};
		\node [style=none] (15) at (-35, 0.5) {\rotatebox{55}{\tiny $[1,m]11$}};
		\node [style=none] (16) at (-33.5, -1.5) {\tiny $1[1,m]1$};
		\node [style=none] (17) at (-33.5, -2.75) {\tiny $1\eta1$};
		\node [style=none] (18) at (-33.5, -4) {\tiny $\eta 11$};
		\node [style=none] (19) at (-33.5, -5.25) {};
		\node [style=none] (20) at (-33.5, -6.5) {};
		\node [style=none] (21) at (-33.5, -7) {\tiny $11\eta $};
		\node [style=none] (22) at (-33.5, -5.75) {\tiny $11[1,m]$};
		\node [style=none] (23) at (-42.5, 1.25) {};
		\node [style=none] (24) at (-43, 1.75) {};
		\node [style=none] (25) at (-43, 6) {};
		\node [style=none] (26) at (-42.5, 6.5) {};
		\node [style=none] (27) at (-42.5, 6.75) {};
		\node [style=none] (28) at (-43, 7.25) {};
		\node [style=none] (29) at (-43, 10) {};
		\node [style=none] (30) at (-42.5, 10.5) {};
		\node [style=none] (31) at (-44, 8.75) {$m_\mathcal V$};
		\node [style=none] (32) at (-44, 4) {$m_\mathcal V 1$};
		\node [style=none] (33) at (-31.75, 6) {};
		\node [style=none] (34) at (-33.5, 3.75) {};
		\node [style=none] (35) at (-34, 2) {};
		\node [style=none] (36) at (-29.25, 5.25) {};
		\node [style=none] (37) at (-33.5, -1) {};
		\node [style=none] (38) at (-33.5, -2.5) {};
		\node [style=none] (39) at (-28.25, -3) {};
		\node [style=squared] (40) at (-26.75, -0.75) {$[1,1s^{-1}]1$};
		\node [style=squared] (41) at (-23.25, 2) {$[1,s^{-1}]1$};
		\node [style=none] (42) at (-30.25, -4) {};
		\node [style=none] (43) at (-29.25, -1.75) {\rotatebox{5}{\tiny $[1,1\eta1]1$}};
		\node [style=none] (44) at (-26.75, 0.5) {{\tiny $[1,1\varepsilon]1$}};
		\node [style=none] (45) at (-27.5, 1.75) {};
		\node [style=none] (46) at (-32.5, 2.5) {\rotatebox{50}{\tiny $[1,[1,m]11]1$}};
		\node [style=none] (47) at (-30, 0) {\rotatebox{30}{\tiny $[1,1[1,m]1]1$}};
		\node [style=none] (48) at (-27, 4.25) {};
		\node [style=none] (49) at (-28.25, -3) {};
		\node [style=none] (50) at (-29.25, 5.25) {};
		\node [style=fixed size node] (51) at (-20.25, 0.75) {\adjustbox{valign=m, max width=.6cm, max height=.6cm}{\tiny $[1,\alpha^{-1}]1$}};
		\node [style=none] (52) at (-32, 4.5) {\rotatebox{40}{\tiny $[1,[1,m]1]1$}};
		\node [style=none] (53) at (-27.5, 2.25) {\tiny $[1,1m]1$};
		\node [style=none] (54) at (-13.75, 1.5) {};
		\node [style=none] (55) at (-18, -2.75) {};
		\node [style=none] (56) at (-18, -0.75) {};
		\node [style=none] (57) at (-23.5, -4) {};
		\node [style=none] (58) at (-28.25, -3.25) {\tiny $[1,\eta11]1$};
		\node [style=none] (59) at (-23, 1) {\tiny $[1,\eta1]1$};
		\node [style=none] (60) at (-22, -1) {\tiny $[1,1m]1$};
		\node [style=none] (61) at (-25.25, 3) {\tiny $[1,\varepsilon]1$};
		\node [style=none] (62) at (-21, 3.5) {\tiny $[1,m]1$};
		\node [style=none] (63) at (-17.5, -0.5) {\tiny $\eta1$};
		\node [style=none] (64) at (-18, -2.25) {\tiny $m1$};
		\node [style=none] (65) at (-20, -0.5) {\tiny $[1,m1]1$};
		\node [style=none] (66) at (-20.25, -5.25) {};
		\node [style=none] (67) at (-19, -6.5) {};
		\node [style=none] (68) at (-13.25, -5.5) {};
		\node [style=none] (69) at (-13.75, 0.25) {};
		\node [style=none] (70) at (-13.75, -1.25) {};
		\node [style=none] (71) at (-15.75, -3.5) {\tiny $1\eta$};
		\node [style=none] (72) at (-13.25, -6) {\tiny $m1$};
		\node [style=none] (73) at (-33.5, 10) {};
		\node [style=none] (74) at (-13.75, 4) {};
		\node [style=none] (75) at (-14.75, 3) {};
		\node [style=none] (76) at (-13.75, 1.5) {};
		\node [style=none] (77) at (-13.75, 0.25) {};
		\node [style=none] (78) at (-16.75, 2.5) {\tiny $[1,m]1$};
		\node [style=none] (79) at (-14.5, 0.5) {\tiny $1[1,m]$};
		\node [style=none] (80) at (-13.75, -1.25) {};
		\node [style=none] (81) at (-11.5, -3.25) {};
		\node [style=none] (82) at (-14.5, -2.75) {\tiny $\eta 1$};
		\node [style=squared] (83) at (-7, 0.5) {$[1,1s^{-1}]$};
		\node [style=squared] (84) at (-3.5, 3.25) {$[1,s^{-1}]$};
		\node [style=none] (85) at (-10.75, -1.75) {};
		\node [style=none] (86) at (-9.5, -0.5) {\rotatebox{5}{\tiny $[1,1\eta1]$}};
		\node [style=none] (87) at (-7, 1.75) {{\tiny $[1,1\varepsilon]$}};
		\node [style=none] (88) at (-7.75, 3) {};
		\node [style=none] (89) at (-11.5, 2.25) {\rotatebox{40}{\tiny $[1,[1,m]11]$}};
		\node [style=none] (90) at (-10.25, 1.25) {\rotatebox{30}{\tiny $[1,1[1,m]1]$}};
		\node [style=none] (91) at (-8.75, 5.25) {};
		\node [style=none] (92) at (-7.5, 6) {};
		\node [style=none] (94) at (-7.5, 6) {};
		\node [style=fixed size node] (95) at (-0.5, 2) {\adjustbox{valign=m, max width=.6cm, max height=.6cm}{\tiny $[1,\alpha^{-1}]$}};
		\node [style=none] (96) at (-9.5, 4.25) {\rotatebox{40}{\tiny $[1,[1,m]1]$}};
		\node [style=none] (97) at (-7.75, 3.5) {\tiny $[1,1m]$};
		\node [style=none] (98) at (5.25, 4) {};
		\node [style=none] (99) at (5.25, 0.75) {};
		\node [style=none] (100) at (-3.75, -2.75) {};
		\node [style=none] (101) at (-7, -1.25) {\tiny $[1,\eta11]$};
		\node [style=none] (102) at (-3.25, 2.25) {\tiny $[1,\eta1]$};
		\node [style=none] (103) at (-2.25, 0.25) {\tiny $[1,1m]$};
		\node [style=none] (104) at (-5.5, 4.25) {\tiny $[1,\varepsilon]$};
		\node [style=none] (105) at (-1.75, 5) {\tiny $[1,m]$};
		\node [style=none] (106) at (6, 4) {$[1,m]$};
		\node [style=none] (107) at (5.75, 0.75) {$\eta$};
		\node [style=none] (108) at (2, -1.75) {\tiny $m$};
		\node [style=none] (109) at (1.25, 0.75) {\tiny $[1,m1]$};
		\node [style=fixed size node] (110) at (3.5, -2.75) {$\alpha$};
		\node [style=none] (111) at (-0.5, -5.5) {};
		\node [style=none] (112) at (5.25, -1.5) {};
		\node [style=none] (113) at (5.25, -4) {};
		\node [style=none] (114) at (5.75, -1.5) {$m$};
		\node [style=none] (115) at (5.75, -4) {$1m$};
		\node [style=none] (116) at (-40, 0.25) {};
		\node [style=none] (117) at (-40, -1) {};
		\node [style=none] (118) at (-40.25, -2.5) {};
		\node [style=none] (119) at (-40.25, -3.75) {};
		\node [style=none] (120) at (-41.5, 0.25) {$1[1,m]1$};
		\node [style=none] (121) at (-41.5, -1) {$11[1,m]$};
		\node [style=none] (122) at (-41, -2.5) {$11\eta$};
		\node [style=none] (123) at (-41, -3.75) {$1\eta 1$};
		\node [style=none] (124) at (-40.25, -5.25) {};
		\node [style=none] (125) at (-40.25, -6.5) {};
		\node [style=none] (126) at (-41, -6.5) {$\eta 11$};
		\node [style=none] (127) at (-41.25, -5.25) {[1,m]11};
		\node [style=none] (128) at (-40.5, 2) {};
		\node [style=none] (129) at (-40.25, 3.75) {};
		\node [style=none] (130) at (-40.25, 6) {};
		\node [style=none] (131) at (-40.5, 7.25) {};
		\node [style=none] (132) at (-40.5, 8.75) {};
		\node [style=none] (133) at (-40.5, 10) {};
		\node [style=none] (134) at (-42.25, 10.5) {};
		\node [style=none] (135) at (-43, -5.25) {};
		\node [style=none] (136) at (-42.5, -4.75) {};
		\node [style=none] (137) at (-43, -6.5) {};
		\node [style=none] (138) at (-42.5, -7) {};
		\node [style=none] (139) at (-43, -4) {};
		\node [style=none] (140) at (-44, -5.75) {$F11$};
		\node [style=none] (141) at (-43, 0.25) {};
		\node [style=none] (142) at (-42.5, 0.75) {};
		\node [style=none] (143) at (-44, -1.75) {$1FF$};
		\node [style=none] (144) at (-42.5, -4.5) {};
		\node [style=none] (146) at (6.5, 0.25) {};
		\node [style=none] (147) at (7, 0.75) {};
		\node [style=none] (148) at (7, 4) {};
		\node [style=none] (149) at (6.5, 4.5) {};
		\node [style=none] (150) at (7.5, 2.75) {$F$};
		\node [style=none] (151) at (-32.5, 1) {\tiny $\eta1$};
		\node [style=none] (152) at (-26.75, -4.25) {\tiny $\eta1$};
		\node [style=none] (153) at (-11, -2) {\tiny $\eta$};
		\node [style=none] (154) at (-3.75, -3.25) {\tiny $\eta$};
		\node [style=none] (155) at (-12.25, -0.25) {\tiny $\eta$};
		
		\draw [in=90, out=-180, looseness=1.25] (26.center) to (25.center);
		\draw (25.center) to (24.center);
		\draw [in=180, out=-90, looseness=1.25] (24.center) to (23.center);
		\draw [in=90, out=-180] (30.center) to (29.center);
		\draw (29.center) to (28.center);
		\draw [in=180, out=-90] (28.center) to (27.center);
		\draw [in=-135, out=0, looseness=0.75] (38.center) to (40);
		\draw [in=-180, out=0, looseness=0.50] (35.center) to (42.center);
		\draw [in=0, out=120, looseness=0.75] (40) to (34.center);
		\draw [in=0, out=-180, looseness=0.50] (45.center) to (37.center);
		\draw [in=180, out=0] (33.center) to (48.center);
		\draw [in=120, out=0, looseness=0.75] (48.center) to (41);
		\draw [in=180, out=0] (39.center) to (49.center);
		\draw [in=-105, out=0] (49.center) to (41);
		\draw [in=105, out=0] (50.center) to (51);
		\draw [in=0, out=-120] (51) to (45.center);
		\draw [in=-75, out=135, looseness=0.75] (55.center) to (51);
		\draw [in=180, out=75] (51) to (54.center);
		\draw (42.center) to (57.center);
		\draw [in=-180, out=0, looseness=0.75] (57.center) to (56.center);
		\draw (67.center) to (20.center);
		\draw (19.center) to (66.center);
		\draw [in=-180, out=0] (66.center) to (69.center);
		\draw [in=-180, out=0] (67.center) to (70.center);
		\draw [in=-45, out=-180, looseness=0.75] (68.center) to (55.center);
		\draw [in=-135, out=0, looseness=0.75] (80.center) to (83);
		\draw [in=135, out=-15, looseness=0.75] (75.center) to (85.center);
		\draw [in=-15, out=120, looseness=0.75] (83) to (74.center);
		\draw [in=0, out=-180, looseness=0.50] (88.center) to (77.center);
		\draw [in=180, out=0] (73.center) to (91.center);
		\draw [in=120, out=0, looseness=0.75] (91.center) to (84);
		\draw [in=-180, out=0, looseness=0.75] (76.center) to (92.center);
		\draw [in=105, out=0] (94.center) to (95);
		\draw [in=0, out=-120] (95) to (88.center);
		\draw [in=180, out=75] (95) to (98.center);
		\draw [in=-180, out=-45, looseness=0.50] (85.center) to (100.center);
		\draw [in=-180, out=0, looseness=0.75] (100.center) to (99.center);
		\draw [in=180, out=0, looseness=0.75] (56.center) to (81.center);
		\draw [in=0, out=165, looseness=0.75] (75.center) to (7.center);
		\draw [in=165, out=0] (6.center) to (74.center);
		\draw [in=300, out=120] (110) to (95);
		\draw [in=0, out=-120, looseness=0.75] (110) to (111.center);
		\draw (68.center) to (111.center);
		\draw [in=60, out=180] (112.center) to (110);
		\draw [in=180, out=-60] (110) to (113.center);
		\draw [in=0, out=-180, looseness=0.50] (20.center) to (118.center);
		\draw [in=0, out=-180, looseness=0.75] (39.center) to (125.center);
		\draw [in=180, out=0, looseness=0.75] (124.center) to (36.center);
		\draw [in=180, out=0] (119.center) to (38.center);
		\draw [in=0, out=180, looseness=0.75] (19.center) to (117.center);
		\draw [in=360, out=-180] (37.center) to (116.center);
		\draw (133.center) to (73.center);
		\draw (6.center) to (132.center);
		\draw (131.center) to (7.center);
		\draw (33.center) to (130.center);
		\draw (129.center) to (34.center);
		\draw (35.center) to (128.center);
		\draw [in=270, out=180] (138.center) to (137.center);
		\draw (137.center) to (135.center);
		\draw [in=180, out=90] (135.center) to (136.center);
		\draw [in=90, out=-180] (142.center) to (141.center);
		\draw (141.center) to (139.center);
		\draw [in=180, out=-90, looseness=1.25] (139.center) to (144.center);
		\draw [in=90, out=0, looseness=1.25] (149.center) to (148.center);
		\draw (148.center) to (147.center);
		\draw [in=0, out=-90, looseness=1.25] (147.center) to (146.center);
		\draw [in=-105, out=0] (81.center) to (84);
\end{tikzpicture}
\end{center}
Let us now focus eventually on the three instances of $\alpha$ in the diagram above. We can slide all of these instances on the top-right part of the diagram, and the same can be done with the instances of $\alpha$ in the right-hand side. If one writes down the results of these shifts, as in \eqref{naturality eta} and \eqref{naturality counit}, it is immediate to observe that the equality of the two sides can be proved by axiom \eqref{AC Vbicat}, in its form
\begin{center}
\begin{tikzpicture}[scale=.5]
		\node [style=none] (0) at (-2, 0) {};
		\node [style=none] (1) at (-2, 2) {};
		\node [style=none] (2) at (-2, -2) {};
		\node [style=fixed size node] (3) at (-3.5, -1) {\adjustbox{valign=m, max width=.6cm, max height=.6cm}{$\alpha 1$}};
		\node [style=fixed size node] (4) at (-5, 1) {\adjustbox{valign=m, max width=.6cm, max height=.6cm}{$\alpha^{-1}$}};
		\node [style=fixed size node] (5) at (-7, 1) {\adjustbox{valign=m, max width=.6cm, max height=.6cm}{$\alpha^{-1}$}};
		\node [style=none] (6) at (-8.5, 0) {};
		\node [style=none] (7) at (-8.5, 2) {};
		\node [style=none] (8) at (-8.5, -2) {};
		\node [style=none] (9) at (2, 2) {};
		\node [style=none] (10) at (2, 0) {};
		\node [style=none] (11) at (2, -2) {};
		\node [style=none] (12) at (8.5, 2) {};
		\node [style=none] (13) at (8.5, 0) {};
		\node [style=none] (14) at (8.5, -2) {};
		\node [style=fixed size node] (15) at (4, -1) {\adjustbox{valign=m, max width=.6cm, max height=.6cm}{$1\alpha^{-1}$}};
		\node [style=fixed size node] (16) at (6.5, 1) {\adjustbox{valign=m, max width=.6cm, max height=.6cm}{$\alpha^{-1}$}};
		\node [style=none] (17) at (-1.5, 2) {$m$};
		\node [style=none] (18) at (-1.5, 0) {$m1$};
		\node [style=none] (19) at (-1.25, -2) {$1m1$};
		\node [style=none] (20) at (-4.5, -2.25) {$m11$};
		\node [style=none] (21) at (-6.75, -0.25) {$m1$};
		\node [style=none] (22) at (-9, 0) {$1m$};
		\node [style=none] (23) at (-9.25, -2) {$11m$};
		\node [style=none] (24) at (-9, 2) {$m$};
		\node [style=none] (25) at (1.5, 0) {$1m$};
		\node [style=none] (26) at (1.25, -2) {$11m$};
		\node [style=none] (27) at (1.5, 2) {$m$};
		\node [style=none] (28) at (9, 2) {$m$};
		\node [style=none] (29) at (9, 0) {$m1$};
		\node [style=none] (30) at (9.25, -2) {$1m1$};
		\node [style=none] (31) at (0, 0) {$=$};
		\node [style=none] (32) at (-6, 2.5) {$m$};
				
		\draw [in=120, out=0] (7.center) to (5);
		\draw [in=0, out=-120] (5) to (6.center);
		\draw [in=105, out=75, looseness=1.25] (5) to (4);
		\draw [in=0, out=-120, looseness=0.75] (4) to (8.center);
		\draw [in=120, out=-60] (4) to (3);
		\draw [in=-180, out=60] (4) to (1.center);
		\draw [in=-120, out=-60, looseness=1.25] (5) to (3);
		\draw [in=180, out=-60] (3) to (2.center);
		\draw [in=180, out=60] (3) to (0.center);
		\draw [in=60, out=-120, looseness=0.75] (16) to (15);
		\draw [in=360, out=120] (15) to (10.center);
		\draw [in=0, out=-120] (15) to (11.center);
		\draw [in=-180, out=-45] (16) to (13.center);
		\draw [in=45, out=-180] (12.center) to (16);
		\draw [in=0, out=135, looseness=0.50] (16) to (9.center);
		\draw [in=180, out=-45, looseness=0.75] (15) to (14.center);
\end{tikzpicture}

\end{center}
(or, more specifically, in the form assumed after applying $[1,-]$ to it).

Let us now come to the unitality axioms for the pseudofunctor $F$. We are only going to prove the first one \eqref{psfun ax2}, as the second one can be shown with a very similar argument. The left-hand side of axiom \eqref{psfun ax2} for our $F$ become
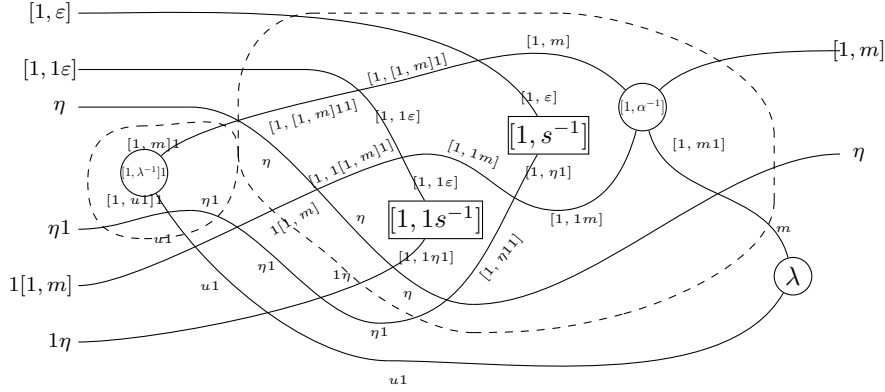
\begin{center}
\begin{tikzpicture}[scale=.5]
		\node [style=none] (0) at (-7, 5) {};
		\node [style=none] (1) at (-1, 3.5) {};
		\node [style=none] (2) at (-7, 2.5) {};
		\node [style=none] (3) at (-7, -0.75) {};
		\node [style=none] (4) at (-7, -2.25) {};
		\node [style=none] (5) at (-7, -3.75) {};
		\node [style=fixed size node] (6) at (-5.25, 0.75) {\adjustbox{valign=m, max width=.6cm, max height=.6cm}{$[1,\lambda^{-1}]1$}};
		\node [style=fixed size node] (7) at (8, 2.5) {\adjustbox{valign=m, max width=.6cm, max height=.6cm}{$[1,\alpha^{-1}]$}};
		\node [style=fixed size node] (8) at (12, -2) {$\lambda$};
		\node [style=squared] (9) at (5.5, 1.75) {$[1,s^{-1}]$};
		\node [style=squared] (10) at (2.5, -0.5) {$[1,1s^{-1}]$};
		\node [style=none] (11) at (13.25, 4) {};
		\node [style=none] (12) at (13.25, 1.25) {};
		\node [style=none] (13) at (1.75, 3.25) {};
		\node [style=none] (14) at (1.25, -4.25) {};
		\node [style=none] (15) at (5.75, -0.25) {};
		\node [style=none] (16) at (2.25, 1.25) {};
		\node [style=none] (17) at (-4, -0.25) {};
		\node [style=none] (18) at (1, -3.25) {};
		\node [style=none] (19) at (-4, 2.5) {};
		\node [style=none] (20) at (3.5, -2.75) {};
		\node [style=none] (21) at (-6.75, 0.5) {};
		\node [style=none] (22) at (-5.25, -1) {};
		\node [style=none] (23) at (-2.75, 1.25) {};
		\node [style=none] (24) at (-5.25, 2) {};
		\node [style=none] (25) at (-2.75, 2.75) {};
		\node [style=none] (26) at (-0.75, 5) {};
		\node [style=none] (27) at (5.75, 5) {};
		\node [style=none] (28) at (11.5, 1.75) {};
		\node [style=none] (29) at (11.5, 0) {};
		\node [style=none] (30) at (7.5, -3) {};
		\node [style=none] (31) at (3.5, -3.5) {};
		\node [style=none] (32) at (0, -1.75) {};
		\node [style=none] (33) at (-7.75, 5) {$[1,\varepsilon]$};
		\node [style=none] (34) at (-7.75, 3.5) {$[1,1\varepsilon]$};
		\node [style=none] (35) at (-7.5, 2.5) {$\eta$};
		\node [style=none] (36) at (-7.5, -0.75) {$\eta 1$};
		\node [style=none] (37) at (-8, -2.25) {$1[1,m]$};
		\node [style=none] (38) at (-7.5, -3.75) {$1\eta$};
		\node [style=none] (39) at (-5.5, 0) {\tiny $[1,u1]1$};
		\node [style=none] (40) at (-3.5, 0) {\tiny $\eta 1$};
		\node [style=none] (41) at (-4.75, -1) {\tiny $u1$};
		\node [style=none] (42) at (-3.5, -2.25) {\tiny $u1$};
		\node [style=none] (43) at (1.5, -4.75) {\tiny $u1$};
		\node [style=none] (44) at (11.75, -0.75) {\tiny $m$};
		\node [style=none] (45) at (-5, 1.5) {\tiny $[1,m]1$};
		\node [style=none] (46) at (0.25, 1) {\rotatebox{25}{\tiny $[1,1[1,m]1]$}};
		\node [style=none] (47) at (-1.25, -0.5) {\rotatebox{30}{\tiny $1[1,m]$}};
		\node [style=none] (48) at (-2, -1.75) {\tiny $\eta 1$};
		\node [style=none] (49) at (1, -3.5) {\tiny $\eta 1$};
		\node [style=none] (50) at (4.25, -1.5) {\rotatebox{50}{\tiny $[1,\eta 11]$}};
		\node [style=none] (51) at (5.5, 0.75) {\tiny $[1,\eta1]$};
		\node [style=none] (52) at (5.25, 2.75) {\tiny $[1,\varepsilon]$};
		\node [style=none] (53) at (1, 3.5) {};
		\node [style=none] (54) at (-7, 3.5) {};
		\node [style=none] (55) at (-0.75, 2.25) {\rotatebox{15}{\tiny $[1,[1,m]11]$}};
		\node [style=none] (56) at (-2, 1) {\tiny $\eta$};
		\node [style=none] (57) at (0.5, -0.5) {\tiny $\eta$};
		\node [style=none] (58) at (1.75, -2.5) {\tiny $\eta$};
		\node [style=none] (59) at (13.75, 1.25) {$\eta$};
		\node [style=none] (60) at (13.75, 4) {$[1,m]$};
		\node [style=none] (61) at (5.5, 4.25) {\tiny $[1,m]$};
		\node [style=none] (62) at (6.25, -0.5) {\tiny $[1,1m]$};
		\node [style=none] (63) at (9.5, 1.5) {\tiny $[1,m1]$};
		\node [style=none] (64) at (2.5, 0.5) {\tiny $[1,1\varepsilon]$};
		\node [style=none] (65) at (1.5, 2.25) {\tiny $[1,1\varepsilon]$};
		\node [style=none] (66) at (1.75, 3.5) {\rotatebox{15}{\tiny $[1,[1,m]1]$}};
		\node [style=none] (68) at (3.5, 1.25) {\rotatebox{-20}{\tiny $[1,1m]$}};
		\node [style=none] (69) at (0, -2) {\tiny $1\eta$};
		\node [style=none] (70) at (2.25, -1.5) {\tiny $[1,1\eta 1]$};
		
		\draw [in=360, out=120, looseness=0.75] (9) to (0.center);
		\draw [in=120, out=0] (1.center) to (10);
		\draw [in=15, out=135] (7) to (13.center);
		\draw [in=45, out=-165, looseness=0.50] (13.center) to (6);
		\draw [in=-60, out=180, looseness=0.75] (14.center) to (6);
		\draw [in=0, out=-120, looseness=0.50] (10) to (5.center);
		\draw [in=-105, out=0] (15.center) to (7);
		\draw [in=0, out=180, looseness=0.50] (16.center) to (4.center);
		\draw [in=-180, out=0] (16.center) to (15.center);
		\draw [in=0, out=-120] (9) to (18.center);
		\draw [in=0, out=-180, looseness=0.75] (18.center) to (17.center);
		\draw [in=180, out=0] (3.center) to (17.center);
		\draw [in=105, out=-75] (7) to (8);
		\draw [in=-180, out=45, looseness=0.75] (7) to (11.center);
		\draw [in=360, out=180] (19.center) to (2.center);
		\draw [in=-120, out=0, looseness=0.75] (14.center) to (8);
		\draw [in=0, out=-180, looseness=0.75] (20.center) to (19.center);
		\draw [in=180, out=0, looseness=0.75] (20.center) to (12.center);
		\draw [style=dashed, in=180, out=-90, looseness=1.25] (21.center) to (22.center);
		\draw [style=dashed, in=-90, out=0, looseness=1.25] (22.center) to (23.center);
		\draw [style=dashed, in=360, out=90, looseness=1.25] (23.center) to (24.center);
		\draw [style=dashed, in=90, out=-180, looseness=1.25] (24.center) to (21.center);
		\draw [style=dashed, in=-90, out=90] (23.center) to (25.center);
		\draw [style=dashed, in=135, out=-90, looseness=0.75] (23.center) to (32.center);
		\draw [style=dashed, in=-180, out=-45, looseness=0.75] (32.center) to (31.center);
		\draw [style=dashed, in=-165, out=0] (31.center) to (30.center);
		\draw [style=dashed, in=-90, out=15, looseness=0.75] (30.center) to (29.center);
		\draw [style=dashed, in=270, out=90] (29.center) to (28.center);
		\draw [style=dashed, in=0, out=90, looseness=0.75] (28.center) to (27.center);
		\draw [style=dashed] (27.center) to (26.center);
		\draw [style=dashed, in=90, out=-180] (26.center) to (25.center);
		\draw [in=180, out=0] (54.center) to (1.center);
\end{tikzpicture}
\end{center}
It still doesn't match the right-hand side, which is just:

\begin{minipage}{.05\textwidth}
\begin{equation}\label{rhs hom psfun}
\phantom{c}
\end{equation}
\end{minipage}
\begin{minipage}{.95\textwidth}
\begin{center}
\begin{tikzpicture}[scale=.5,yscale=.8]
		\node [style=none] (0) at (-1.75, 1) {};
		\node [style=none] (1) at (-1.75, 3) {};
		\node [style=none] (2) at (-1.75, -1) {};
		\node [style=none] (3) at (-2.75, 1) {$ [1,1\varepsilon]$};
		\node [style=none] (4) at (-2.5, 3) {$[1,\varepsilon]$};
		\node [style=none] (5) at (-2.25, -1) {$\eta$};
		\node [style=none] (8) at (-1.75, -3) {};
		\node [style=none] (9) at (-2.5, -3) {$\eta 1$};
		\node [style=squared] (10) at (3.5, 1.75) {$[1,s^{-1}]$};
		\node [style=squared] (11) at (3.5, -1.5) {$t$};
		\node [style=none] (12) at (0.5, -0.5) {};
		\node [style=none] (13) at (1.5, -2.5) {};
		\node [style=none] (14) at (3.75, 0.5) {\tiny $[1,\eta 1]$};
		\node [style=none] (16) at (3.5, -0.5) {\tiny $[1,\varepsilon]$};
		\node [style=none] (17) at (1.5, -2.75) {\tiny $\eta$};
		\node [style=none] (18) at (0, -0.5) {\tiny $[1,\eta11]$};
		\node [style=none] (19) at (-1.75, -4.5) {};
		\node [style=none] (20) at (-1.75, -6) {};
		\node [style=none] (21) at (-2.75, -4.5) {$1[1,m]$};
		\node [style=none] (22) at (-2.5, -6) {$\eta 1$};
		\node [style=none] (23) at (5.5, -4.5) {};
		\node [style=none] (24) at (5.5, -6) {};
		\node [style=none] (25) at (6.5, -4.5) {$1[1,m]$};
		\node [style=none] (26) at (6.25, -6) {$\eta 1$};
		
		\draw [in=120, out=0, looseness=0.75] (1.center) to (10);
		\draw [in=-105, out=45] (12.center) to (10);
		\draw [in=0, out=-135] (12.center) to (8.center);
		\draw [in=120, out=0, looseness=0.75] (0.center) to (11);
		\draw [in=0, out=180] (13.center) to (2.center);
		\draw [in=-120, out=0] (13.center) to (11);
		\draw (19.center) to (23.center);
		\draw (20.center) to (24.center);
\end{tikzpicture}
\end{center}
\end{minipage}
However, with a usual combination of naturality of $\eta$ and modification property for the triangulators, we are able to bring it to the following form:
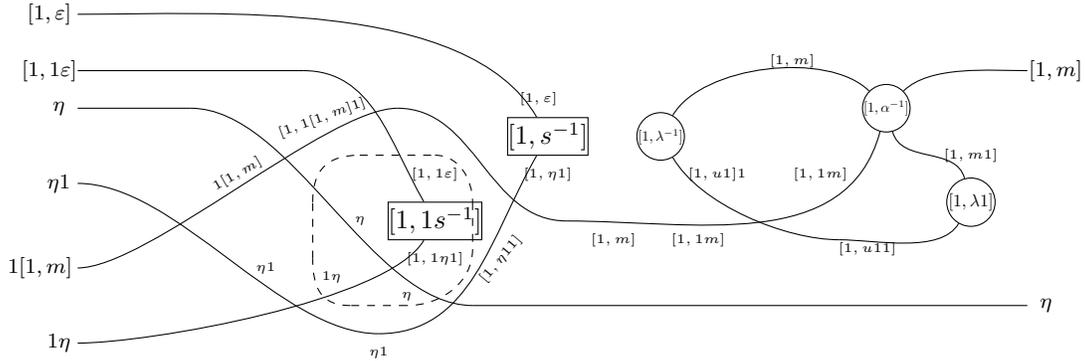
\begin{center}
\begin{tikzpicture}[scale=.5]
		\node [style=none] (0) at (-9.25, 4.75) {};
		\node [style=none] (1) at (-3.25, 3.25) {};
		\node [style=none] (2) at (-9.25, 2.25) {};
		\node [style=none] (3) at (-9.25, 0.25) {};
		\node [style=none] (4) at (-9.25, -2) {};
		\node [style=none] (5) at (-9.25, -4) {};
		\node [style=fixed size node] (6) at (6.25, 1.5) {\adjustbox{valign=m, max width=.6cm, max height=.6cm}{$[1,\lambda^{-1}]$}};
		\node [style=fixed size node] (7) at (12.25, 2.25) {\adjustbox{valign=m, max width=.6cm, max height=.6cm}{$[1,\alpha^{-1}]$}};
		\node [style=fixed size node] (8) at (14.5, -0.25) {\adjustbox{valign=m, max width=.6cm, max height=.6cm}{$[1,\lambda 1]$}};
		\node [style=squared] (9) at (3.25, 1.5) {$[1,s^{-1}]$};
		\node [style=squared] (10) at (0.25, -0.75) {$[1,1s^{-1}]$};
		\node [style=none] (11) at (16, 3.25) {};
		\node [style=none] (12) at (16, -3) {};
		\node [style=none] (14) at (11, -1.25) {};
		\node [style=none] (15) at (3.75, -0.75) {};
		\node [style=none] (16) at (-0.75, 2.25) {};
		\node [style=none] (18) at (-1.25, -3.75) {};
		\node [style=none] (19) at (-6.25, 2.25) {};
		\node [style=none] (20) at (1.25, -3) {};
		\node [style=none] (33) at (-10, 4.75) {$[1,\varepsilon]$};
		\node [style=none] (34) at (-10, 3.25) {$[1,1\varepsilon]$};
		\node [style=none] (35) at (-9.75, 2.25) {$\eta$};
		\node [style=none] (36) at (-9.75, 0.25) {$\eta 1$};
		\node [style=none] (37) at (-10.25, -2) {$1[1,m]$};
		\node [style=none] (38) at (-9.75, -4) {$1\eta$};
		\node [style=none] (39) at (7.75, 0.5) {\tiny $[1,u1]1$};
		\node [style=none] (47) at (-5, 0.5) {\rotatebox{30}{\tiny $1[1,m]$}};
		\node [style=none] (48) at (-4.25, -2) {\tiny $\eta 1$};
		\node [style=none] (49) at (-1.25, -4.25) {\tiny $\eta 1$};
		\node [style=none] (50) at (2, -1.75) {\rotatebox{50}{\tiny $[1,\eta 11]$}};
		\node [style=none] (51) at (3.25, 0.5) {\tiny $[1,\eta1]$};
		\node [style=none] (52) at (3, 2.5) {\tiny $[1,\varepsilon]$};
		\node [style=none] (53) at (-1.25, 3.25) {};
		\node [style=none] (54) at (-9.25, 3.25) {};
		\node [style=none] (57) at (-1.75, -0.75) {\tiny $\eta$};
		\node [style=none] (58) at (-0.5, -2.75) {\tiny $\eta$};
		\node [style=none] (59) at (16.5, -3) {$\eta$};
		\node [style=none] (60) at (16.75, 3.25) {$[1,m]$};
		\node [style=none] (61) at (9.75, 3.5) {\tiny $[1,m]$};
		\node [style=none] (62) at (10.5, 0.5) {\tiny $[1,1m]$};
		\node [style=none] (63) at (14.5, 1) {\tiny $[1,m1]$};
		\node [style=none] (64) at (0.25, 0.5) {\tiny $[1,1\varepsilon]$};
		\node [style=none] (68) at (-2.5, -2.25) {\tiny $1\eta$};
		\node [style=none] (69) at (0.25, -1.75) {\tiny $[1,1\eta 1]$};
		\node [style=none] (71) at (7.25, -1.25) {\tiny $[1,1m]$};
		\node [style=none] (72) at (5, -1.25) {\tiny $[1,m]$};
		\node [style=none] (74) at (-2.75, 2) {\rotatebox{20}{\tiny $[1,1[1,m]1]$}};
		\node [style=none] (75) at (0.25, 1) {};
		\node [style=none] (76) at (-1.75, 1) {};
		\node [style=none] (77) at (-3, -0.25) {};
		\node [style=none] (78) at (-3, -1.75) {};
		\node [style=none] (79) at (-2, -3) {};
		\node [style=none] (80) at (-0.5, -3) {};
		\node [style=none] (81) at (1.25, -1.5) {};
		\node [style=none] (82) at (1.25, -0.25) {};
		\node [style=none] (83) at (11.75, -1.5) {\tiny $[1,u11]$};
		
		\draw [in=360, out=120, looseness=0.75] (9) to (0.center);
		\draw [in=120, out=0] (1.center) to (10);
		\draw [in=-60, out=180, looseness=0.75] (14.center) to (6);
		\draw [in=0, out=-120, looseness=0.50] (10) to (5.center);
		\draw [in=-105, out=0] (15.center) to (7);
		\draw [in=0, out=180, looseness=0.50] (16.center) to (4.center);
		\draw [in=-180, out=0] (16.center) to (15.center);
		\draw [in=0, out=-120] (9) to (18.center);
		\draw [in=105, out=-75] (7) to (8);
		\draw [in=-180, out=45, looseness=0.75] (7) to (11.center);
		\draw [in=360, out=180] (19.center) to (2.center);
		\draw [in=-120, out=0, looseness=0.75] (14.center) to (8);
		\draw [in=0, out=-180, looseness=0.75] (20.center) to (19.center);
		\draw [in=180, out=0, looseness=0.75] (20.center) to (12.center);
		\draw [in=180, out=0] (54.center) to (1.center);
		\draw [in=135, out=60, looseness=0.75] (6) to (7);
		\draw [in=0, out=-180, looseness=0.75] (18.center) to (3.center);
		\draw [style=dashed, in=0, out=-90] (81.center) to (80.center);
		\draw [style=dashed] (80.center) to (79.center);
		\draw [style=dashed, in=270, out=180, looseness=1.25] (79.center) to (78.center);
		\draw [style=dashed] (78.center) to (77.center);
		\draw [style=dashed, in=180, out=90, looseness=1.25] (77.center) to (76.center);
		\draw [style=dashed] (76.center) to (75.center);
		\draw [style=dashed, in=90, out=0] (75.center) to (82.center);
		\draw [style=dashed] (82.center) to (81.center);
\end{tikzpicture}
\end{center}
If one writes down the objects in each region of this 2-cell, one can easily observe that in the 2-cell $[1,1s^{-1}]$, the second instance of $1$ refers to a monoidal unit. Therefore, one can see this 2-cell as simply $[1,s^{-1}]$. This allows us to recognize the possibility of applying the swallowtail equation \eqref{swallowtail1} in the highlighted area. Also, we invoke Mac Lane coherence theorem (since the coherence axioms clearly hold for the $\cat V$-bicategory $\cat C$) to conclude that the composition of the instances of $\lambda$ and $\alpha$ are the identity, and we find thus
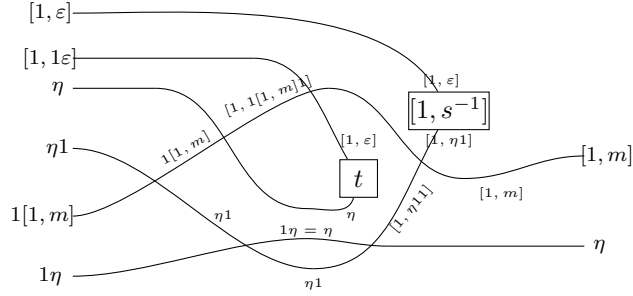
\begin{center}
\begin{tikzpicture}[scale=.4]
		\node [style=none] (0) at (-11.75, 4.25) {};
		\node [style=none] (1) at (-5.75, 2.75) {};
		\node [style=none] (3) at (-11.75, -0.25) {};
		\node [style=none] (4) at (-11.75, -2.5) {};
		\node [style=none] (5) at (-11.75, -4.5) {};
		\node [style=squared] (9) at (0.75, 1) {$[1,s^{-1}]$};
		\node [style=squared] (10) at (-2.25, -1.25) {$t$};
		\node [style=none] (11) at (5.25, -0.5) {};
		\node [style=none] (12) at (5.25, -3.5) {};
		\node [style=none] (14) at (1.25, -1.25) {};
		\node [style=none] (15) at (-3.25, 1.75) {};
		\node [style=none] (16) at (-3.75, -4.25) {};
		\node [style=none] (17) at (-9, 1.75) {};
		\node [style=none] (18) at (-1.25, -3.5) {};
		\node [style=none] (19) at (-12.5, 4.25) {$[1,\varepsilon]$};
		\node [style=none] (20) at (-12.5, 2.75) {$[1,1\varepsilon]$};
		\node [style=none] (21) at (-12.25, 1.75) {$\eta$};
		\node [style=none] (22) at (-12.25, -0.25) {$\eta 1$};
		\node [style=none] (23) at (-12.75, -2.5) {$1[1,m]$};
		\node [style=none] (26) at (-8, -0.25) {\rotatebox{30}{\tiny $1[1,m]$}};
		\node [style=none] (27) at (-6.75, -2.5) {\tiny $\eta 1$};
		\node [style=none] (28) at (-3.75, -4.75) {\tiny $\eta 1$};
		\node [style=none] (29) at (-0.5, -2.25) {\rotatebox{50}{\tiny $[1,\eta 11]$}};
		\node [style=none] (30) at (0.75, 0) {\tiny $[1,\eta1]$};
		\node [style=none] (31) at (0.5, 2) {\tiny $[1,\varepsilon]$};
		\node [style=none] (32) at (-3.75, 2.75) {};
		\node [style=none] (33) at (-11.75, 2.75) {};
		\node [style=none] (34) at (-2.5, -2.5) {\tiny $\eta$};
		\node [style=none] (36) at (5.75, -3.5) {$\eta$};
		\node [style=none] (37) at (6, -0.5) {$[1,m]$};
		\node [style=none] (41) at (-2.25, 0) {\tiny $[1,\varepsilon]$};
		\node [style=none] (42) at (-4, -3) {\tiny $1\eta = \eta$};
		\node [style=none] (45) at (2.5, -1.75) {\tiny $[1,m]$};
		\node [style=none] (46) at (-5.25, 1.5) {\rotatebox{20}{\tiny $[1,1[1,m]1]$}};
		\node [style=none] (48) at (-12.5, -4.5) {$1\eta$};
		\node [style=none] (49) at (-4, -3.25) {};
		\node [style=none] (50) at (-4, -2.25) {};
		\node [style=none] (51) at (-11.75, 1.75) {};

		\draw [in=360, out=120, looseness=0.75] (9) to (0.center);
		\draw [in=120, out=0] (1.center) to (10);
		\draw [in=0, out=180, looseness=0.50] (15.center) to (4.center);
		\draw [in=-180, out=0] (15.center) to (14.center);
		\draw [in=0, out=-120] (9) to (16.center);
		\draw [in=180, out=0, looseness=0.75] (18.center) to (12.center);
		\draw [in=180, out=0] (33.center) to (1.center);
		\draw [in=0, out=-180, looseness=0.75] (16.center) to (3.center);
		\draw [in=-180, out=0] (14.center) to (11.center);
		\draw [in=0, out=-180] (18.center) to (49.center);
		\draw [in=0, out=-180, looseness=0.75] (49.center) to (5.center);
		\draw [in=0, out=180] (50.center) to (17.center);
		\draw [in=255, out=0] (50.center) to (10);
		\draw [in=180, out=0] (51.center) to (17.center);
\end{tikzpicture}
\end{center}
which is eventually the desired right-hand side \eqref{rhs hom psfun} simply by pulling down the multiplication 1-cell by the modification property of $t$. This concludes the proof.
\end{proof}

\section{Extra-pseudonautral transformations}\label{section extrapsnat}

In this section we introduce the enriched version of the notion of extra-pseudonatural transformation between $\cat V$-pseudofunctors $P,Q$ of sort
\begin{align*}
P&\colon\cat E\otimes\cat B\op\otimes\cat B\longrightarrow\cat D\\
Q&\colon\cat E\otimes\cat C\op\otimes\cat C\longrightarrow\cat D.
\end{align*}

The non-enriched case was introduced in \cite{Corner}. The axioms we are going to deal with are a lot, but still they are reduced in size by passing to the enriched context, just as it happens for usual pseudonatural transformations. Moreover, by keeping in mind that for every axiom (concerning $\cat B$) we have a symmetric one (concerning $\cat C$), everything reduces to keeping track of \emph{unitality}, \emph{functoriality} (as for pseudonatural transformations) and \emph{compatibility} with parameters (axioms \eqref{EU},\eqref{EF},\eqref{EC} below). What we get for free in the enriched context is \emph{naturality}. We specify at the outset that because of the considerable space that writing the axioms for such structures requires, we will often use intuitive (and frequent in literature) abbreviations, such as $P_{eab}$ to mean $P(e,a,b)$ (and the same for $Q$), or omitting to explicitly write the parameter $e$ if it happens to not being crucially involved in a specific diagram.

\begin{rmk}
It is easy to check that when a $\cat V$-pseudofunctor $P\colon \cat A\otimes\cat B\to\cat D$ is defined over a tensor product of $\cat V$-bicategories, we can define for every object $a$ in $\cat A$ and $b\in\cat B$, two $\cat V$-pseudofunctors $P_{a-}\colon\cat B\to\cat D$ and $P_{-b}\colon\cat B\to\cat D$. The definition of, for example, $P_{a-}$, is clearly given by $P_{ab}$ on each object $b$, and on each hom-object is the composition
\[P_{a-}\colon \cat B(b,b')\overset{u1}{\longrightarrow}\cat A(a,a)\cat B(b,b')\overset{P}{\longrightarrow}\cat D(P_{ab},P_{ab'}).\]
Similarly, $P_{-b}\coloneqq P\circ 1u$. At the level of 2-cells, we have
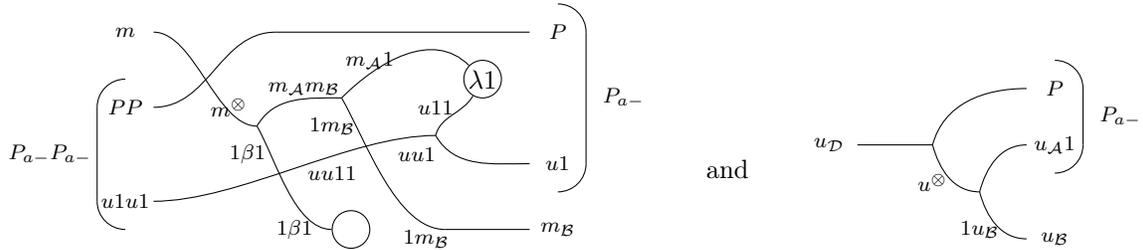
\begin{center}
\begin{tikzpicture}[scale=.5]
		\node [style=none] (0) at (-5, 2) {};
		\node [style=none] (1) at (-5, 0) {};
		\node [style=none] (2) at (-5, -2.5) {};
		\node [style=none] (3) at (-5.75, 2) {$m$};
		\node [style=none] (4) at (-5.75, 0) {$PP$};
		\node [style=none] (5) at (-5.75, -2.5) {$u1u1$};
		\node [style=none] (6) at (-5.75, 0.75) {};
		\node [style=none] (7) at (-6.5, 0) {};
		\node [style=none] (8) at (-6.5, -2.5) {};
		\node [style=none] (9) at (-5.75, -3.25) {};
		\node [style=none] (10) at (-7.75, -1.25) {$P_{a-}P_{a-}$};
		\node [style=none] (11) at (-1.75, 2) {};
		\node [style=none] (12) at (-2.25, -0.5) {};
		\node [style=none] (13) at (5, 2) {};
		\node [style=none] (14) at (0, 0.25) {};
		\node [style=fixed size node] (15) at (0.25, -3.25) {};
		\node [style=fixed size node] (16) at (3.75, 0.75) {$\lambda 1$};
		\node [style=none] (17) at (2.75, -3.25) {};
		\node [style=none] (18) at (2.5, -0.75) {};
		\node [style=none] (19) at (5, -1.5) {};
		\node [style=none] (20) at (5, -3.25) {};
		\node [style=none] (21) at (-0.25, -1.75) {$uu11$};
		\node [style=none] (22) at (-2.5, -1.25) {$1\beta1$};
		\node [style=none] (23) at (5.75, 2) {$P$};
		\node [style=none] (24) at (-3, 0) {$m^\otimes$};
		\node [style=none] (25) at (-1, 0.5) {$m_\mathcal A m_\mathcal B$};
		\node [style=none] (26) at (0.75, 1.25) {$m_\mathcal A 1$};
		\node [style=none] (27) at (2.5, 0) {$u11$};
		\node [style=none] (28) at (5.75, -1.5) {$u1$};
		\node [style=none] (29) at (2, -1.25) {$uu1$};
		\node [style=none] (30) at (-0.25, -0.5) {$1m_\mathcal B$};
		\node [style=none] (31) at (2.25, -3.5) {$1m_\mathcal B$};
		\node [style=none] (32) at (5.75, -3.25) {$m_\mathcal B$};
		\node [style=none] (33) at (-1.25, -3.25) {$1\beta 1$};
		\node [style=none] (34) at (5.75, 2.75) {};
		\node [style=none] (35) at (6.5, 2) {};
		\node [style=none] (36) at (6.5, -1.5) {};
		\node [style=none] (37) at (5.75, -2.25) {};
		\node [style=none] (38) at (7.5, 0.25) {$P_{a-}$};

		\draw [in=270, out=180] (9.center) to (8.center);
		\draw [in=90, out=-90] (7.center) to (8.center);
		\draw [in=180, out=90] (7.center) to (6.center);
		\draw [in=180, out=0] (1.center) to (11.center);
		\draw [in=0, out=180, looseness=0.75] (12.center) to (0.center);
		\draw [in=360, out=180] (13.center) to (11.center);
		\draw [in=-180, out=60] (12.center) to (14.center);
		\draw [in=180, out=-60, looseness=0.75] (12.center) to (15);
		\draw [in=-60, out=180, looseness=0.75] (17.center) to (14.center);
		\draw [in=135, out=45] (14.center) to (16);
		\draw [in=-180, out=0, looseness=0.75] (2.center) to (18.center);
		\draw [in=-120, out=75] (18.center) to (16);
		\draw [in=360, out=180] (20.center) to (17.center);
		\draw [in=180, out=-60] (18.center) to (19.center);
		\draw [in=270, out=0] (37.center) to (36.center);
		\draw (36.center) to (35.center);
		\draw [in=360, out=90] (35.center) to (34.center);
\end{tikzpicture}
\hspace{2em}\raisebox{3em}{and}\hspace{2em}
\begin{tikzpicture}[scale=.5]
		\node [style=none] (0) at (-2, 0) {};
		\node [style=none] (1) at (-2.75, 0) {$u_\mathcal D$};
		\node [style=none] (2) at (0, 0) {};
		\node [style=none] (3) at (2.5, 1.5) {};
		\node [style=none] (4) at (1.25, -1.25) {};
		\node [style=none] (5) at (2.5, 0) {};
		\node [style=none] (6) at (2.5, -2.5) {};
		\node [style=none] (7) at (3.25, 1.5) {$P$};
		\node [style=none] (8) at (3.25, 0) {$u_\mathcal A1$};
		\node [style=none] (9) at (3.25, -2.5) {$u_\mathcal B$};
		\node [style=none] (10) at (1.25, -2.25) {$1u_\mathcal B$};
		\node [style=none] (11) at (0, -1) {$u^\otimes$};
		\node [style=none] (12) at (3.25, 2.25) {};
		\node [style=none] (13) at (3.25, -0.75) {};
		\node [style=none] (14) at (4, 0) {};
		\node [style=none] (15) at (4, 1.5) {};
		\node [style=none] (16) at (5, 0.75) {$P_{a-}$};

		\draw [in=-75, out=-180] (6.center) to (4.center);
		\draw [in=180, out=75] (4.center) to (5.center);
		\draw [in=285, out=180] (4.center) to (2.center);
		\draw [in=-180, out=75] (2.center) to (3.center);
		\draw (0.center) to (2.center);
		\draw [in=90, out=0] (12.center) to (15.center);
		\draw (15.center) to (14.center);
		\draw [in=0, out=-90] (14.center) to (13.center);
\end{tikzpicture}
\end{center}
where the symbols $m^\otimes$ and $u^\otimes$ denote the multiplication and the unit for the $\cat V$-bicategory $\cat A\otimes\cat B$. The proof of the axioms are an easy consequence of the corresponding axioms for $P$.
\end{rmk}

\begin{notat}
Let us provide a name for some useful 2-cells which will be subsequently used in the definition of enriched extra-pseudonatural transformation. We can observe that for every pair of objects $a,a'$ in $\cat A$ and $b,b'$ in $\cat B$, we have isomorphisms
\begin{center}
\begin{tikzpicture}[scale=.4]
		\node [style=none] (0) at (-4, 0) {};
		\node [style=squared] (1) at (0, 0) {$\nu$};
		\node [style=none] (2) at (5, 3) {};
		\node [style=none] (3) at (5, -3) {};
		\node [style=none] (4) at (5.5, 3) {$m$};
		\node [style=none] (5) at (6.25, -3) {$P_{-b'}P_{a-}$};
		\node [style=none] (6) at (-4.5, 0) {$P$};
		\node [style=none] (7) at (-2.25, -2) {$\mathcal A(a,a')\mathcal B(b,b')$};
		\node [style=none] (8) at (5.5, 0) {$\mathcal D(P_{ab'},P_{a'b'})\mathcal D(P_{ab},P_{ab'})$};
		\node [style=none] (9) at (-2, 2) {$\mathcal D(P_{ab},P_{a'b'})$};

		\draw (0.center) to (1);
		\draw [in=180, out=-75] (1) to (3.center);
		\draw [in=75, out=180] (2.center) to (1);
\end{tikzpicture}
\hspace{2em}\raisebox{4em}{$\coloneqq$}\hspace{2em}
\begin{tikzpicture}[scale=.5]
		\node [style=none] (0) at (-5.5, 0.75) {$P$};
		\node [style=none] (1) at (-5, 0.75) {};
		\node [style=squared] (2) at (-2.75, -1.5) {$\rho^{-1}\lambda^{-1}$};
		\node [style=none] (3) at (-1, -0.5) {};
		\node [style=none] (4) at (-0.25, -3.25) {};
		\node [style=none] (5) at (-2.25, -0.5) {$mm$};
		\node [style=none] (6) at (0.5, -3.5) {$1uu1$};
		\node [style=none] (7) at (0.25, -1.25) {};
		\node [style=fixed size node] (8) at (-4.25, -3.75) {};
		\node [style=none] (9) at (5, 0.75) {};
		\node [style=none] (10) at (3.5, -1.5) {};
		\node [style=none] (11) at (3.25, -0.75) {$PP$};
		\node [style=none] (12) at (5.5, 0.75) {$m$};
		\node [style=none] (13) at (-3, -4) {$1\beta 1$};
		\node [style=none] (14) at (1.25, -1.5) {$m^\otimes$};
		\node [style=none] (15) at (5, -1.5) {};
		\node [style=none] (16) at (6.25, -1.5) {$P_{-b'}P_{a-}$};
		\node [style=none] (17) at (-2.75, -2.5) {$1uu1$};

		\draw [in=60, out=-180] (3.center) to (2);
		\draw [in=-180, out=-60] (2) to (4.center);
		\draw [in=0, out=120] (7.center) to (3.center);
		\draw [in=-120, out=0, looseness=0.75] (8) to (7.center);
		\draw [in=120, out=0, looseness=0.50] (1.center) to (10.center);
		\draw [in=-180, out=0, looseness=0.75] (7.center) to (9.center);
		\draw [in=-105, out=0] (4.center) to (10.center);
		\draw (10.center) to (15.center);
\end{tikzpicture}
\end{center}
a similar use of $\lambda$ and $\rho$ also leads to build the following two 2-isomorphisms, which we avoid to write down since they will only be used to state an axiom (the axiom \eqref{EC} of compatibility with parameters) that won't need to be checked for the rest of the paper, since we won't explicitly deal with bicoends with parameters.
\begin{center}
\begin{tikzpicture}[scale=.4]
		\node [style=squared] (0) at (0, 0) {$\kappa$};
		\node [style=none] (1) at (4.5, 0) {$\beta$};
		\node [style=none] (2) at (4.5, 3) {$m$};
		\node [style=none] (3) at (5.25, -3) {$P_{-b'}P_{a'-}$};
		\node [style=none] (4) at (-3, 2) {};
		\node [style=none] (5) at (-3, -2) {};
		\node [style=none] (6) at (-3.5, 2) {$m$};
		\node [style=none] (7) at (-4.25, -2) {$P_{-b}P_{a-}$};
		\node [style=none] (8) at (4, 0) {};
		\node [style=none] (9) at (4, 3) {};
		\node [style=none] (10) at (4, -3) {};
		\node [style=none] (11) at (-0.5, -3) {$ \mathcal A(a,a')\mathcal B(b',b)$};
		\node [style=none] (12) at (-4.75, 0) {$\mathcal D(P_{ab},P_{a'b})\mathcal D(P_{ab'},P_{ab})$};
		\node [style=none] (13) at (5.25, -1.5) {$\mathcal D(P_{ab'},P_{a'b'})\mathcal D(P_{a'b'},P_{a'b})$};
		\node [style=none] (14) at (5.25, 1.5) {$\mathcal D(P_{a'b'},P_{a'b})\mathcal D(P_{ab'},P_{a'b'})$};
		\node [style=none] (15) at (-0.5, 3) {$\mathcal D(P_{ab'},P_{a'b})$};
		
		\draw [in=75, out=-180] (9.center) to (0);
		\draw [in=180, out=0] (0) to (8.center);
		\draw [in=-75, out=180] (10.center) to (0);
		\draw [in=0, out=-120, looseness=0.75] (0) to (5.center);
		\draw [in=0, out=120, looseness=0.75] (0) to (4.center);
\end{tikzpicture}
\hspace{1em}\raisebox{3em}{,}\hspace{2em}
\begin{tikzpicture}[scale=.4]
		\node [style=squared] (0) at (0, 0) {$\upsilon$};
		\node [style=none] (1) at (-3, 0) {};
		\node [style=none] (2) at (4, 3) {};
		\node [style=none] (3) at (4, -3) {};
		\node [style=none] (4) at (4, 0) {};
		\node [style=none] (5) at (-3.75, 0) {$P$};
		\node [style=none] (6) at (4.75, 0) {$\beta$};
		\node [style=none] (7) at (4.75, 3) {$m$};
		\node [style=none] (8) at (5.5, -3) {$P_{-b'}P_{a'-}$};
		\node [style=none] (9) at (-1.5, -2.5) {$\mathcal A(a,a')\mathcal B(b',b)$};
		\node [style=none] (10) at (-1.5, 2.5) {$\mathcal D(P_{ab'},P_{a'b})$};
		\node [style=none] (11) at (5.25, -1.5) {$\mathcal D(P_{ab'}P_{a'b'})\mathcal D(P_{a'b'},P_{a'b})$};
		\node [style=none] (12) at (5.25, 1.5) {$\mathcal D(P_{a'b'},P_{a'b})\mathcal D(P_{ab'}P_{a'b'})$};

		\draw [in=180, out=0] (1.center) to (0);
		\draw [in=180, out=75, looseness=1.25] (0) to (2.center);
		\draw [in=360, out=180] (4.center) to (0);
		\draw [in=180, out=-75, looseness=1.25] (0) to (3.center);
\end{tikzpicture}

\end{center}
\end{notat}

\begin{defn}
Let $P\colon \cat E\otimes\cat{B}\op\otimes\cat{B}\to\cat{D}$ and $Q\colon\cat E\otimes\cat{C}\op\otimes\cat{C}\to \cat{D}$ be $\cat V$-pseudofunctors, for $\cat V$ a braided monoidal bicategory. An (enriched) \emph{extra-pseudonatural transformation} from $P$ to $Q$, denoted $i\colon P\epn Q$, consists of, for every pair of objects $a,b$ in $\cat{B}$, $x,y$ in $\cat{C}$, 
\begin{itemize}
    \item An enriched pseudonatural transformation
    $$i_{-,a,x}\colon P(-,a,a)\Rightarrow Q(-,x,x)$$
    \item Two 2-isomorphisms in $\cat V$
    \end{itemize}
    
\noindent
\begin{minipage}{.1\textwidth}
\begin{equation}\label{epnbb'}
\phantom{c}
\end{equation}
\end{minipage}
\begin{minipage}{.9\textwidth}
\begin{center}
\begin{tikzpicture}[scale=.5]
		\node [style=fixed size node] (0) at (0, 0) {$i_{ab}$};
		\node [style=none] (1) at (-4, 2) {};
		\node [style=none] (2) at (-4, -2) {};
		\node [style=none] (3) at (4, 2) {};
		\node [style=none] (4) at (4, -2) {};
		\node [style=none] (5) at (4.5, 2) {$i_\ast$};
		\node [style=none] (6) at (4.75, -2) {$P_{b-}$};
		\node [style=none] (7) at (-4.75, -2) {$P_{- a}$};
		\node [style=none] (8) at (-4.5, 2) {$i_\ast$};
		\node [style=none] (9) at (0, -2.75) {$\mathcal B(a,b)$};
		\node [style=none] (10) at (0, 2.75) {$\mathcal D(P_{ba},Q_{xx})$};
		\node [style=none] (11) at (-3.25, 0) {$\mathcal D(P_{ba},P_{aa})$};
		\node [style=none] (12) at (3.25, 0) {$\mathcal D(P_{ba},P_{bb})$};

		\draw [in=-60, out=-180] (4.center) to (0);
		\draw [in=180, out=60] (0) to (3.center);
		\draw [in=120, out=0] (1.center) to (0);
		\draw [in=0, out=-120] (0) to (2.center);
\end{tikzpicture}
\end{center}
\end{minipage}
which should more precisely be called $i_{e,ab,x}$, having one also for every $e$ in $\cat E$, but we allow ourselves some lightening in notations by in particular suppressing the $\cat E$ variable, and

\noindent
\begin{minipage}{.1\textwidth}
\begin{equation}\label{epncc'}
\phantom{c}
\end{equation}
\end{minipage}
\begin{minipage}{.9\textwidth}
\begin{center}
\begin{tikzpicture}[scale=.5]
		\node [style=fixed size node] (0) at (0, 0) {$i_{xy}$};
		\node [style=none] (1) at (-4, 2) {};
		\node [style=none] (2) at (-4, -2) {};
		\node [style=none] (3) at (4, 2) {};
		\node [style=none] (4) at (4, -2) {};
		\node [style=none] (5) at (4.5, 2) {$i^\ast$};
		\node [style=none] (6) at (4.75, -2) {$Q_{x-}$};
		\node [style=none] (7) at (-4.75, -2) {$Q_{-y}$};
		\node [style=none] (8) at (-4.5, 2) {$i^\ast$};
		\node [style=none] (9) at (0, -2.75) {$\mathcal C(x,y)$};
		\node [style=none] (10) at (0, 2.75) {$\mathcal D(P_{aa},Q_{xy})$};
		\node [style=none] (11) at (-3.25, 0) {$\mathcal D(Q_{yy},Q_{xy})$};
		\node [style=none] (12) at (3.25, 0) {$\mathcal D(Q_{xx},Q_{xy})$};

		\draw [in=-60, out=-180] (4.center) to (0);
		\draw [in=180, out=60] (0) to (3.center);
		\draw [in=120, out=0] (1.center) to (0);
		\draw [in=0, out=-120] (0) to (2.center);
\end{tikzpicture}
\end{center}
\end{minipage}
which again should be denoted $i_{e,a,xy}$. This structure is subject to the following three axioms.

\emph{Unitality}: for every object $b$ in $\cat B$

\vspace{1em}
\noindent
\begin{minipage}{.1\textwidth}
\begin{equation}\label{EU}
\phantom{c}
\tag{EU}
\end{equation}
\end{minipage}
\begin{minipage}{.9\textwidth}
\begin{center}
\begin{tikzpicture}[scale=.5, xscale=1.5]
		\node [style=none] (0) at (0, 0) {$=$};
		\node [style=none] (1) at (-1.5, 1.75) {$i_\ast$};
		\node [style=none] (2) at (-1.5, -1.75) {$u$};
		\node [style=none] (3) at (-7, 1.75) {};
		\node [style=none] (4) at (-7, -1.75) {};
		\node [style=none] (5) at (-2, 1.75) {};
		\node [style=none] (6) at (-2, -1.75) {};
		\node [style=none] (7) at (-7.5, 1.75) {$i_\ast$};
		\node [style=none] (8) at (-7.5, -1.75) {$u$};
		\node [style=none] (9) at (7.5, 1.75) {$i_\ast$};
		\node [style=none] (10) at (7.5, -1.75) {$u.$};
		\node [style=none] (11) at (2, 1.75) {};
		\node [style=none] (12) at (2, -1.75) {};
		\node [style=none] (13) at (7, 1.75) {};
		\node [style=none] (14) at (7, -1.75) {};
		\node [style=none] (15) at (1.5, 1.75) {$i_\ast$};
		\node [style=none] (16) at (1.5, -1.75) {$u$};
		\node [style=none] (17) at (-6, -1.75) {};
		\node [style=none] (18) at (-3, -1.75) {};
		\node [style=fixed size node] (19) at (-4.5, 0) {$i_{bb}$};
		\node [style=none] (20) at (-5.75, -0.75) {$P_{- b}$};
		\node [style=none] (21) at (-4.5, -3) {$u$};
		\node [style=none] (22) at (-3.25, -0.75) {$P_{b-}$};

		\draw (11.center) to (13.center);
		\draw (12.center) to (14.center);
		\draw [in=105, out=0] (3.center) to (19);
		\draw [in=180, out=75] (19) to (5.center);
		\draw [in=75, out=-120] (19) to (17.center);
		\draw [in=360, out=180] (17.center) to (4.center);
		\draw [in=255, out=-75] (17.center) to (18.center);
		\draw [in=180, out=0] (18.center) to (6.center);
		\draw [in=-60, out=105] (18.center) to (19);
\end{tikzpicture}
\end{center}
\end{minipage}
\vspace{1em}

\emph{Functoriality}: for every triple of objects $a,b,c$ in $\cat B$
\vspace{1em}

\noindent
\begin{minipage}{.01\textwidth}
\begin{equation}\label{EF}
\phantom{c}
\tag{EF}
\end{equation}
\end{minipage}
\begin{minipage}{.99\textwidth}
\begin{center}
\begin{tikzpicture}[scale=.5,xscale=.8]
		\node [style=none] (0) at (-10.25, 2.5) {};
		\node [style=none] (1) at (-10.25, 0) {};
		\node [style=none] (2) at (-10.25, -2) {};
		\node [style=fixed size node] (3) at (-7.25, 1.25) {$i_{ac}$};
		\node [style=none] (4) at (-6.25, -2) {};
		\node [style=none] (5) at (-3, 0.75) {};
		\node [style=none] (6) at (-3, -1) {};
		\node [style=none] (7) at (-3, -2.75) {};
		\node [style=none] (8) at (-3, 2.75) {};
		\node [style=none] (9) at (-10.75, 2.5) {$i_\ast$};
		\node [style=none] (10) at (-10.75, 0) {$m$};
		\node [style=none] (11) at (-11.25, -2) {$P_{-a}P_{-a}$};
		\node [style=none] (12) at (-2.5, 0.75) {$m$};
		\node [style=none] (13) at (-2, -1) {$P_{c-}P_{c-}$};
		\node [style=none] (14) at (-2.5, -2.75) {$\beta$};
		\node [style=none] (15) at (-2.5, 2.75) {$i_\ast$};
		\node [style=none] (16) at (0, 0) {$=$};
		\node [style=none] (17) at (3, 2.5) {};
		\node [style=none] (18) at (3, 0) {};
		\node [style=none] (19) at (3, -2) {};
		\node [style=none] (23) at (21.75, 0.75) {};
		\node [style=none] (24) at (21.75, -1) {};
		\node [style=none] (25) at (21.75, -2.75) {};
		\node [style=none] (26) at (21.75, 2.75) {};
		\node [style=none] (27) at (22.25, 0.75) {$m$};
		\node [style=none] (28) at (22.75, -1) {$P_{c-}P_{c-}$};
		\node [style=none] (29) at (22.25, -2.75) {$\beta$};
		\node [style=none] (30) at (22.25, 2.75) {$i_\ast$};
		\node [style=fixed size node] (31) at (5.5, -1) {$i_{ab}1$};
		\node [style=fixed size node] (32) at (18.5, 0) {$i_{bc}1$};
		\node [style=none] (33) at (5.5, 1.5) {};
		\node [style=none] (34) at (10.5, 2.5) {};
		\node [style=none] (35) at (10.75, -1) {};
		\node [style=none] (36) at (7.75, -2.25) {};
		\node [style=squared] (37) at (13.75, -1) {$\lambda\rho$};
		\node [style=squared] (38) at (14.25, 1) {$\nu$};
		\node [style=none] (39) at (18.75, 2.75) {};
		\node [style=none] (40) at (16.25, 1.75) {};
		\node [style=none] (41) at (-8.25, 0.25) {$P_{-a}$};
		\node [style=none] (42) at (-7.75, -2.25) {$m^\mathsf{op}$};
		\node [style=none] (43) at (-6, -1.25) {$m$};
		\node [style=none] (44) at (-5.25, 0.25) {$P_{c-}$};
		\node [style=none] (45) at (2.5, 2.5) {$i_\ast$};
		\node [style=none] (46) at (2.5, 0) {$m$};
		\node [style=none] (47) at (2, -2) {$P_{-a}P_{-a}$};
		\node [style=none] (48) at (5.5, 1.75) {$m$};
		\node [style=none] (49) at (4.25, -0.25) {$i_\ast1$};
		\node [style=none] (50) at (6.5, -0.25) {$i_\ast 1$};
		\node [style=none] (51) at (6, -2.8) {$P_{b-}P_{-a}$};
		\node [style=none] (52) at (7.75, -1) {$PP$};
		\node [style=none] (53) at (9.75, -3.5) {$u11u$};
		\node [style=none] (54) at (9.75, -1.25) {$m^\otimes$};
		\node [style=none] (55) at (8.25, 0.5) {$m$};
		\node [style=none] (56) at (12.5, 1.25) {$P$};
		\node [style=none] (57) at (14.25, 1.85) {$m$};
		\node [style=none] (58) at (17, -1.5) {$P_{-b}P_{c-}$};
		\node [style=none] (59) at (16.75, 1) {$i_\ast 1$};
		\node [style=none] (60) at (19.75, 1) {$i_\ast 1$};
		\node [style=none] (61) at (18.75, 3) {$m$};
		\node [style=none] (62) at (15.75, -3.25) {$1\beta 1$};
		\node [style=none] (63) at (11, -2) {$1\beta 1$};
		\node [style=none] (64) at (12.25, 0.25) {$mm$};
		\node [style=none] (65) at (12, 3) {$i_\ast$};

		\draw [in=120, out=0, looseness=0.75] (0.center) to (3);
		\draw [in=-180, out=60, looseness=0.75] (3) to (8.center);
		\draw [in=180, out=-60, looseness=0.50] (3) to (6.center);
		\draw [in=180, out=60, looseness=0.50] (4.center) to (5.center);
		\draw [in=180, out=-60, looseness=0.50] (4.center) to (7.center);
		\draw [in=-180, out=0] (1.center) to (4.center);
		\draw [in=-105, out=0, looseness=0.75] (2.center) to (3);
		\draw [in=120, out=0, looseness=0.75] (17.center) to (31);
		\draw [in=0, out=-120] (31) to (19.center);
		\draw [in=180, out=0] (18.center) to (33.center);
		\draw [in=180, out=60, looseness=0.75] (31) to (34.center);
		\draw [in=180, out=0, looseness=0.75] (33.center) to (35.center);
		\draw [in=180, out=-60] (31) to (36.center);
		\draw [in=105, out=75, looseness=0.75] (35.center) to (37);
		\draw [in=-180, out=-75, looseness=0.50] (35.center) to (25.center);
		\draw [in=180, out=75] (36.center) to (38);
		\draw [in=-120, out=-75, looseness=0.75] (36.center) to (37);
		\draw [in=-60, out=-180, looseness=0.75] (24.center) to (32);
		\draw [in=-180, out=60] (32) to (26.center);
		\draw [in=0, out=-180] (23.center) to (39.center);
		\draw [in=75, out=-180, looseness=0.75] (39.center) to (38);
		\draw [in=-45, out=120, looseness=0.75] (32) to (40.center);
		\draw [in=360, out=135, looseness=0.75] (40.center) to (34.center);
		\draw [in=-120, out=-75] (38) to (32);
\end{tikzpicture}
\end{center}
\end{minipage}

\vspace{1em}
\noindent
and similarly for triples of objects in $\cat C$, by replacing $P$ by $Q$ and $(-)_\ast$ by $(-)^\ast$.

\emph{Compatibility} with parameters, roughly consisting in an interchange law for the structure of pseudonatural transformation in the parametric variable and the isomorphisms \eqref{epnbb'} and \eqref{epncc'}: for every $x$ in $\cat C$, every pair of objects $e,f$ in $\cat E$ and $a,b$ in $\cat B$

\noindent
\begin{minipage}{.01\textwidth}
\begin{equation}\label{EC}
\phantom{c}
\tag{EC}
\end{equation}
\end{minipage}
\begin{minipage}{.99\textwidth}
\begin{center}
\begin{tikzpicture}[scale=.5,xscale=.8,yscale=1.2]
		\node [style=none] (0) at (-14.5, 0) {$1i_\ast$};
		\node [style=none] (1) at (-14.5, 2) {$m$};
		\node [style=none] (2) at (-15.75, -2) {$Q_{-xx}P_{e-a}$};
		\node [style=none] (3) at (-14, -2) {};
		\node [style=none] (4) at (-14, 0) {};
		\node [style=none] (5) at (-14, 2) {};
		\node [style=squared] (6) at (-11.5, -1) {$1i_{e,ab}$};
		\node [style=fixed size node] (7) at (-9.5, 1) {$m_i$};
		\node [style=squared] (8) at (-7.5, -1) {$i_{ef,b}1$};
		\node [style=none] (9) at (-11, 0.25) {$1i_\ast$};
		\node [style=none] (10) at (-9.5, -2) {};
		\node [style=none] (11) at (-9.5, -2.5) {$Q_{-xx}P_{eb-}$};
		\node [style=none] (12) at (-5.5, -2) {};
		\node [style=fixed size node] (13) at (-3.5, -1) {$\nu^{-1}$};
		\node [style=none] (14) at (-7.25, 2) {};
		\node [style=none] (15) at (-2, -1) {};
		\node [style=none] (16) at (-2, 1.5) {};
		\node [style=none] (17) at (-1.25, -1) {$P_{-b-}$};
		\node [style=none] (18) at (-1.5, 1.5) {$i_\ast$};
		\node [style=none] (19) at (-5.5, -2.5) {$P_{-bb}P_{eb-}$};
		\node [style=none] (20) at (-5.75, 2) {$m$};
		\node [style=none] (21) at (-4, 0.5) {$m$};
		\node [style=none] (22) at (-6.25, 0) {$i_\ast 1$};
		\node [style=none] (23) at (-8.25, 0.25) {$i^\ast 1$};
		\node [style=none] (24) at (0, 0) {$=$};
		\node [style=none] (25) at (2.25, 0) {$1i_\ast$};
		\node [style=none] (26) at (2.25, 2) {$m$};
		\node [style=none] (27) at (1, -2) {$Q_{-xx}P_{e-a}$};
		\node [style=none] (28) at (2.75, -2) {};
		\node [style=none] (29) at (2.75, 0) {};
		\node [style=none] (30) at (2.75, 2) {};
		\node [style=none] (31) at (17.75, -1) {};
		\node [style=none] (32) at (17.75, 1.5) {};
		\node [style=none] (33) at (18.75, -1) {$P_{-b-}$};
		\node [style=none] (34) at (18.25, 1.5) {$i_\ast$};
		\node [style=fixed size node] (35) at (16.5, -1) {$\upsilon^{-1}$};
		\node [style=fixed size node] (36) at (4, 1) {$m_i$};
		\node [style=squared] (37) at (5.5, -1.25) {$i_{ef,a}1$};
		\node [style=none] (38) at (7, -2) {};
		\node [style=fixed size node] (39) at (9, -1) {$\kappa$};
		\node [style=none] (40) at (11, -2) {};
		\node [style=squared] (41) at (13, -1.25) {$1i_{f,ab}$};
		\node [style=none] (42) at (12.75, 1.25) {};
		\node [style=none] (43) at (13, 2.5) {};
		\node [style=none] (44) at (13.25, 0) {$1i_\ast$};
		\node [style=none] (45) at (14.75, 1.25) {$i_\ast 1$};
		\node [style=none] (46) at (15.25, 0) {$\beta$};
		\node [style=none] (47) at (5.75, 2) {};
		\node [style=none] (48) at (9, 2) {};
		\node [style=none] (49) at (7, -2.5) {$P_{-aa}P_{e-a}$};
		\node [style=none] (50) at (11, -2.3) {$P_{-ba}P_{f-a}$};
		\node [style=none] (51) at (14.75, -2.75) {$P_{-ba}P_{fb-}$};
		\node [style=none] (52) at (11.75, -0.25) {$1i_\ast$};
		\node [style=none] (53) at (16.25, 0.5) {$m$};
		\node [style=none] (54) at (13, 2.75) {$m$};
		\node [style=none] (55) at (9, 2.25) {$i_\ast$};
		\node [style=none] (56) at (6.75, -0.25) {$i_\ast 1$};
		\node [style=none] (57) at (4.5, -0.25) {$i^\ast 1$};
		\node [style=none] (58) at (5.75, 2.25) {$m$};
		\node [style=none] (59) at (8.25, 0.5) {$m$};
		\node [style=none] (60) at (10.25, -1.25) {$\beta$};
		\node [style=none] (61) at (12.75, 1.5) {$\beta$};
		\node [style=none] (62) at (9.5, 0.75) {$m$};
		\node [style=none] (63) at (11.3, 1.2) {$i_\ast 1$};

		\draw [in=120, out=0] (4.center) to (6);
		\draw [in=0, out=-120] (6) to (3.center);
		\draw [in=-105, out=75] (6) to (7);
		\draw [in=360, out=120, looseness=0.75] (7) to (5.center);
		\draw [in=180, out=-60] (6) to (10.center);
		\draw [in=-120, out=0] (10.center) to (8);
		\draw [in=-60, out=120] (8) to (7);
		\draw [in=-180, out=-60] (8) to (12.center);
		\draw [in=0, out=-120] (13) to (12.center);
		\draw [in=-180, out=60] (7) to (14.center);
		\draw [in=120, out=0, looseness=0.75] (14.center) to (13);
		\draw (15.center) to (13);
		\draw [in=75, out=-180, looseness=0.75] (16.center) to (8);
		\draw [in=-120, out=0] (29.center) to (36);
		\draw [in=0, out=120] (36) to (30.center);
		\draw [in=0, out=-135] (37) to (28.center);
		\draw [in=105, out=-60] (36) to (37);
		\draw [in=-180, out=-60] (37) to (38.center);
		\draw [in=-120, out=0] (38.center) to (39);
		\draw [in=180, out=-60] (39) to (40.center);
		\draw [in=-120, out=0] (40.center) to (41);
		\draw [in=-105, out=-75] (41) to (35);
		\draw (31.center) to (35);
		\draw [in=0, out=-180] (35) to (42.center);
		\draw [in=0, out=105] (35) to (43.center);
		\draw [in=180, out=75, looseness=0.75] (41) to (32.center);
		\draw [in=60, out=180] (47.center) to (36);
		\draw [in=120, out=0, looseness=0.75] (47.center) to (39);
		\draw [in=60, out=180, looseness=0.75] (48.center) to (37);
		\draw [in=120, out=0, looseness=0.75] (48.center) to (41);
		\draw [in=-180, out=0, looseness=0.75] (39) to (42.center);
		\draw [in=75, out=180] (43.center) to (39);
\end{tikzpicture}

\end{center}
\end{minipage}
\vspace{1em}

and similarly for pairs of objects in $\cat C$ and a fixed one in $\cat B$. The notation $m_i$ is the naturality for the multiplication introduced at Section \ref{subsect vpsnat}.

\end{defn}

\begin{rmk}
If $\cat V=\mathsf{Cat}$, by whiskering the two structural 2-cells \eqref{epnbb'} and \eqref{epncc'}, with enriched morphisms $g\colon\mathbb 1\to\cat B(b,b')$ and $h\colon\mathbb{1}\to\cat C(c,c')$ respectively, we get the non-enriched version of the structure, that is, the 2-isomorphisms $i_{g,c}$ and $i_{b,h}$
    \begin{center}
\begin{tikzcd}[row sep=3em]
{P(e,b',b)} \arrow[r, "{P(e,g,b)}"] \arrow[d, "{P(e,g,b')}"'] \arrow[rd, phantom, "{i_{g,c}\Arrowdl}" description] & {P(e,b,b)} \arrow[r, "{i_{e,b,c'}}"] \arrow[d, "{i_{e,b,c}}" description] \arrow[rd, phantom, "{i_{b,h}\Arrowdl}" description] & {Q(e,c',c')} \arrow[d, "{Q(e,h,c')}"] \\
{P(e,b',b')} \arrow[r, "{i_{e,b',c}}"']                                                                   & {Q(e,c,c)} \arrow[r, "{Q(e,c,h)}"']                                                                                   & {Q(e,c,c')}                          
\end{tikzcd}
    \end{center}
and the corresponding axioms which can be checked to be precisely those given in \cite{Corner}.
\end{rmk}
There is a notion of morphism of extra-pseudonatural transformations:
\begin{defn}\label{def extrapsnat category}
Let $j,j'\colon P\epn Q$ be extra-pseudonatural transformation of $\cat V$-pseudofunctors $P\colon \cat B\op\otimes\cat B\to\cat D$ and $Q\colon\cat C\op\otimes\cat C\to\cat D$. A \emph{morphism} $\Gamma\colon j\to j'$ is the data of an indexed family of 2-morphisms in $\cat V$ (for $b$ in $\cat B$ and $c$ in $\cat C$)
$$\Gamma_{b,c}\colon j_{b,c}\Longrightarrow j'_{b,c}$$
such that the following equalities hold true for all pairs $(b,b')$ and $(c,c')$ of objects of $\cat B$ and $\cat C$ respectively:
\begin{equation}
\begin{tikzcd}[row sep=5em]
	{\mathcal B(b,b')} & {\mathcal D(Pb'b,Pbb)} \\
	{\mathcal D(Pb'b,Pb'b')} & {\mathcal D(Pb'b,Qcc)}
	\arrow["{P(-,b)}", from=1-1, to=1-2]
	\arrow["{P(b',-)}"', from=1-1, to=2-1]
	\arrow["{\Arrowdl j_{bb',c}}"{description}, draw=none, from=1-2, to=2-1]
	\arrow["{{(j_{b,c})}_*}", from=1-2, to=2-2]
	\arrow[""{name=0, anchor=center, inner sep=0}, "{{(j_{b',c})}_*}", from=2-1, to=2-2]
	\arrow[""{name=1, anchor=center, inner sep=0}, "{{(j'_{b',c})}_*}"', curve={height=30pt}, from=2-1, to=2-2]
	\arrow["{\Arrowd{(\Gamma_{b',c})}_*}"{description}, draw=none, from=0, to=1]
\end{tikzcd}
=
\begin{tikzcd}[row sep=5em]
	{\mathcal B(b,b')} & {\mathcal D(Pb'b,Pbb)} \\
	{\mathcal D(Pb'b,Pb'b')} & {\mathcal D(Pb'b,Qcc)}
	\arrow["{P(-,b)}", from=1-1, to=1-2]
	\arrow["{P(b',-)}"', from=1-1, to=2-1]
	\arrow["{\Arrowdl j'_{bb',c}}"{description}, draw=none, from=1-2, to=2-1]
	\arrow[""{name=0, anchor=center, inner sep=0}, "{{(j_{b,c})}_*}", curve={height=-36pt}, from=1-2, to=2-2]
	\arrow[""{name=1, anchor=center, inner sep=0}, "{{(j'_{b,c})}_*}"', from=1-2, to=2-2]
	\arrow["{j'_{b',c}}"', from=2-1, to=2-2]
	\arrow["{\overset{(\Gamma_{b,c})_*}{\Leftarrow}}"{description}, draw=none, from=0, to=1]
\end{tikzcd}
\end{equation}
\begin{equation}
\begin{tikzcd}[row sep=5em]
	{\mathcal C(c,c')} & {\mathcal D(Qc'c', Qcc')} \\
	{\mathcal D(Qcc,Qcc')} & {\mathcal D(Pbb,Qcc')}
	\arrow["{Q(-,c')}", from=1-1, to=1-2]
	\arrow["{Q(c,-)}"', from=1-1, to=2-1]
	\arrow["{\Arrowdl j_{b,cc'}}"{description}, draw=none, from=1-2, to=2-1]
	\arrow["{{(j_{b,c'})}^*}", from=1-2, to=2-2]
	\arrow[""{name=0, anchor=center, inner sep=0}, "{{(j_{b,c})}^*}", from=2-1, to=2-2]
	\arrow[""{name=1, anchor=center, inner sep=0}, "{{(j'_{b,c})}^*}"', curve={height=30pt}, from=2-1, to=2-2]
	\arrow["{\Arrowd{(\Gamma_{b,c})}^*}"{description}, draw=none, from=0, to=1]
\end{tikzcd}
=
\begin{tikzcd}[row sep=5em]
	{\mathcal C(c,c')} & {\mathcal D(Qc'c', Qcc')} \\
	{\mathcal D(Qcc,Qcc')} & {\mathcal D(Pbb,Qcc')}
	\arrow["{Q(-,c')}", from=1-1, to=1-2]
	\arrow["{Q(c,-)}"', from=1-1, to=2-1]
	\arrow["{\Arrowdl j'_{b,cc'}}"{description}, draw=none, from=1-2, to=2-1]
	\arrow[""{name=0, anchor=center, inner sep=0}, "{{(j_{b,c'})}^*}", curve={height=-40pt}, from=1-2, to=2-2]
	\arrow[""{name=1, anchor=center, inner sep=0}, "{{(j'_{b,c'})}^*}"', from=1-2, to=2-2]
	\arrow["{({j'_{b,c})}^*}"', from=2-1, to=2-2]
	\arrow["{\overset{(\Gamma_{b,c'})^*}{\Leftarrow}}"{description}, draw=none, from=0, to=1]
\end{tikzcd}
\end{equation}
This definition clearly gives rise to a category $\cat V$-$\psnat^\mathsf{e}(P,Q)$, where composition is defined componentwise and the identity of an extra-pseudonatural transformation $j$ is at each component the identical 2-cell of $j_{b,c}$ in $\cat V$.
\end{defn}

\subsection{Examples}\label{subsect ex}

In this section we provide one main example of extra-pseudonatural transformation, which in the enriched 1-dimensional setting is the motivating example for the very introduction of the notion. First, let us recall a fact about non-enriched extra-pseudonatural transformations.

\begin{rmk}\label{parametric extra}
It is a straightforward construction that a \emph{parametric} family of pseudoadjunctions of pseudofunctors (such as the tensor-hom, see Remark \ref{parametric family rmk}) gives rise to extra-pseudonatural transformations. Indeed, suppose in fact we have pseudofunctors $F\colon \cat E\times\cat D\to \cat C$ and $G\colon\cat E\op\times\cat C\to\cat D$ together with, for every $e$ in $\cat E$ a pseudoadjunction $F(e,-)\dashv G(e,-)$, and suppose the resulting equivalences of hom-categories are part of a pseudonatural transformation
$$\phi\colon \cat C(F(-,d),c)\Rightarrow \cat D(d,G(-,c))$$
for every $c$ in $\cat C$, $d$ in $\cat D$. Then, the units and counits define extra-pseudonatural transformations
\[\eta_d\colon d\epn G(-,F(-,d))\]
and \[\varepsilon_c\colon F(-,G(-,c))\epn c.\]
From the natural isomorphism
\begin{center}
\begin{tikzcd}
{\cat C(F(e,d),c)} \arrow[r, "\phi_e"] \arrow[d, "{\cat C(F(f,d),c)}"'] & {\cat D(c,G(e,c))} \arrow[d, "{\cat D(c,G(f,c))}"] \arrow[ld, phantom, "\Arrowdl \phi_f" description] \\
{\cat C(F(e',d),c)} \arrow[r, "\phi_{e'}"']                             & {\cat D(c,G(e',c))}                                                                         
\end{tikzcd}
\end{center}
given by the parametric family, and since $\phi\colon g\mapsto R(c,g)\circ\eta^c_d$, we get for every $g\colon F(e,d)\to c$ an isomorphism $$(\phi_f)_g\colon G(f,e)\circ G(e,g)\circ\eta^c_d\overset{\sim}{\Rightarrow} G(e',g\circ F(f,d))\circ\eta^{c'}_d.$$
Now if we take $c=F(e,d)$ and $g=\id$, we specialize the isomorphism above to a pseudonatural transformation
\begin{center}
\begin{tikzcd}
d \arrow[r, "\eta^c_d"] \arrow[d, "\eta^{c'}_d"'] & {G(c,F(c,d))} \arrow[d, "{G(f,F(c,d))}"] \arrow[ld, phantom, "\Arrowdl" description] \\
{G(c',F(c',d))} \arrow[r, "{G(c',F(f,d))}"']      & {G(c',F(c,d))}                                                                 
\end{tikzcd}
\end{center}
which is indeed the data of an extra-pseudonatural transformation $\eta_d\colon d\epn G(-,F(-,d))$. Then, unitality, functoriality and naturality axioms for this structure are deduced from the correspondent axioms on the pseudonatural transformation $\phi$. The argument for $\varepsilon$ follows an analogous pattern.
\end{rmk}

The construction given in Remark \ref{parametric extra} above will be needed in order to prove the following:

\begin{prop}\label{correspondence psnat epn}
Let $\cat V$ be a right closed braided monoidal bicategory. The data of a $\cat V$-pseudonatural transformation of $\cat V$-pseudofunctors $t\colon F\Rightarrow G$, for $F,G\colon\cat C\to\cat D$, is the data of an extra-pseudonatural transformation $i\colon \mathbb{1}\epn\mathcal D(F-,G-),$ where $\mathbb 1$ stands for the constant pseudofunctor over $\mathbb 1$ in $\cat V$
\end{prop}

\begin{proof}
Let's construct the structure of an extra-pseudonatural transformation $i\colon \mathbb 1\epn \cat D(F-,G-)$ from the one of pseudonatural transformation $t\colon F\Rightarrow G$. First, the 1-dimensional structure is the very same, that is a 1-morphism $i_a=t_a\colon\mathbb 1\to\cat D(Fa,Ga)$ in $\cat V$. Then, observe that Remark \ref{parametric extra} gives, in the case of the tensor-hom parametric pseudoadjunction, an invertible 2-cell in $\cat V$
\begin{center}

\end{center}
for every triple of objects $x,a,b$ in $\cat V$ and $i\colon a\to b$. Therefore, we can define the structural 2-cell $i_{ab}$ as
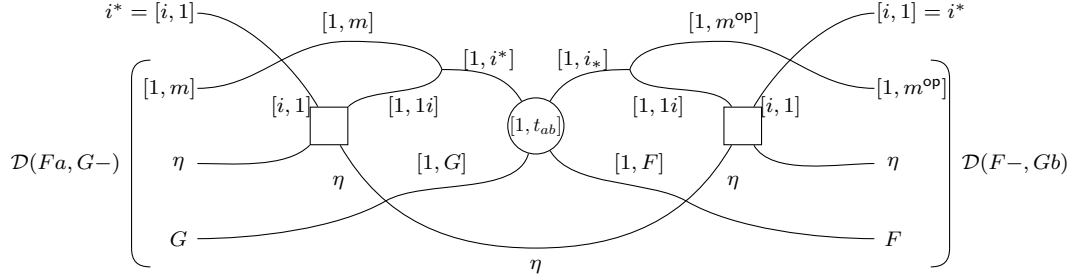
\begin{center}
%
\end{center}
Let us prove the axioms \eqref{EU}:
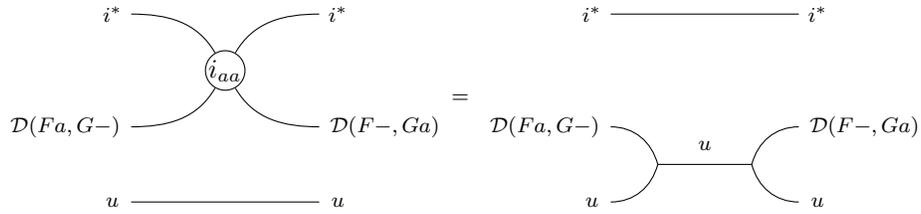
\begin{center}
%
\end{center}
The proof does not use more than axiom \eqref{psnat unitality} for $t\colon F\Rightarrow G$ and naturality of $\eta$. We start from the left-hand side, which is
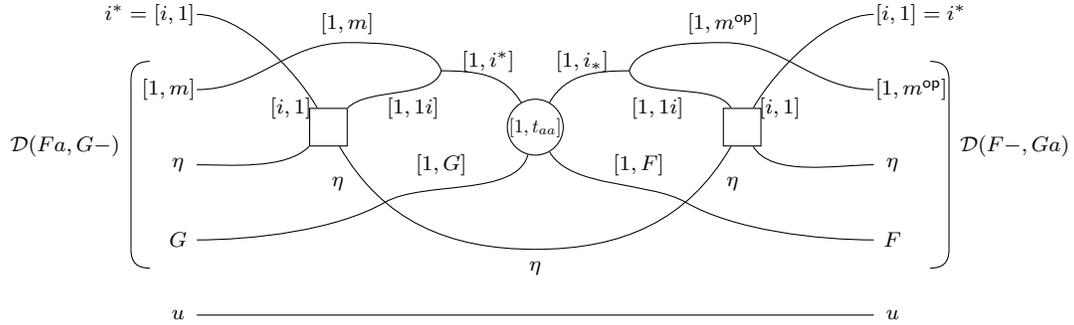
\begin{center}
%
\end{center}
and observe that we can slide $\eta$ on top of the 2-cell $[1,t_{aa}]$, in order to apply \eqref{psnat unitality}:
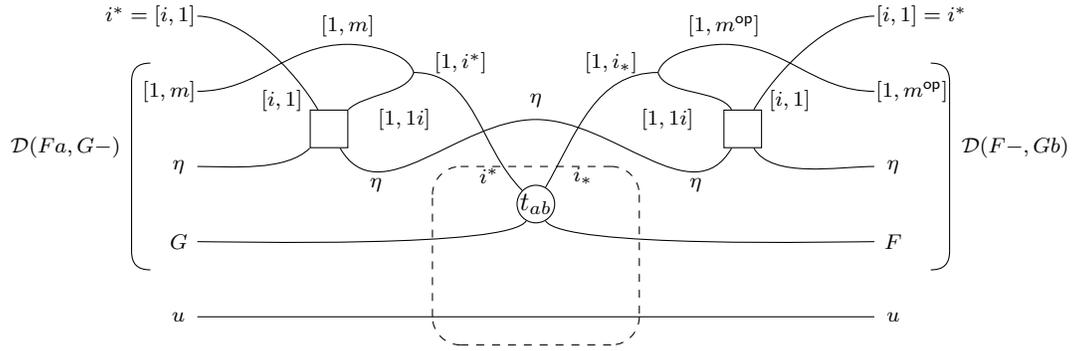
\begin{center}
%
\end{center}
Hence, by unitality for $t$ we get
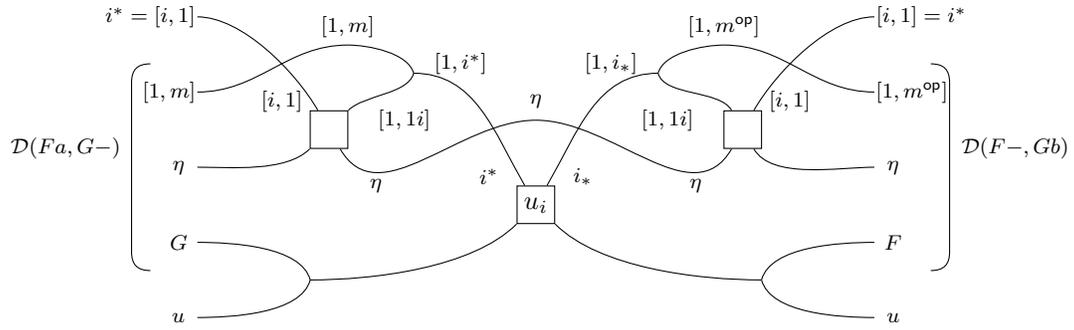
\begin{center}
%
\end{center}
If we now expand the definition of $u_i$, which is given by

\begin{minipage}{.05\textwidth}
\begin{equation}\label{def u structure}
\phantom{c}
\end{equation}
\end{minipage}
\begin{minipage}{.9\textwidth}
\begin{center}
%
\end{center}
\end{minipage}

\noindent
we find the left-hand side to have the following shape. Observe that $\beta$ ends up being equal to the identity since it refers to a monoidal unit.
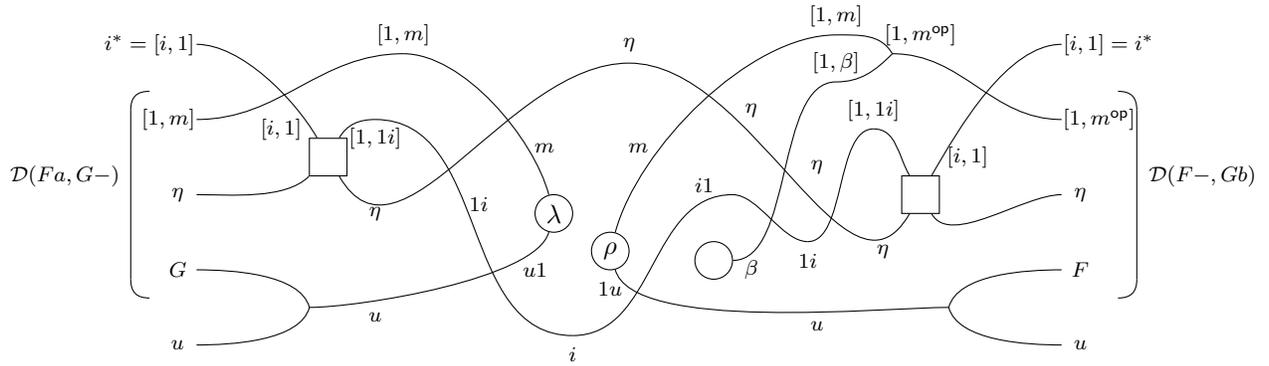
\begin{center}
%
\end{center}
Now, we can slide again $\eta$ below.
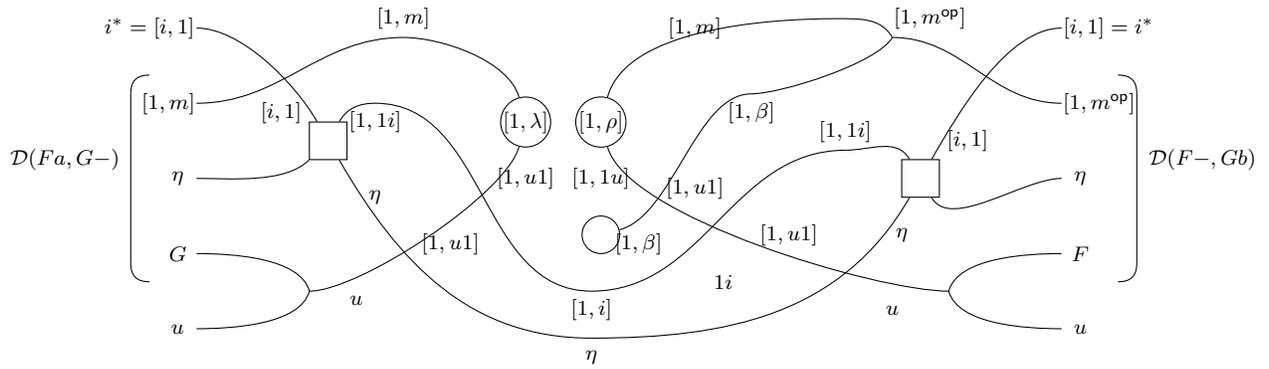
\begin{center}
%
\end{center}
Then, we can bring together the two squared morphisms, which are each other's inverse, and the 2-cell above becomes
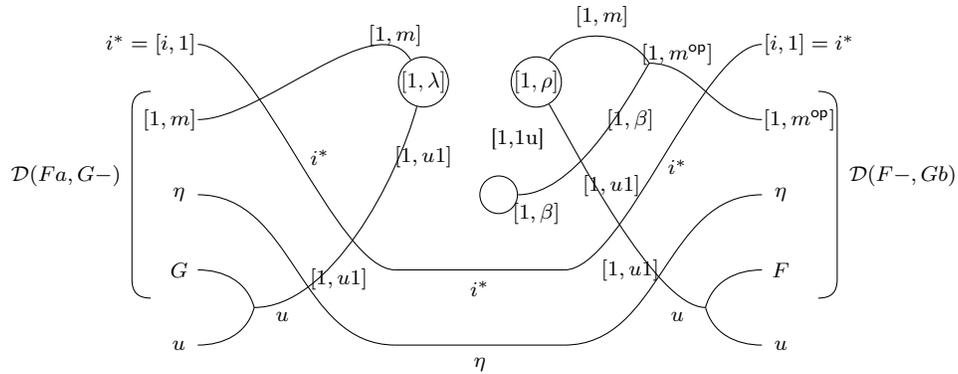
\begin{center}
%
\end{center}
If we eventually slide $i^\ast$ on top, we find precisely the right-hand side, since the highlighted parts below are by definition the unital structures for the hom-pseudofunctors (see Section \ref{subsect hom pseudo}).
\begin{center}
%
\end{center}
This concludes the proof for \eqref{EU}. The proof of axiom \eqref{EF} is not conceptually different and is left as an exercise to the reader. It basically consists of expanding the left-hand side via the functoriality axiom \eqref{psnat functoriality} and of playing with the naturality of $\eta$. It is convenient to also observe how the right-hand side of \eqref{EF} can be simplified by mean of the definition of $\nu$, in which the 2-cell $\rho^{-1}\lambda^{-1}$ cancels out with $\lambda\rho$.
\end{proof}

\section{Enriched bi(co)ends}\label{section bicoends}

In this section we introduce the notion of end and coend in the enriched bicategorical context. As customary in the enriched context, we will be able to define a good notion by first defining it when valued in the ground braided monoidal bicategory $\cat V$, and then by representably extending to the general case. Our first object of interest are then $\cat V$-pseudofunctors of type $\cat B\op\otimes\cat B\longrightarrow\cat V$, therefore assuming $\cat V$ to be a right closed braided monoidal bicategory.

\begin{defn}\label{def Vbicoend}
Let $\cat V$ be a right closed braided monoidal bicategory, and let $P\colon\cat{B}\op\otimes\cat{B}\to\cat{V}$ be a $\cat V$-pseudofunctor. An \emph{(enriched) bicoend} of $P$ is, if it exists, an object \[\int^bP(b,b)\]
of $\cat{V}$ together with an (enriched) extra-pseudonatural transformation $i\colon P\epn\displaystyle\int^bP(b,b)$ (to the constant pseudofunctor, see Remark \ref{Iop times I}) such that the following two axioms hold true.
\begin{enumerate}
    \item[(BC1)] For any object $x$ in $\cat{V}$ and any extra-pseudonatural transformation $j\colon P\epn x$ into the constant $\cat V$-pseudofunctor over $x$, there's a 1-cell $\tilde{j}\colon\int^bP(b,b)\to x$, and for every $a$ in $\cat{B}$ there is a 2-isomorphism $J_a$
    \begin{center}
\begin{tikzcd}
                                              & {\int^bP(b,b)} \arrow[rd, "\tilde{j}"] \arrow[d, phantom, "\Arrowd J_a" description] &   \\
{P(a,a)} \arrow[ru, "i_a"] \arrow[rr, "j_a"'] & {}                                                                               & x
\end{tikzcd}
    \end{center}
satisfying the following identity of 2-cells:
\begin{center}
\begin{tikzcd}[column sep=0em]
{\mathcal B(a,c)} \arrow[r, "{P(-,a)}"] \arrow[d, "{P(c,-)}"'] & {[P(c,a),P(a,a)]} \arrow[rd, "{i_a}_*"] \arrow[dd, "{j_a}_*"'] &                                                                                                                    \\
{[P(c,a),P(c,c)]} \arrow[rd, "{j_c}_*"']             & {} \arrow[lu, phantom, "\Arrowdl j_{ac}" description]                             & {[P(c,a),\int^bP(b,b)]} \arrow[ld, "\tilde{j}_*"] \arrow[l, phantom, "\adorn{\Leftarrow}{{J_a}_*}" description] \\
                                                               & {[P(c,a),x]}                                                   &                                                                                                                   
\end{tikzcd}
\end{center}  
\begin{center}
=
\end{center}  
\begin{center}
\begin{tikzcd}[column sep=0em]
{\mathcal B(a,c)} \arrow[r, "{P(-,a)}"] \arrow[d, "{P(c,-)}"']           & {[P(c,a),P(a,a)]} \arrow[rd, "{i_a}_*"]                                      &                                                             \\
{[P(c,a),P(c,c)]} \arrow[rd, "{j_c}_*"'] \arrow[rr, "{i_c}_*"] & {} \arrow[lu, phantom, "\Arrowdl i_{ac}" description] \arrow[d, phantom, "\Arrowdl {J_c}_*" description] & {[P(c,a),\int^bP(b,b)]} \arrow[ld, "\tilde{j}_*"] \\
                                                                         & {[P(c,a),x]}                                                                 &                                                            
\end{tikzcd}
\end{center}
\item[(BC2)]\label{BC2} For every pair of 1-cells $h,k\colon\int^bP(b,b)\to x$ and any family of 2-cells $\Gamma_a\colon hi_a\Rightarrow ki_a$ such that
    
\begin{center}
\begin{tikzcd}[column sep=0em]
{ \mathcal B(a,c)} \arrow[r, "{P(-,a)}"] \arrow[d, "{P(c,-)}"'] & {[P(c,a),P(a,a)]} \arrow[rd, "{i_a}_*"] \arrow[dd, "{i_a}_*"] \arrow[ld, phantom, "{\Arrowdl i_{ac}}" description, bend left=10, pos=0.6] &                                                                                                       \\
{[P(c,a),P(c,c)]} \arrow[rd, "{i_c}_*"']              &                                                                                                                      & {[P(c,a),\int^bP(b,b)]} \arrow[d, "h_\ast"] \arrow[ld, phantom, "\Arrowdl {\Gamma_a}_*" description, bend right=10, pos=0.4] \\
                                                                & {[P(c,a),\int^bP(b,b)]} \arrow[r, "k_*"']                                                                  & {[P(c,a),x]}                                                                               
\end{tikzcd}  
\end{center}
  
\begin{center}
=
\end{center}
\begin{center}
\begin{tikzcd}[column sep=0em]
{ \mathcal B(a,c)} \arrow[r, "{P(-,a)}"] \arrow[d, "{P(c,-)}"']          & {[P(c,a),P(a,a)]} \arrow[rd, "{i_a}_*"] \arrow[ld, phantom, "{\Arrowdl i_{ac}}" description, bend left=10] &                                                                                                                  \\
{[P(c,a),P(c,c)]} \arrow[rd, "{i_c}_*"'] \arrow[rr, "{i_c}_*"] &                                                                                                           & {[P(c,a),\int^bP(b,b)]} \arrow[d, "h_\ast"] \arrow[ld, phantom, "\Arrowdl{\Gamma_c}_*" description, bend right=10] \\
                                                                         & {[P(c,a),\int^bP(b,b)]} \arrow[r, "k_*"']                                                       & {[P(c,a),x]}                                                                                          
\end{tikzcd}
\end{center}
there is a unique 2-cell $\gamma\colon h\Rightarrow k$ such that $\Gamma_a=\gamma\ast i_a$, namely whiskering of the 2-cell $\gamma$ with the 1-cell $i_a$.
\end{enumerate}
\end{defn}

\begin{rmk}
One can see that an extra-pseudonatural transformation $i\colon P\epn Q$ between $\cat V$-pseudofunctors
\begin{align*}
P\colon\cat E\otimes\cat B\op\otimes\cat B\longrightarrow\cat D\\
Q\colon\cat E\otimes\cat C\op\otimes\cat C\longrightarrow\cat D
\end{align*}
consists of the same data of an extra-pseudonatural transformation $P\op\epn Q\op$ between the $\cat V$-pseudofunctors
\begin{align*}
P\op\colon\cat E\op\otimes(\cat B\op)\op\otimes\cat B\op\longrightarrow\cat D\op\\
Q\op\colon\cat E\op\otimes(\cat C\op)\op\otimes\cat C\op\longrightarrow\cat D\op.
\end{align*}
This is a straightforward observation based on Remarks \ref{oponfunctors} and \ref{opdistributeovertensor}.
\end{rmk}

\begin{defn}\label{def Vbiend}
Let $P\colon \cat B\op\otimes\cat B\to\cat V$ be a $\cat V$-pseudofunctor. An object $\int_b P(b,b)$ in $\cat V$ together with an enriched extra-pseudonatural transformation $\int_bP(b,b)\epn P$ defines a \emph{biend} if the equivalent data
$$P\op\epn \int_b P(b,b)$$
defines a bicoend.
\end{defn}

At this stage, the two bicoend axioms, which may look kind of obscure at first, deserve an explanation that will make them much clearer. The key point lies in the following proposition, which provides a higher point of view on bi(co)ends by defining them as representing objects.
\begin{prop}\label{(co)end representing}
Let $P\colon \cat B\op\otimes\cat B\to\cat V$ be a pseudofunctor admitting a bicoend. Then, there is for every object $x$ in $\cat V$ an equivalence, pseudonatural in $x$, between the categories
\begin{equation}\label{(co)end representing 1}
\cat V(\int^bP(b,b),x)\simeq\cat V\text{-}\psnat^\mathsf{e}(P,x).
\end{equation}Dually, if $P$ admits a biend, there is for every $x$ an equivalence
\begin{equation}\label{(co)end representing 2}
\cat V(x,\int_bP(b,b))\simeq\cat V\text{-}\psnat^\mathsf{e}(x,P).
\end{equation}
Conversely, such representing objects define bi(co)ends.
\end{prop}
\begin{proof}
Let us prove the first statement, and call $i\colon P\epn\int^bP_{bb}$ the bicoend of $P$. The dual statement will clearly follow a symmetric argument. The equivalence is here explicitly given as the precomposition
\begin{align}\label{bicoend represent equiv}
\cat V(\int^bP(b,b),x)&\longrightarrow\cat V\text{-}\psnat^\mathsf{e}(P,x)
\end{align}
defined on an object $k\colon\int^bP(b,b)\to x$ to be the extra-pseudonatural transformation $k\circ i$ with components $(k\circ i)_b=k\circ i_b$. On morphisms $\gamma\colon k\Rightarrow h$, the obvious whiskering $\gamma\ast i$, with components given by $\gamma\ast i_b\colon ki_b\Rightarrow hi_b$. One needs to check that $\gamma\ast i$ actually defines a morphism of extra-pseudonatural transformations, which is the equality
\begin{center}
\begin{tikzcd}[row sep=5em, column sep=3.5em]
	{\mathcal B(b,b')} & {[Pb'b,Pbb]} \\
	{[Pb'b,Pb'b']} & {[Pb'b,x]}
	\arrow["{P(-,b)}", from=1-1, to=1-2]
	\arrow["{P(b',-)}"', from=1-1, to=2-1]
	\arrow["{{k i_b}_*}", from=1-2, to=2-2]
	\arrow[phantom, "{\Arrowdl (k i)_{bb'}}"{description}, draw=none, from=2-1, to=1-2]
	\arrow[""{name=0, anchor=center, inner sep=0}, "{{k i_{b'}}_*}", from=2-1, to=2-2]
	\arrow[""{name=1, anchor=center, inner sep=0}, "{{h i_{b'}}_*}"', curve={height=40pt}, from=2-1, to=2-2]
	\arrow[phantom, "{\Arrowd{\gamma i_{b'}}_*}"{description}, draw=none, from=0, to=1]
\end{tikzcd}
=
\begin{tikzcd}[row sep=5em, column sep=3.5em]
	{\mathcal B(b,b')} & {[Pb'b,Pbb]} \\
	{[Pb'b,Pb'b']} & {[Pb'b,x]}
	\arrow["{P(-,b)}", from=1-1, to=1-2]
	\arrow["{P(b',-)}"', from=1-1, to=2-1]
	\arrow[""{name=0, anchor=center, inner sep=0}, "{{k i_b}_*}", curve={height=-40pt}, from=1-2, to=2-2]
	\arrow[""{name=1, anchor=center, inner sep=0}, "{{hi_b}_*}"', from=1-2, to=2-2]
	\arrow[phantom, "{\Arrowdl (h i)_{bb'}}"{description}, draw=none, from=2-1, to=1-2]
	\arrow["{{h i_{b'}}_*}"', from=2-1, to=2-2]
	\arrow[phantom, "{\adorn{\Leftarrow}{{\gamma i_b}_\ast}}"{description}, draw=none, from=0, to=1]
\end{tikzcd}
\end{center}
This is true, since expanding the composition $(ki_b)_\ast$ and $(hi_{b'})\ast$ we find both squares to be equal to the horizontal composition $\gamma_*i_{bb'}$.
\begin{center}
\begin{tikzcd}[column sep = 2.5em]
	& {[Pb'b,Pbb]} \\
	{\mathcal B(b,b')} && {[Pb'b,\int^cPcc]} & {[Pb'b,x]} \\
	& {[Pb'b,Pb'b']}
	\arrow["{{i_b}_*}", from=1-2, to=2-3]
	\arrow[phantom, "{\Arrowd i_{bb'}}"{description}, draw=none, from=1-2, to=3-2]
	\arrow["{P(-,b)}", from=2-1, to=1-2]
	\arrow["{P(b',-)}"', from=2-1, to=3-2]
	\arrow[""{name=0, anchor=center, inner sep=0}, "{h_*}"', curve={height=18pt}, from=2-3, to=2-4]
	\arrow[""{name=1, anchor=center, inner sep=0}, "{k_*}", curve={height=-18pt}, from=2-3, to=2-4]
	\arrow["{{i_{b'}}_*}"', from=3-2, to=2-3]
	\arrow[phantom, "{\Arrowd\gamma_*}"{description}, draw=none, from=1, to=0]
\end{tikzcd}
\end{center}
Fully faithfulness of \eqref{bicoend represent equiv} is now precisely axiom (BC2), which states that for all families $\Gamma_b\colon ki_b\Rightarrow hi_b$ defining a morphism of extra-pseudonatural transformations there is a unique morphism $h\Rightarrow k$ in the domain category which is mapped to it by precomposition with $i$. On the other hand, axiom (BC1) says that for all objects $j\colon P\epn x$ in the codomain category, there exists a $\tilde{j}\colon \int^bPbb\to x$ in the domain and an isomorphism of extra-pseudonatural transformations $J\colon \tilde{j}\circ i\cong j$, which is precisely essential surjectivity.
\end{proof}

\subsection{Arbitrary-valued bi(co)ends}
Now, the general definition of a bi(co)end valued in any $\cat V$-bicategory $\cat D$ can be given representably. As before, we suppress the parametric variable in $\cat E$ for simplicity.
\begin{defn}
Let $P\colon\cat B\op\otimes\cat B\to\cat D$ be a $\cat V$-pseudofunctor, where $\cat D$ is any $\cat V$-bicategory. A \emph{biend} for $P$ is an object $\int_bP(b,b)$ in $\cat D$ together with an extra-pseudonatural transformation $i\colon\int_bP(b,b)\epn P$ such that for every object $d$ in $\cat D$ the extra-pseudonatural transformation
$$\cat D(d,i)\colon\cat D(d,\int_bP(b,b))\epn\cat D(d,P(-,-))$$
is a biend for $\cat D(d,P(-,-))\colon\cat{B}\op\otimes\cat{B}\to\cat V$, in the sense of Definition \ref{def Vbiend}.

A \emph{bicoend} for $P$ is an object $\int^bP(b,b)$ in $\cat D$ together with an extra-pseudonatural transformation $i\colon P\epn\int^bP(b,b)$ such that for every object $d$ in $\cat D$ the extra-pseudonatural transformation $$\cat D(i,d)\colon\cat D(\int^bP(b,b),d)\epn\cat D(P(-,-),d)$$
is a biend for $\cat D(P(-,-),d)\colon\cat{B}\op\otimes\cat{B}\to\cat V.$
\end{defn}
It immediately follows that we have a pair of equivalences in $\cat V$
\begin{equation}\label{(co)continuity}
\cat D(d,\int_bP(b,b))\simeq\int_b\cat D(d,P(b,b))
\end{equation}
\begin{equation}\label{(co)continuity 2}
\cat D(\int^bP(b,b),d)\simeq\int_b\cat D(P(b,b),d).
\end{equation}
pseudonatural in $d$.


\begin{rmk}\label{extrapsnat object}
A natural next step would be to have the analogue result of Proposition \ref{(co)end representing} in the case of $\cat D$-valued bi(co)ends for an arbitrary $\cat V$-bicategory $\cat D$. In order to do so, we need to upgrade the category of bicowedges $\cat V$-$\psnat^\mathsf{e}(P,d)$ to an object in $\cat V$. In other words, we want an object $\underline{\cat V\text{-}\psfun^\mathsf{e}}(P,d)$ of $\cat V$ giving the desired equivalence
$$\cat D(\int^bP(b,b),d)\simeq \underline{\cat V\text{-}\psnat^\mathsf{e}}(P,d).$$
for every object (and constant pseudofunctor) $d$. The obvious choice is hence by means of \eqref{(co)continuity 2}, to define 
\[\underline{\cat V\text{-}\psfun^\mathsf{e}}(P,d)\coloneqq \int_b\mathcal{D}(P(b,b),d).\]
\end{rmk}

\subsection{The enriched pseudofunctor bicategory}\label{section enriched pseudofunctor bicat}

In the non-enriched setting it is customary and straightforward to prove that if $L,S\colon\cat C\to \cat D$ are pseudofunctors, then the biend of the pseudofunctor $\cat D(L-,S-)\colon \cat C\op\times\cat C\to\mathsf{Cat}$ exists and is given up to a canonical equivalence by $\psnat(L,S)$. By virtue of this result, we can enrich the \emph{bicategory of enriched pseudofunctors}. The following construction will provide the structure inducing unit and composition for hom-objects.

\begin{construction}\label{unit comp epn}
We subsequently explain how to construct the data for the $\cat V$-pseudofunctor enriched bicategory for two $\cat V$-bicategories $\cat C,\cat D$. Whenever $F,G,H\colon\cat C\to \cat D$ are $\cat V$-pseudofunctors, and if the relevant following biends exist, we look for extra-pseudonatural transformations $$\underline u_F\colon\mathbb{1}\epn\mathcal{D}(F-,F-)$$
and
$$\underline m_{F,G,H}\colon \int_c\mathcal D(Gc,Hc)\otimes\int_c\mathcal D(Fc,Gc)\epn\mathcal{D}(F-,H-).$$
For what concerns $u_F$, one can define it via Proposition \ref{correspondence psnat epn} to be the same data of the $\cat V$-pseudonatural $\id\colon F\Rightarrow F$. Hence, one has each 1-component at $d$ given as \[{(\underline u_F)}_d\coloneqq u_{(Fd)}\colon\mathbb{1}\to\cat D(Fd,Fd).\]
Each $d$-component of $\underline{m}$ is the composition
\[\underline m_d\colon\int_c\mathcal D(Gc,Hc)\otimes\int_c\mathcal D(Fc,Gc)\overset{k_d\otimes j_d}\longrightarrow\mathcal D(Gd,Hd)\otimes\mathcal D(Fd,Gd)\overset{m}\longrightarrow\cat D(Fd,Hd)\]
and 2-isomorphisms are built using the structure of extra-pseudonaturality of $k$ and $j$. The 2-cell $\underline m_{d,d'}$
\begin{center}
\begin{tikzcd}
	{\cat C(d,d')} & {[\cat D(Fd',Hd'),\cat D(Fd,Hd')]} \\
	{[\cat D(Fd,Hd),\cat D(Fd,Hd')]} & {[\int_c\cat D(Gc,Hc)\otimes\int_c\cat D(Fc,Gc),\cat D(Fd,Hd')]}
	\arrow["{\cat D(F-,Hd')}", from=1-1, to=1-2]
	\arrow["{\cat D(Fd,H-)}"', from=1-1, to=2-1]
	\arrow["{\Arrowdl \underline m_{dd'}}"{description}, draw=none, from=1-2, to=2-1]
	\arrow["{\underline m^*}", from=1-2, to=2-2]
	\arrow["{\underline m^*}"', from=2-1, to=2-2]
\end{tikzcd}
\end{center}
is given by the diagram
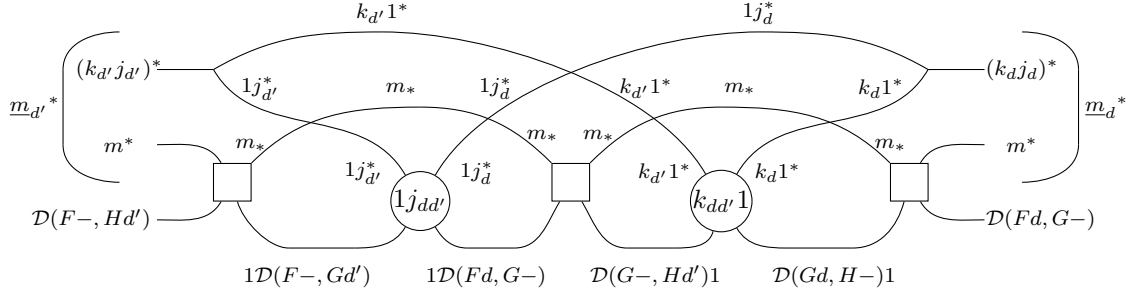
\begin{center}
\begin{tikzpicture}[scale=.5]
		\node [style=none] (0) at (-11, 3) {};
		\node [style=none] (1) at (-11, 1) {};
		\node [style=none] (2) at (-11, -1) {};
		\node [style=none] (3) at (-12, 3) {$(k_{d'}j_{d'})^*$};
		\node [style=none] (4) at (-12, 1) {$m^\ast$};
		\node [style=none] (5) at (-12.75, -1) {$\mathcal D(F-,Hd')$};
		\node [style=none] (6) at (12, 3) {$(k_dj_d)^\ast$};
		\node [style=none] (7) at (12, 1) {$m^\ast$};
		\node [style=none] (8) at (12.5, -1) {$\mathcal D(Fd,G-)$};
		\node [style=none] (9) at (11, 3) {};
		\node [style=none] (10) at (11, 1) {};
		\node [style=none] (11) at (11, -1) {};
		\node [style=none] (12) at (-12, 4) {};
		\node [style=none] (13) at (-13.5, 3) {};
		\node [style=none] (14) at (-12, 0) {};
		\node [style=none] (15) at (-13.5, 1) {};
		\node [style=none] (16) at (-9.5, 3) {};
		\node [style=none] (17) at (9.5, 3) {};
		\node [style=fixed size node] (18) at (-4, -0.5) {$1j_{dd'}$};
		\node [style=fixed size node] (19) at (4, -0.5) {$k_{dd'}1$};
		\node [style=squared] (20) at (-9, 0) {};
		\node [style=squared] (21) at (0, 0) {};
		\node [style=squared] (22) at (9, 0) {};
		\node [style=none] (23) at (-4, 2) {};
		\node [style=none] (24) at (4, 2) {};
		\node [style=none] (25) at (5, 4) {};
		\node [style=none] (26) at (-5, 4) {};
		\node [style=none] (27) at (-14.25, 2) {${\underline m_{d'}}^\ast$};
		\node [style=none] (28) at (-8.25, 2.5) {$1j_{d'}^\ast$};
		\node [style=none] (29) at (-5, 4.5) {$k_{d'}1^\ast$};
		\node [style=none] (30) at (-7, -2.5) {$1\mathcal D(F-,Gd')$};
		\node [style=none] (31) at (-2.25, -2.5) {$1\mathcal D(Fd,G-)$};
		\node [style=none] (32) at (2.25, -2.5) {$\mathcal D(G-,Hd')1$};
		\node [style=none] (33) at (7, -2.5) {$\mathcal D(Gd,H-)1$};
		\node [style=none] (34) at (-5.5, 0.25) {$1j_{d'}^\ast$};
		\node [style=none] (35) at (-2.5, 0.25) {$1j_{d}^\ast$};
		\node [style=none] (36) at (-8.5, 1) {$m_\ast$};
		\node [style=none] (37) at (-4.5, 2.5) {$m_\ast$};
		\node [style=none] (38) at (-0.75, 1.25) {$m_\ast$};
		\node [style=none] (39) at (0.75, 1.25) {$m_\ast$};
		\node [style=none] (40) at (4.5, 2.5) {$m_\ast$};
		\node [style=none] (41) at (8.5, 1) {$m_\ast$};
		\node [style=none] (42) at (12, 4) {};
		\node [style=none] (43) at (13.5, 3) {};
		\node [style=none] (44) at (13.5, 1) {};
		\node [style=none] (45) at (12, 0) {};
		\node [style=none] (47) at (14.25, 2) {${\underline m_d}^\ast$};
		\node [style=none] (48) at (-2, 2.5) {$1j_{d}^\ast$};
		\node [style=none] (49) at (5, 4.5) {$1j_{d}^\ast$};
		\node [style=none] (50) at (2, 2.5) {$k_{d'}1^\ast$};
		\node [style=none] (51) at (2.5, 0.25) {$k_{d'}1^\ast$};
		\node [style=none] (52) at (5.5, 0.25) {$k_{d}1^\ast$};
		\node [style=none] (53) at (8.25, 2.5) {$k_{d}1^\ast$};
		\node [style=none] (54) at (-7.5, -1.75) {};
		\node [style=none] (55) at (-5.75, -1.75) {};
		\node [style=none] (56) at (7.5, -1.75) {};
		\node [style=none] (57) at (5.75, -1.75) {};
		\node [style=none] (58) at (-2.75, -1.75) {};
		\node [style=none] (59) at (-1.5, -1.75) {};
		\node [style=none] (60) at (1.5, -1.75) {};
		\node [style=none] (61) at (2.75, -1.75) {};

		\draw [in=270, out=180] (14.center) to (15.center);
		\draw (15.center) to (13.center);
		\draw [in=180, out=90] (13.center) to (12.center);
		\draw [in=240, out=0] (2.center) to (20);
		\draw [in=120, out=0] (1.center) to (20);
		\draw [in=180, out=45] (20) to (23.center);
		\draw [in=135, out=0] (23.center) to (21);
		\draw [in=45, out=-180] (24.center) to (21);
		\draw [in=135, out=0] (24.center) to (22);
		\draw [in=60, out=-120] (17.center) to (19);
		\draw [in=180, out=60, looseness=0.75] (18) to (25.center);
		\draw [in=120, out=0, looseness=0.75] (26.center) to (19);
		\draw [in=180, out=0] (0.center) to (16.center);
		\draw [in=180, out=30] (16.center) to (26.center);
		\draw [in=120, out=-60] (16.center) to (18);
		\draw [in=150, out=0] (25.center) to (17.center);
		\draw [in=180, out=0] (17.center) to (9.center);
		\draw [in=180, out=60] (22) to (10.center);
		\draw [in=180, out=-60] (22) to (11.center);
		\draw [in=-90, out=0] (45.center) to (44.center);
		\draw (44.center) to (43.center);
		\draw [in=0, out=90] (43.center) to (42.center);
		\draw [in=240, out=0] (55.center) to (18);
		\draw [in=285, out=-180, looseness=0.75] (54.center) to (20);
		\draw [in=180, out=0] (54.center) to (55.center);
		\draw [in=300, out=180] (57.center) to (19);
		\draw [in=180, out=0] (57.center) to (56.center);
		\draw [in=-105, out=0] (56.center) to (22);
		\draw [in=0, out=-105] (19) to (61.center);
		\draw [in=360, out=180] (61.center) to (60.center);
		\draw [in=300, out=-180] (60.center) to (21);
		\draw [in=0, out=-120] (21) to (59.center);
		\draw [in=360, out=180] (59.center) to (58.center);
		\draw [in=-60, out=180] (58.center) to (18);
\end{tikzpicture}
\end{center}
The squared 2-cells come from the pseudonaturality of $m$, and can be derived from \eqref{pseudonaturality of m} (the middle one) and \eqref{m crossing} (the side ones). The extra-pseudonatural transformation axioms then follow from the corresponding axioms for $j$ and $k$.
\end{construction}
With the previous construction we are able to enhance the structure of the bicategory $\cat V\text{-}\psfun(\cat C,\cat D)$ to a $\cat V$-bicategory, which we refer as $\llbracket\cat C,\cat D\rrbracket$, in order to distinguish from the plain underlying bicategory (see Remark \ref{underlying bicat}).
\begin{defn}[The $\cat V$-bicategory of $\cat V$-pseudofunctors]\label{A-PsNat}
Let $\cat V$ be a right closed braided monoidal bicategory and $\cat C,\cat D$ be two $\cat V$-bicategories, then $\llbracket\cat C,\cat D\rrbracket$ is - if the relevant biends exist - the $\cat V$-bicategory whose:
\begin{itemize}[leftmargin=*]
\item objects are the $\cat V$-pseudofunctors $F\colon\cat C\to\cat D,$
\item each hom-object $\llbracket\cat C,\cat D\rrbracket(F,G)$ is the object in $\cat V$ defined by the biend
\begin{equation*}
\int_c \cat D(Fc,Gc)
\end{equation*}
of the $\cat V$-pseudofunctor $\cat D(F-,G-)\colon\cat C\op\otimes\cat C\to\cat V$.
\item The unit is defined via the extra-pseudonatural transformation $\underline u_F\colon\mathbb{1}\epn\mathcal{D}(F-,F-)$ of Construction \ref{unit comp epn}, which provides a unique pair $(u_F,\{U_d\})$
\begin{equation}\label{unit psfun bicat}
\begin{tikzcd}
\mathbb 1 \arrow[rd, "u_{Fd}"'] \arrow[rr, "u_F"] & {} \arrow[d, phantom, "\Arrowur U_d" description] & {\int_c\mathcal D(Fc,Fc)} \arrow[ld, "i_d"] \\
                                                  & {\mathcal D(Fd,Fd)}                      &                                            
\end{tikzcd}
\end{equation}
\item The composition is defined analogously as induced by the other extra-pseudonatural transformation $\underline m$ of Construction \ref{unit comp epn}. That is,
\begin{equation}\label{composition psfun bicat}
\begin{tikzcd}
	{\int_c\mathcal D(Gc,Hc)\otimes\int_c\mathcal D(Fc,Gc)} && {\int_c\mathcal D(Fc,Hc)} \\
	& {\mathcal D(Fd,Hd)}
	\arrow[""{name=0, anchor=center, inner sep=0}, "{\tilde{\underline{m}}}", from=1-1, to=1-3]
	\arrow["{\underline m_d}"', from=1-1, to=2-2]
	\arrow["{i_d}", from=1-3, to=2-2]
	\arrow[phantom, "{\Arrowur M_d}"{description}, draw=none, from=2-2, to=0]
\end{tikzcd}
\end{equation}
\item The left unitor $\underline{\lambda}$ is defined via the equivalence \eqref{(co)end representing 2} by defining (each component of) a morphism of extra-pseudonatural transformations (biwedges) by composing the inverses of $U$ and $M$, and the unitor $\lambda$ for $\cat D$ as follows.
\begin{equation}
\begin{tikzcd}[column sep=0.5em]
	{\mathbb 1\otimes\int_c\mathcal D(Fc,Gc)} \\
	{\mathcal D(Fd,Fd)\otimes\int_c\mathcal D(Fc,Gc)} & {\mathbb 1\otimes\mathcal D(Fd,Gd)} & {\mathcal D(Fd,Gd)} && {\int_c\mathcal D(Fc,Gc)} \\
	& {\mathcal D(Fd,Fd)\otimes\mathcal D(Fd,Gd)} \\
	{\int_c\mathcal D(Fc,Fc)\otimes\int_c\mathcal D(Fc,Gc)}
	\arrow["{u_{Fd}1}", from=1-1, to=2-1]
	\arrow["{1j_d}", from=1-1, to=2-2]
	\arrow[""{name=0, anchor=center, inner sep=0}, equals, curve={height=-30pt}, from=1-1, to=2-5]
	\arrow[""{name=1, anchor=center, inner sep=0}, "{u_F1}"', bend right, out=-80, in=-100, from=1-1, to=4-1]
	\arrow[phantom, "\Arrowur\Sigma^{-1}"{description}, draw=none, from=2-1, to=2-2]
	\arrow["{1j_d}", from=2-1, to=3-2]
	\arrow[""{name=2, anchor=center, inner sep=0}, equals, from=2-2, to=2-3]
	\arrow["{u_{Fd}1}"', from=2-2, to=3-2]
	\arrow["{j_d}"', from=2-5, to=2-3]
	\arrow[""{name=3, anchor=center, inner sep=0}, "m"', from=3-2, to=2-3]
	\arrow["{i_d1}", from=4-1, to=2-1]
	\arrow[""{name=4, anchor=center, inner sep=0}, "{\tilde {\underline m}}"', bend right, from=4-1, to=2-5]
	\arrow[""{name=5, anchor=center, inner sep=0}, "{i_d\otimes j_d}", from=4-1, to=3-2]
	\arrow["="{description}, draw=none, from=0, to=2-2]
	\arrow["\cong"{description}, draw=none, from=2-1, to=5]
	\arrow["{\Arrowur U_d^{-1}1}"{description}, draw=none, bend left, from=2-1, to=1]
	\arrow["{\Arrowur\lambda}"{description}, draw=none, from=3-2, to=2]
	\arrow["\adorn{\Leftarrow}{M_d^{-1}}"{description}, draw=none, from=3, to=4, pos=.7]
\end{tikzcd}
\end{equation}
Similarly, the right unitor.
\item The associator $\underline \alpha$ is analogously built as corresponding to a morphism of extra-pseudonatural transformations (biwedges). The latter having components as follows, where the arrows without a name are just the tensor products of the structural biend morphisms.
\begin{equation}
\begin{tikzcd}
	{\int_c\mathcal D(Gc,Hc)\int_c\mathcal D(Fc,Gc)\int_c\mathcal D(Ec,Fc)} & {\int_c\mathcal D(Gc,Hc)\int_c\mathcal D(Ec,Gc)} \\
	{\mathcal D(Gd,Hd)\mathcal D(Fd,Gd)\mathcal D(Ed,Fd)} & {\mathcal D(Gd,Hd)\mathcal D(Ed,Gd)} \\
	{\mathcal D(Fd,Hd)\mathcal D(Ed,Fd)} & {\mathcal D(Ed,Hd)} \\
	{\int_c\mathcal D(Fc,Hc)\int_c\mathcal D(Ec,Fc)} & {\int_c\mathcal D(Ec,Hc)}
	\arrow[""{name=0, anchor=center, inner sep=0}, "{1\tilde{\underline m}}", from=1-1, to=1-2]
	\arrow[from=1-1, to=2-1]
	\arrow[""{name=1, anchor=center, inner sep=0}, "{\tilde{\underline m}1}"', bend right, from=1-1, to=4-1, out=-80, in=-100]
	\arrow[from=1-2, to=2-2]
	\arrow[""{name=2, anchor=center, inner sep=0}, "{\tilde{\underline m}}", bend left, from=1-2, to=4-2, out=80, in=100]
	\arrow[""{name=3, anchor=center, inner sep=0}, "1m", from=2-1, to=2-2]
	\arrow[""{name=4, anchor=center, inner sep=0}, "m1", from=2-1, to=3-1]
	\arrow[""{name=5, anchor=center, inner sep=0}, "m"', from=2-2, to=3-2]
	\arrow["{\Arrowur\alpha}"{description}, draw=none, from=3-1, to=2-2]
	\arrow[""{name=6, anchor=center, inner sep=0}, "m", from=3-1, to=3-2]
	\arrow[from=4-1, to=3-1]
	\arrow[""{name=7, anchor=center, inner sep=0}, "{\tilde{\underline m}}"', from=4-1, to=4-2]
	\arrow["{k_d}"', from=4-2, to=3-2]
	\arrow["{\Arrowdr M_d}"{description}, draw=none, from=2, to=5]
	\arrow["{\Arrowur 1M_d}"{description}, draw=none, from=3, to=0]
	\arrow["{\Arrowur M_d^{-1}1}"{description}, draw=none, from=4, to=1]
	\arrow["{\Arrowul M_d^{-1}}"{description}, draw=none, from=7, to=6]
\end{tikzcd}
\end{equation}
\end{itemize} 
The coherence identities \eqref{IC Vbicat} and \eqref{AC Vbicat} can then easily be proved from these definitions, and then be transposed again via the equivalence of categories \eqref{(co)end representing 2}.
\end{defn}

\section{Future perspectives}

The work presented is part of my Ph.D. thesis \cite{mythesis}. In there, this bicategorical enriched machinery has been applied to the context of Mackey pseudofunctors. Mackey pseudofunctors, whose theory has been introduced in \cite{Balmer_2020}, arise naturally in the representation theory of finite groups, and in that work from Balmer--Dell'Ambrogio it is shown that they can be seen as $\cat V$-pseudofunctor over a universal $\cat V$-bicategory for the braided monoidal bicategory $\cat V=\mathsf{Add}$ of additive categories, additive functors and natural transformations.

In particular, all of the constructions presented in this paper and leading to the theory of bi(co)ends has been necessary in my thesis to define a Day convolution of Mackey pseudofunctor. This allows us to prove that the correct notion of a monoidal structure on such a Mackey pseudofunctor, which goes under the name of Green pseudofunctor, is precisely that of a pseudomonoid with respect to the Day convolution. A second paper about these applications will hopefully be ready soon.

\newpage
\printbibliography

\end{document}